\documentclass[fleqn]{article}
\usepackage[utf8]{inputenc}

\title{Functional limit theorems for P\'olya urns with growing initial compositions.}
\usepackage{amssymb,amsmath,amsthm}
\usepackage{gensymb}
\usepackage{cite}
\usepackage{graphicx}
\usepackage[utf8]{inputenc}
\usepackage[english]{babel}
\usepackage{geometry}
\usepackage{float}
\usepackage{xcolor}
\usepackage{nicefrac}
\usepackage{mdwlist}
\graphicspath{ {./images/} }
\usepackage{xcolor}
\usepackage{hyperref}

\numberwithin{equation}{section}
\newtheorem{theorem}{Theorem}[section]
\newtheorem{corollary}[theorem]{Corollary}
\newtheorem{lemma}[theorem]{Lemma}

\newtheorem{definition}[theorem]{Definition}
\theoremstyle{definition}
\newtheorem{remark}[theorem]{Remark}
\newtheorem{example}[theorem]{Example}
\usepackage{mathtools}
\newcommand{\R}{\right}
\newcommand{\cst}{\mathrm{cst}\hspace{0.2mm}}
\newcommand{\F}{\left}
\newcommand{\M}{\middle}

\newcommand{\Var}{\mathrm{Var}}
\newcommand{\E}{\mathbb{E}}
\newcommand{\Cov}{\mathrm{Cov}}

\newcommand{\bs}[1]{\boldsymbol{#1}}

\newcommand{\e}{\mathrm{e}}

\newcommand\eqindis{\stackrel{\mathclap{\mathrm{d}}}{=}}

\newcommand\conindis{\stackrel{\mathclap{\mathrm{d}}}{\rightarrow}}
\newcommand\coninas{\stackrel{\mathclap{\mathrm{a.s.}}}{\rightarrow}}
\newcommand\coninprob{\stackrel{\mathclap{\mathrm{p}}}{\rightarrow}}

\author{Christopher B.\ C.\ Dean\footnote{Department of Mathematical Sciences, University of Bath, UK. Email: cbcd20@bath.ac.uk.}}
\begin{document}
\maketitle
\begin{center}
    \Large \textbf{Abstract}
\end{center}
In this paper, we prove functional limit theorems for P\'olya urn processes whose number of draws and initial number of balls tend to infinity together. This 
is motivated by recent work of Borovkov \cite{Borov}, where they prove a functional limit theorem for this model when the urn has identity replacement rule.
We generalize this result to arbitrary balanced replacement rules (the total number of balls added to the urn is deterministic). Three asymptotic regimes are possible depending on how one lets the number of initial balls scale with the number of draws of the urn. In each regime, we show a first order deterministic limit and Gaussian second order fluctuations, where the behaviour of these limit processes depend on the regime, the initial composition of the urn, and the urns replacement rule. 
\\~\\
To prove our main results, we embed the process in continuous-time and use martingale theory. Although these methods are classical since the works of Athreya \& Karlin \cite{AthreyaandKarlin} and Janson \cite{Janson}, our setting with initial growing composition necessitates many new ideas. The main difference in proving limiting results for the continuous time embedding in our setting is that, when the initial composition is large compared to the number of draws, the branching process does not have time to reach equilibrium. Because of this, translating the results back to discrete-time is also much harder than in Janson \cite{Janson}. Interestingly, our continuous-time results hold under weaker assumptions on the replacement structure than classical results for multi-type branching processes; in particular, we do not need any ``irreducibility’’ assumption.
\section{Introduction}
\label{Section: Introduction}
A \(d\)-dimensional generalized P\'olya urn process \((\bs U(n))_{n\geq 0}\) is a Markov process that depends on three parameters: the initial composition \(\bs U(0) \in \mathbb{Z}^d_{\geq 0}\), the colour-weights \(a_1,...,a_d > 0\), and the (random) replacement rule \(\bs \xi_1,...,\bs \xi_d\in ~\mathbb{Z}^d\). The process represents the evolution of an urn containing balls of \(d\) colours, where, for \(1\leq i \leq d\), \(U(n)_i\) denotes the number of balls of colour \(i\) in the urn at step \(n\). The process evolves as follows: At each step \(n\) a ball is randomly drawn from the urn with probability proportional to its colour-weight. In other words, the probability of choosing a ball of colour \(i\) is~\(a_iU(n)_i/\sum_{j=1}^da_jU(n)_j\). (Many models in the literature assume the ball is chosen uniformly at random which in our setting equates to \(\bs a :=(a_1,...,a_d)\propto \bs 1\).) Then, if a ball of colour \(i\) is chosen, the ball is placed back into the urn along with a (random) set of new balls whose law is given by \(\bs \xi_i\). We assume that
\begin{equation}
    \xi_{ii}\geq -1, \text{ and } \xi_{ij}\geq 0, \text{ } i \neq j, \text{ a.s.} \label{eq: cond 111}
\end{equation}
The case of \(\xi_{ii} = -1\) represents the chosen ball being removed from the urn. Later, we consider more general urn models which allow for more than a single ball to be removed including colours other than the ball drawn (see Remark \ref{remark: tenability}).

We call the pair~\((\bs a,(\bs \xi_i)_{i=1}^d)\) the replacement structure of the urn. We call \(\sum_{j=1}^d a_j U(n)_j\) the total mass in the urn at time \(n\) and \(\sum_{j=1}^d a_j (U(n+1)_j-U(n)_{j})\) the amount of mass added to the urn from draw \(n\). Since we can remove balls from the urn, it is possible that \(\bs U(n) = 0\) for some \(n\). If this happens, we can no longer draw a ball from the empty urn, so we set \(\bs U(m) = 0\) for \(m \geq n\). In this case, we say the urn has gone extinct at time \(n\). 

First appearing in the literature in Markov \cite{Markov2006}, the properties of P\'olya urn models with varying replacement structures have been studied for over 100 years. A typical question when studying P\'olya urns (and the one we shall discuss in this paper) is how the composition of the urn behaves as the number of draws \(n\) tends to infinity. More specifically, of interest is how the proportion of each colour \(\bs U(n)/\sum_{i=1}^d U(n)_i\) (which we call the colour composition at time \(n\)) behaves as \(n\) tends to infinity, and the size, shape, and scale of the fluctuations of the urn around its asymptotic colour composition. As expected, this asymptotic behaviour depends on the replacement structure and initial composition of the urn. From the literature, one can isolate two canonical replacement structures from which the asymptotic behaviour of most urns can be inferred. \\~\\
\textbf{Identity urn:} Take \(\bs a =\bs 1\) (a ball is chosen uniformly at random), and, for \(1 \leq i \leq d\), \(S\geq 1\), let \(\bs \xi_i\) have a point mass distribution at \(S\bs e_i\), where \(\bs e_i\) is the \(i\)th canonical basis vector (when a ball is drawn, \(S\) balls of the same colour are added to the urn). 
\\~\\
\textbf{Irreducible urn:} The replacement structure is such that, for all \(1 \leq i,j \leq d\), there exists~\(n^*\geq 0\), such that given~\(\bs U(0)=\bs e_i\), there is positive probability of \(U(n^*)_j>0\). (Starting with a single ball of any colour can lead to a ball of any other colour eventually being contained in the urn.)
\\~\\
In both canonical cases, the asymptotic behaviour of the composition of the urn as the number of draws tends to infinity is known. Furthermore, the behaviour seen in the canonical cases drastically differ from each other. We give a brief overview of these results to highlight this fact. For more details, we refer the reader to: (identity case) Blackwell \& MacQueen \cite{blackwell}, and M\"uller \cite{muller}, (irreducible case) Athreya \& Karlin \cite{AthreyaandKarlin}, Janson \cite{Janson}, Pouyanne \cite{Pouyanne}, and Kesten \& Stigum \cite{Kesten}.
\\~\\
 \textbf{Identity urn:} Eggenberger \& P\'olya \cite{Polya} shows the asymptotic behaviour of the colour composition of the two colour urn. This has since been extended to a general \(d\)-colour urn (see Blackwell \& MacQueen \cite{blackwell}). They show \(\bs U(n)/\sum_{i=1}^dU(n)_i \coninas ~\bs V\) as~\(n~\rightarrow~\infty\), where \(\bs V\) is a Dirichlet vector with parameter \(\bs U(0)/S\). Furthermore, the fluctuations of the urn around~\(\bs V\) converge in distribution to a Gaussian random variable of order \(\sqrt{n}\), where the covariance matrix depends on \(\bs V\) (see M\"uller \cite{muller}).
\\~\\
\textbf{Irreducible urn:} Conditionally on non-extinction of the urn, \(\bs U(n)/\sum_{i=1}^dU(n)_i \coninas~\bs v_1\) as~\(n~\rightarrow~\infty\), where \(\bs v_1\) is a left eigenvector of the matrix \(A:=(a_j\mathbb{E}[\xi_{ji}])_{i,j=1}^d\) (see Athreya \& Karlin \cite{AthreyaandKarlin}). This result can be shown to hold for slightly weaker assumptions than irreducibility, and the size, shape, and scale of the fluctuations around \(\bs v_1\) can be categorised into three cases (see Janson \cite{Janson}). For brevity, assume the eigenvalues of \(A\) are real and simple. (This is an unnecessary restriction, but will speed up the presentation of these results whilst still highlighting their core features.) Let \(\lambda_1\) be the largest eigenvalue of \(A\), and \(\lambda_2\) the second largest. Then, conditionally on non-extiction of the urn: If \(\lambda_2<\lambda_1/2\), the fluctuations converge in distribution to a Gaussian random variable of order \(\sqrt{n}\). If \(\lambda_2 = \lambda_1/2\), the fluctuations converge in distribution to a Gaussian random variable of order \(\sqrt{n\log(n)}\). Lastly, if \(\lambda_2>\lambda_1/2\), the fluctuations converge a.s.\ to some non-Gaussian random variable of order \(n^{\lambda_2/\lambda_1}\). Furthermore, the fluctuations only depend on the initial composition in the case of \(\lambda_2>\lambda_1/2\).
\\~\\
In the present paper, we are interested in the asymptotic behaviour of urns whose number of initial balls and draws tend to infinity together. This is motivated by a recent paper Borovkov \cite{Borov}, where this behaviour is studied for the identity urn. These urns are defined as follows: Let~\((\bs U_n)_{n\geq 1}\) be a sequence of P\'olya urns each with the same replacement structure. Let \(N=N(n) := \sum_{i=1}^d U_{n}(0)_i\) denote the number of initial balls of \(\bs U_n\). We say the sequence has growing initial composition (GIC) if \(N\rightarrow \infty\) as \(n\rightarrow \infty\). We will often refer to such a sequence as a GIC urn. With this definition in mind, understanding the behaviour of an urn whose number of initial balls and draws go to infinity together can be formally defined as understanding the object~\(\bs U_n(n)\) as \(n\rightarrow \infty\).

We assume that the replacement structure of our GIC urn is balanced, this meaning that the total mass added at each time step is some fixed deterministic constant. Note, when \(\bs a =1\), the balanced property corresponds to the number of balls in the urn changing by a deterministic constant at each time step.
\begin{definition}[Balanced replacement structure]
 We say a replacement structure is balanced if there exists an \(S \in \mathbb{R}\setminus\{0\}\), such that, for each \(1 \leq i \leq d\), \(\sum_{j=1}^d a_j \xi_{ij}=S\) a.s.\ (the amount of mass added by each draw is a deterministic constant).
\end{definition}\noindent We exclude the case of \(S=0\), but note in this case the urn can be viewed as a discrete-time Markov chain on a finite state-space (there are only finitely many compositions that keep the mass of the urn constant). We give details as to why we require the balanced assumption in Section \ref{Sec: methods}. For now, note that results under the balanced assumption are still of significant interest with the assumption commonly appearing in the literature (see Pouyanne \cite{Pouyanne}, Pouyanne \& Janson \cite{PouyandJan}, M\"uller \cite{muller}). 

We also assume that the initial colour composition~\( N^{-1} \bs U_n(0)= \bs \mu\) for all \(n\geq 1\), for some vector \(\bs \mu\). This is necessary since the asymptotic behaviour of the GIC urn depends on its initial colour composition, so convergence without this assumption is not possible. This dependence is analogous to the asymptotic behaviour of the fixed initial composition~(FIC) urn (i.e.\ a single urn \(\bs U\) with initial composition \(\bs U(0)\)) depending on \(\bs U(0)\). Note, this assumption can be relaxed to the initial colour composition converging to \(\bs \mu\) (see Remark \ref{remark: converge to mu}). 
\subsection{Our Contribution}
\label{Sec: Our contribution}
 Our main contributions are Theorems \ref{theorem: main result simplified}, \ref{theorem: Main results discrete} and \ref{theorem: main results discrete 2}. These combined are a functional limit result for balanced GIC urns. They state the behaviour of the asymptotic colour composition of the urn and its fluctuations. 

As Borovkov \cite{Borov} shows for the identity urn, three asymptotic regimes are possible depending on how one lets the number of initial balls \(N\) of \(\bs U_n\) scale with the number of draws \(n\). The following is a brief overview of the results our theorems give on the asymptotic behaviour of \(\bs U_n(n)\) as \(n\rightarrow \infty\):
\\~\\
\textbf{Initial ball dominant \(n=o(N)\):} The colour composition converges to \(\bs \mu\), and the fluctuations are a Brownian motion of order \(\sqrt{n}\). The covariance matrix of the Brownian motion depends on \(\bs \mu\) and the replacement structure.
\\~\\
\textbf{Transitional regime \(n\sim o(N)\):} The colour composition converges to a deterministic process, and the fluctuations are a Gaussian process of order \(\sqrt{n}\). Both the deterministic process and the covariance function of the Gaussian process are dependent on \(\bs \mu\) and the replacement structure. For certain replacement structures, these processes have nice forms. For example, for the identity urn, Borovkov \cite{Borov} shows the asymptotic colour composition to be \(\bs \mu\) and the Gaussian process to be the sum of an independent Brownian motion and Gaussian random variable. Another example where these processes have nice forms is given in Example \ref{example: matching problem}.
\\~\\
\textbf{Time step dominant \(N=o(n)\):} The asymptotic behaviour seen depends on whether the Perron-Frobenius eigenvalue~\(\lambda_1\) (see Corollary \ref{cor: PF}) of \(A=(a_j\mathbb{E}[\xi_{ji}])_{i,j=1}^d\) is simple. This is analogous to the differing behaviour seen in the FIC case for the identity and irreducible urn. The identity urn is an example of \(\lambda_1\) non-simple, and the irreducible urn an example of~\(\lambda_1\) simple.
\\~\\
\textit{Case of \(\lambda_1\) non-simple:} The colour composition converges to a deterministic vector which depends on \(\bs \mu\). The fluctuations are a Gaussian random variable of order \(N^{-1/2}n\) which also depend on \(\bs \mu\).
\\~\\
\textit{Case of \(\lambda_1\) simple:} The colour composition converges to the right-eigenvector of \(\lambda_1\). The fluctuations are a Gaussian process whose size and covariance function depend on the spectral gap of \(A\). Assume the eigenvalues of \(A\) are real and simple. (As in the FIC case, this is unnecessary and we do not assume this in our theorems, but this assumption allows for a nicer overview of the results.) Let \(\lambda_2\) be the second largest eigenvalue of \(A\): If \(\lambda_2<\lambda_1/2\), the Gaussian process is of order \(\sqrt{n}\) and is an Ornstein-Uhlenbeck process. If \(\lambda_2 = \lambda_1/2\) a general Gaussian process is seen with size \(\sqrt{n\log(n/N)}\). Lastly, if \(\lambda_2 > \lambda_1/2\), the fluctuations are a Gaussian random variable of order \(N^{1/2}(n/N)^{\lambda_2/\lambda_1}\). Furthermore, the fluctuations only depend on \(\bs \mu\) when \(\lambda_2 > \lambda_1/2\).
\subsection{Our Methods}
\label{Sec: methods} 
We embed the P\'olya urn into continuous-time through a multi-type continuous-time Markov branching process (MCBP). This embedding was introduced by Athreya \& Karlin \cite{AthreyaandKarlin}. An MCBP~\((\bs X(t))_{t\geq 0}\) is a Markov process that depends on three parameters: the initial condition \(\bs X(0) \in \mathbb{Z}^d_{\geq 0}\), the particle lifespan parameters~\(a_1,...,a_d >0\), and the (random) particle replacement rule \(\bs \xi_1,...,\bs \xi_d \in \mathbb{Z}^d\), where we assume \eqref{eq: cond 111} holds. The process represents the evolution of a set of particles of \(d\) types, where, for \(1\leq i \leq d\), \(X(t)_i\) represents the number of particles of type \(i\) in the system at time \(t\). Each particle of type \(i\) lives for a random \(a_i\)-exponentially-distributed amount of time, and at time of death is replaced by a (random) set of new particles whose law is given by \(\bs \xi_i+\bs e_i\). (We add an additional particle of the type that died to keep our definition of the MCBP and P\'olya urn consistent with each other, since the particle dies in the MCBP, whereas the ball drawn is returned to the P\'olya urn.) We assume that a particle's lifespan and offspring are independent of each other and of all particles in the system. Let \(\bs U\) be a P\'olya urn with initial composition \(\bs X(0)\) and replacement structure \((\bs a, (\bs \xi_i)_{i=1}^d)\). Let \((\tau_{i})_{i\geq 1}\) be the ordered death times of particles in \(\bs X\). Then, by the memoryless property of the exponential distribution, one can show (see Athreya \& Karlin \cite{AthreyaandKarlin})
\begin{equation}
    \label{eq: branchurn}
    (\bs X(\tau_i))_{i\geq 1} \eqindis (\bs U(i))_{i\geq  1}.
\end{equation}
This relation allows one to prove results about the urn by proving results about \(\bs X\) and the random time change \((\tau_i)_{i\geq 1}\). For the GIC urn, we instead have a sequence of MCBPs defined \((\bs X_n)_{n\geq 1}\), where \eqref{eq: branchurn} holds for each \(n\). Let \((\tau_{n,i})_{i\geq 1}\) be the ordered death times of particles in \(\bs X_n\). Our goal is to show a functional limit theorem about the \((\bs X_n)_{n\geq 1}\) at a timescale that contains the \((\tau_{n,n})_{n\geq 1}\) with high probability (probability tending to 1 as \(n\rightarrow \infty\)). Then, we recover a result on the urn by using known results on random time changes of random processes along with showing good control of the \((\tau_{n,n})_{n\geq 1}\). The order of magnitude of \(\tau_{n,n}\) as \(n\rightarrow \infty\) depends on how \(N\) scales with \(n\). Thus, to achieve our goal, we study the branching process on various timescales corresponding to these orders of magnitude. 
\begin{remark}
\label{remark: bp timescales}
One can show in probability as \(n\rightarrow \infty\):
\begin{enumerate}
    \item If \(n=o(N)\), then \(\tau_{n,n}\) is of order  \(n/N\).
    \item   If \(n\sim N\), then \(\tau_{n,n}\) is of order \(1\).
    \item If \(N=o(n)\), then \(\tau_{n,n}\) is of order  \(\log(n/N)\).
\end{enumerate}
\end{remark}  
\noindent This alternate setting adds additional difficulties which necessitate a change in approach from the FIC case studied in Janson \cite{Janson}. Firstly, to deal with the growing number of initial particles of the MCBP, we use a natural property of MCBPs often called the branching property: For any three MCBPs~\(\bs X, \bs Y, \bs Z\) with identical replacement structures and initial conditions \(\bs X(0), \bs Y(0), \bs X(0)+\bs Y(0)\) respectively, we have
\begin{equation}
\label{equ:branching property}
   ( \bs X(t)+\bs Y(t))_{t\geq 0} \eqindis (\bs Z(t))_{t\geq 0}.
\end{equation}
This property is a simple consequence of the fact that particles in the MCBP behave independently of each other. The branching property implies our MCBP with \(N\) initial particles can be represented as the sum of \(N\) independent MCBPs each with 1 initial particle. Our approach is to apply a functional CLT to this sum, where the timescale used for this functional CLT depends on the regime being studied as explained in Remark \ref{remark: bp timescales}. The proof of this functional CLT is non-trivial and we introduce several new tools to handle it. One such tool is new sharp fourth moment bounds for the MCBP. Although not needed for us, our method to prove these fourth moment bounds could be extended to show higher moment bounds. These moment bounds are shown using martingale methods for the MCBP which build off preliminary results of Janson \cite{Janson}. Also, we show that standard CLT assumptions can be extended to allow for an independent random number of summands which is necessary for proving the functional CLT. Interestingly, since the behaviour of the functional CLT only depends on asymptotic variances of the MCBP, which are tractable for most replacement structures, our results hold under very weak assumptions (only a fourth moment bound on the replacement structure is necessary).

The second difficulty is obtaining good control of \(\tau_{n,n}\) as \(n\rightarrow \infty\). Unlike the FIC case, a first order result about \(\tau_{n,n}\) is insufficient to transfer the MCBP results to the urn. Instead, we need to understand the behaviour of \(N^{1/2}\tau_{n,n}\) as \(n\rightarrow \infty\). This heuristically follows because the GIC MCBP satisfies a CLT with number of summands \(N\). When we unembed to the urn setting, the deterministic term of this CLT becomes a random term involving the \(\tau_{n,n}\). The value \(N^{1/2}\) appears since, in \(N\), the deterministic term of the CLT grows \(N^{1/2}\) faster than the random fluctuations.

Not only do we need better control of the \(\tau_{n,n}\), but this control is harder to obtain. Take the \(n\sim N\) regime where these difficulties are most prevalent. Classically, one uses the equilibrium behaviour of the MCBP to obtain a first order result on the \(\tau_{n,n}\) (see Athreya \& Ney Chapter V Section 7.6 \cite{Athreya}). However, in this regime, by Remark \ref{remark: bp timescales}, the MCBP has not run for long enough to reach ``long term" equilibrium, but has had enough time to escape any equilibrium behaviour seen by staying close to its initial composition. Thus, we cannot apply classical methods in this regime. Even in the other regimes where the MCBP displays stationary-like behaviour, classical methods only give us first order convergence of \(\tau_{n,n}\), which is not sufficient for our setting. To obtain the necessary lower order fluctuations of \(\tau_{n,n}\), we assume our replacement structure satisfies the balanced property. Under the balanced property, the \((\tau_{n,n})_{n\geq 1}\) no longer depend on the composition of the MCBP. Indeed, the \((\tau_{n,n})_{n\geq 1}\) only depend on the MCBP through its total mass, which is a deterministic function of the number of particle deaths when the replacement structure is balanced. Thus, we are able to ignore the fact that the MCBP is not in equilibrium. It is unclear if the balanced assumption is an artifact of our method, or if different behaviour does appear for non-balanced replacement structures due to this non-equilibrium behaviour. In Remark \ref{remark: nonbal heuristic}, we give a heuristic which suggests that the non-balanced case displays different behaviour to the balanced case.

Before moving on, we state a property of MCBPs that we will often use. The strong Markov property for MCBPs: For any MCBP \(\bs X\) and stopping time \(\tau\), let
\begin{equation}
    \label{eq: strong markov prop}
    (\bs Y(t))_{t\geq 0}:= (\bs X(\tau +t))_{t\geq 0}.
\end{equation}
Then \(\bs Y\) is a MCBP with initial condition \(\bs X(\tau)\), replacement structure identical to \(\bs X\), and only depends on \((\bs X(t))_{t\in [0,\tau]}\) through its initial condition.
\\~\\
\textbf{Outline of the paper:} In Section \ref{Sec: Prelim}, we state the assumptions and preliminaries needed for our main results, which we give in Section \ref{Sec: main results}. In Section \ref{Sec: Heuristics}, we give detailed heuristics for the results in Section \ref{Sec: main results}. Lastly, Sections \ref{Sec: Prelim for proofs} - \ref{Sec: proof of discrete time} contain the proofs of the results in Section \ref{Sec: main results}.
\section{Preliminaries}
\label{Sec: Prelim}
Throughout this work, \(\cst\) will represent an arbitrary constant that may change from line to line. The constant will be independent of any variables of interest unless stated otherwise. In this case, we may write \(\cst_{a}\) to mean a dependency on some variable \(a\).
\subsection{Assumptions}
\label{Sec: assumptions}
\begin{theorem}[Perron-Frobenius theorem for non-negative matrices]
\label{theorem: PF}
Let \(B\) be a \(d\times d\) matrix with real non-negative entries. There exists a non-negative real eigenvalue \(\lambda_1\) of \(B\), such that, for any other eigenvalue \(\lambda\) of \(B\), \(|\lambda|\leq \lambda_1\). We call~\(\lambda_1\) the Perron-Frobenius eigenvalue of \(B\). This eigenvalue satisfies
\begin{equation*}
    \lambda_1 \leq \max_{1\leq j\leq d}\sum_{i=1}^d B_{ij}.
\end{equation*}
Furthermore, there exists a left and right eigenvector \(\bs u_1\) and \(\bs v_1\) of \(\lambda_1\) that have non-negative entries. If \(B\) is also irreducible, then \(\lambda_1\) is simple and positive with \(\bs u_1\) and \(\bs v_1\) also positive.
\end{theorem}
\begin{corollary} 
\label{cor: PF}
Let \(B\) be a \(d\times d\) matrix with real diagonal entries and real non-negative entries elsewhere. Then, the eigenvalue \(\lambda_1\) of \(B\) with largest real part is real. We call this the Perron-Frobenius eigenvalue of \(B\). This eigenvalue satisfies
\begin{equation*}
        \lambda_1 \leq \max_{1\leq j\leq d}\sum_{i=1}^d B_{ij}.
\end{equation*}
\end{corollary}
\begin{proof}
Since only the diagonal entries of \(B\) can be negative, there exists a \(\gamma\geq 0\) such that \(B+\gamma I\) is a non-negative matrix. We can apply Theorem \ref{theorem: PF} to obtain the Perron-Frobenius eigenvalue \(\lambda_1+\gamma\) of \(B+\gamma\). This implies \(\lambda_1\) is an eigenvalue of~\(B\) and satisfies the properties of the corollary. 
\end{proof}

\noindent Recall the matrix \(A:=(a_j\mathbb{E}[\xi_{ji}])_{i,j=1}^d\) from the introduction. Since only the diagonal entries of \(A\) can be negative, there exists a real Perron-Frobenius eigenvalue \(\lambda_1\) by Corollary \ref{cor: PF}.

Our assumptions on the replacement structure of the urn are:
\begin{itemize}
    \item[](A1) For all \(i,j \in \{1,...,d\}\), \(\mathbb{E}[\xi_{ij}^4] < \infty\).
    \item[](A2) \(\xi_{ii}\geq -1, \text{ and } \xi_{ij}\geq 0\), \(i \neq j,\) holds a.s.
    \item[](A3) The Perron-Frobenius eigenvalue \(\lambda_1\) of \(A\) is positive.
\end{itemize}
In the \(n=o(N)\) and \(n\sim N\) regimes we assume (A1)-(A2). In the \(N=o(n)\) regime we assume (A1)-(A3). 

We use (A1) to show CLT conditions are satisfied during the proof of the functional limit theorem for the MCBP. This assumption is likely an artifact of our method, and our results likely still hold if one only assumes there exists some \(\varepsilon>0\) such that the \(2+\varepsilon\) moment is bounded. Assumption (A2) is necessary for the particles in the MCBP to be independent of each other (otherwise one particle death could cause the death of another). It is possible to relax this assumption, see Remark \ref{remark: tenability}.
 Assumption (A3) is only needed in the \(N=o(n)\) regime. If (A3) does not hold, it is well known (see Athreya \& Ney Chapter V Section 7.6 \cite{Athreya}) that the FIC urn goes extinct a.s.\ (If \(\lambda_1\leq 0\) the associated MCBP is (sub)critical, thus goes extinct a.s.) In the \(N=o(n)\) regime for the GIC urn, one can show this equates to (at least) extinction of the urn in probability. Therefore, to see a non-zero limit, (A3) is necessary. In the other two regimes, there are enough initial balls compared to the number of draws to see non-extinction in the limit without (A3).  
\subsection{Jordan spaces of the matrix $A$}
\label{Sec: Jordan spaces of A}
In the \(N=o(n)\) regime, projecting the urn into different Jordan spaces of \(A\) leads to different asymptotic behaviour. This is why the behaviour of the fluctuations of the urn are dependent on the spectral gap of \(A\). Therefore, to capture the full behaviour of the urn, one needs to study its projection on each of the Jordan spaces of \(A\). In this section, we introduce the necessary notation for working with the Jordan spaces of \(A\). See Janson \cite{Janson} for further details.

There exists a basis \((\bs v_1,...,\bs v_d)\) 
in which the operator $\bs x\mapsto A \bs x$ 
on $\mathbb C^d$ is represented by a Jordan matrix; 
we call this basis the Jordan basis of $A$.
This Jordan matrix is formed of diagonal blocks;
each of these blocks is of the form
\[J = \begin{pmatrix}
\lambda & 1 & 0 & \ldots  & 0\\
0 & \lambda & 1 & \ddots &  \vdots \\
\vdots & \ddots & \ddots & \ddots  & 0 \\
\vdots &  & \ddots & \ddots  & 1 \\
0 & \ldots & \ldots   & 0 & \lambda
\end{pmatrix}
\]
where $\lambda$ is an eigenvalue of $A$. Let \(\mathcal{J}_{\lambda}\) be the set of all Jordan blocks with diagonal coefficients \(\lambda\). We call \(\lambda\) the eigenvalue of the Jordan blocks in \(\mathcal{J}_{\lambda}\).
In the following, we fix such a Jordan block $J\in \mathcal{J}_{\lambda}$;
we let
$m$ denote its size, and let \((\bs v_{J,1},\bs v_{J,2},...,\bs v_{J,m})\) be
the subset of \((\bs v_1,...,\bs v_d)\) corresponding to this block.
In particular,
\begin{equation}\label{eq:Jordan}
A \bs v_{J,1} = \lambda \bs v_{J,1}\quad \text{ and }\quad
A \bs v_{J,i}=\lambda \bs v_{J,i}+\bs v_{J,i-1} \quad (\forall 2\leq i\leq m).
\end{equation}
In the following, we let $J$ denote both the matrix block and the subspace of $\mathbb C^d$ generated by \((\bs v_{J,1},\bs v_{J,2},...,\bs v_{J,m})\), which we call the Jordan space $J$. Since \((\bs v_1,..., \bs v_d)\) is a basis, 
for any vector \(\bs x \in \mathbb{C}^d\), there exists a unique vector~$(\alpha_1(\bs x), \ldots, \alpha_d(\bs x))\in~\mathbb C^d$, such that 
\begin{equation}\label{eq:def_alphas}
\bs x = \sum_{i=1}^d \alpha_i(\bs x) \bs v_i.
\end{equation}
For all $1\leq i \leq d$ and all $\bs x\in\mathbb C^d$,
we call $\alpha_i(\bs x)$ the coefficient of $\bs v_i$ in the Jordan basis representation of $\bs x$. 

We define the projection on the Jordan space \(J\) as follows
\begin{equation}
\label{eq: coef func}
P_J\bs x = \alpha_{J,1}(\bs x)\bs v_{J,1}+\alpha_{J,2}(\bs x)\bs v_{J,2}+...+\alpha_{J,m}(\bs x)\bs v_{J,m}.
\end{equation}
By \eqref{eq:def_alphas},
\begin{equation}
    \sum_{\lambda \in \Lambda}\sum_{J \in \mathcal{J}_{\lambda}} P_{J}=I, \label{eq:identity}
\end{equation}
where \(I\) is the identity map. By \eqref{eq:Jordan}, we have
\begin{equation}
\label{eq: projection}
    AP_J = P_JA = \lambda P_J + N_{A}P_J,
\end{equation}
where $N_A$ is defined on any Jordan space $J$ by
\begin{equation}
\label{eq:nilpotent operator}
N_A\bs v_{J,1} = 0
\quad\text{ and }\quad 
N_A\bs v_{J,i}=\bs v_{J,i-1}\quad (\forall 2\leq i\leq m).
\end{equation}
By definition, $N_A$ is a nilpotent operator on~$J$; its nilpotency is equal to $m$ (the dimension of the Jordan space $J$). (Recall that the nilpotency of an operator $N$ is the smallest integer \(k\geq 0\) such that \(N^k=0\).) 

Since both \(P_J\) and \(N_A\) are linear maps, they can be represented as matrices in \(\mathbb{C}^{d\times d}\) for any given basis. From now on, we take \(P_J\) and \(N_A\) to also denote their matrix representation with respect to the canonical basis. 

When \(J\) is of size \(1\), the only basis vector of \(J\), \(\bs v_{J,1} \), is a right eigenvector of \(A\) by \eqref{eq:Jordan}. In this case, there is a corresponding left eigenvector of \(A\) denoted \(\bs u_{J,1}\), such that
\begin{equation}
    P_J = \frac{\bs v_{J,1} \bs u_{J,1}'}{\bs v_{J,1}\cdot \bs u_{J,1}}. \label{eq: eigenvector proj}
\end{equation}
Since \(A\) is real, by taking the complex conjugate in \eqref{eq:Jordan}, we see there is a Jordan space of \(A\) generated by \((\overline{\bs v}_{J,1},...,\overline{\bs v}_{J,m})\) with eigenvalue \(\overline{\lambda}\). We call this the conjugate Jordan space to \(J\) and use \(\overline{J}\) to denote it. For any Jordan basis vector \(\bs v_i\), we see
\begin{align}
\label{eq: conjugate block}
    \overline{P}_{J}\bs v_i=\overline{P_{J} \overline{\bs v}_i} = \bs 1_{\overline{\bs v}_i \in \{\bs v_{J,1},\bs v_{J,2},...,\bs v_{J,m}\}}\bs v_i= P_{\overline{J}}\bs v_i,
\end{align}
therefore we have the relation \(\overline{P}_J=P_{\overline{J}}\). We also see, for any Jordan space \(J\), \(1\leq i \leq m\),
\begin{equation*}
    \overline{N}_A \bs v_{J,i} = \overline{N_A \overline{\bs v}_{J,i}} = \bs 1_{i\geq 2}\bs v_{J,i-1}=N_A \bs v_{J,i},
\end{equation*}
so \(N_A\) is a real matrix.
 
By \eqref{eq: projection}, for \(1 \leq \kappa \leq m\), and \(t \in \mathbb{R}\),
\begin{equation}
   N_A^{m-\kappa} P_J \mathrm{e}^{At}= N_A^{m-\kappa} \mathrm{e}^{(\lambda +N_A)t}P_J = \mathrm{e}^{\lambda t}N_A^{m-\kappa} \sum_{i=0}^{m-1}\frac{(tN_A)^{i}}{i!}P_J=\mathrm{e}^{\lambda t}\sum_{i=0}^{\kappa-1}\frac{t^{i}N_A^{i+m-\kappa}}{i!}P_J. \label{eq: mat exp bound start}
\end{equation}
This implies, for all \(\bs x \in \mathbb{C}^d\), and \(t \in \mathbb{R}\),
\begin{equation}
\label{eq: matrix exp bounds}
\|  N_A^{m-\kappa} P_J\mathrm{e}^{At}\bs x\|_2 \leq \cst(1+|t|)^{\kappa-1}\mathrm{e}^{\mathrm{Re}\lambda t} \|\bs x\|_2,
\end{equation}
 where
\begin{equation*}
\|\bs x\|_2 = \F(\sum_{i=1}^d |x_i|^2\R)^{1/2}
\end{equation*}
is the euclidean norm on \(\mathbb{C}^d\). Furthermore, \eqref{eq:identity} and \eqref{eq: matrix exp bounds} imply
\begin{equation}
\|\mathrm{e}^{At}\bs x\|_2 \leq \sum_{\lambda \in \Lambda}\sum_{P_{J}\in \mathcal{J}_{\lambda}}\|P_{J}\e^{At}\bs x\|_2\leq \cst (1+|t|)^{m_1-1} \mathrm{e}^{\lambda_1 t}\|\bs x\|_2,
\label{eq: matrix exp bounds all}
\end{equation}
where we have used \(\lambda_1\) is the Perron-Frobenius eigenvalue of \(A\), and we let \(m_1\) denote the size of the largest Jordan block across all Jordan blocks with eigenvalue \(\lambda_1\).

\begin{lemma}
\label{lemma: balanced frob}
For a balanced replacement structure, the Perron-Frobenius eigenvalue \(\lambda_1=S\) with left eigenvector \(\bs u_1 = \bs a\).
\end{lemma}
\begin{proof}
By the balanced property, it is clear \(S\) is an eigenvalue of \(A\) with left eigenvector \(\bs a\). Therefore, it is left to check this is the Perron-Frobenius eigenvalue. Let \(B:= (A_{ij}a_i/a_j)_{i,j=1}^d=(\mathbb{E}[ \xi_{ji}]a_i)_{i,j=1}^d\). By the balanced property, the column sums of \(B\) are equal to \(S\), so \(S\) must be an eigenvalue of \(B\) with left eigenvector \(\bs 1\). By definition of \(A\), only the diagonal entries of \(B\) can be negative. Therefore, by Corollay \ref{cor: PF}, the largest real eigenvalue of \(B\) is bounded by the max of the column sums of \(B\) which is \(S\). This implies \(S\) is the Perron-Frobenius eigenvalue of \(B\). Let \(\lambda\) be an eigenvalue of \(A\) with left eigenvector \(\bs v\), and let \(\bs w:=(v_1 a_1^{-1}, v_2 a_2^{-1},...,v_d a_d^{-1})\). We see, for \(1 \leq j \leq d\),
\begin{equation*}
    (\bs w' B)_j = a_j^{-1}\sum_{i=1}^d w_i a_i A_{ij} = a_j^{-1} \sum_{i=1}^d v_i A_{ij}= \lambda a_j^{-1} v_j=\lambda w_j.
\end{equation*}
Thus, \(\lambda\) is also an eigenvalue of \(B\). This, and the fact that \(S\) is the Perron-Frobenius eigenvalue of \(B\) imply that the Perron-Frobenius eigenvalue of \(A\) is bounded by \(S\), and therefore must be \(S\).
\end{proof}

\begin{definition}
 We call an eigenvalue \(\lambda\) of \(A\) (and any Jordan block with eigenvalue \(\lambda\)) small if \(\mathrm{Re} \lambda < \lambda_1/2\), critical if~\(\mathrm{Re} \lambda = \lambda_1/2\), and large if \(\mathrm{Re} \lambda > \lambda_1/2\). Let \(\Lambda\) be the set of eigenvalues of \(A\), and set
\begin{align*}
    &\Lambda_s = \{\lambda \in \Lambda : \mathrm{Re}\lambda< \lambda_1/2\}, \quad
    \Lambda_c = \{\lambda \in \Lambda : \mathrm{Re}\lambda= \lambda_1/2\}, \quad 
    \Lambda_{\ell} = \{\lambda \in \Lambda : \mathrm{Re}\lambda> \lambda_1/2\}.
\end{align*}
\end{definition}
\noindent 
As seen in the introduction, the asymptotic behaviour of the GIC urn in the \(N=o(n)\) regime depends on whether the eigenvalue in \(\Lambda\setminus\{\lambda_1\}\) with largest real part is small, critical, or large.
\begin{example}[Friedman's Urn]
\label{example: friedmans urn}
A balanced irreducible urn that has seen much interest is Friedman's urn \cite{Friedman}. This is a 2-colour urn and the replacement structure is as follows: Balls are drawn uniformly at random with \(\alpha\) balls of the colour drawn being added into the urn along with \(\gamma\) balls of the other colour. Therefore, we have 
\begin{equation*}
  A =
  \left( {\begin{array}{cc}
    \alpha & \gamma \\
    \gamma & \alpha \\
  \end{array} } \right).
\end{equation*}
To satisfy (A1) and (A3), we must take \(\alpha \geq -1\), and \(\alpha+\gamma >0\). To exclude the identity urn, we also take \(\gamma >0\). By Lemma~\ref{lemma: balanced frob}, we have \(\lambda_1 = \alpha +\gamma\) and \(\bs u_1 = (1,1)\). One can further check \(\bs v_1 = \frac{1}{2}(1,1)\), and the other eigenvalue and left and right eigenvectors of \(A\) are \(\lambda_2 = \alpha-\gamma\), \(\bs u_2 = (1,-1)\) and \(\bs v_2 = \frac{1}{2}(1,-1)\). Since both of the eigenvalues of \(A\) are simple, \eqref{eq: eigenvector proj} implies the projection matrices satisfy
\begin{equation*}
    P_{J_1} = \bs v_1 \bs u_1' = \frac{1}{2}\left( {\begin{array}{cc}
    1 & 1 \\
    1 & 1 \\
  \end{array} } \right),\quad  P_{J_2} = \bs v_2 \bs u_2' = \frac{1}{2}\left( {\begin{array}{cc}
    1 & -1 \\
    -1 & 1 \\
  \end{array} } \right),
\end{equation*}
where \(J_1\) has eigenvalue \(\lambda_1\), and \(J_2\) has eigenvalue \(\lambda_2\). Moreover, \(\lambda_2\) is small/critical/large depending on whether \(\alpha\) is less than/equal to/greater than \(3\gamma\). Once we have stated our results, we return to this example to show how our results describe the asymptotic behaviour of Friedman's urn with number of initial balls tending to infinity.
\end{example}
\noindent In the final part of this section, we spend some time discussing an object which appears in many of our main results in the~\(N=o(n)\) regime. Recall \(m_1\), the size of the largest Jordan block across all Jordan blocks which have eigenvalue~\(\lambda_1\). Let \(P_{\lambda_1}~=~\sum_{J\in \mathcal{J}_{\lambda_1}}P_J\). The object is
\begin{equation}
    \label{eq: imp ob 1}
   \bs v(\bs \mu) :=  \frac{N_A^{m_1-1}P_{\lambda_1}\bs \mu}{(m_1-1)!}.
\end{equation}
This is the dominating term of \(\e^{At}\bs\mu\) as \(t\rightarrow \infty\) (See \eqref{eq: mat exp bound start}), which appears due to its heavy involvement in the moments of the MCBP (see Section \ref{Sec: Prelim for proofs}). We discuss this object now due to its prevalence in our results, and its simpler
forms when the replacement structure is balanced/irreducible/the identity, which will often be the case in application. We start with the balanced case.
\begin{lemma}
\label{lemma: import ob}
For a balanced replacement structure, the Jordan blocks of the matrix \(A\) with eigenvalue \(S\) are all of size 1.
\end{lemma}
\noindent 
The proof of this Lemma requires results presented later in this work and is given in Appendix \ref{Appendix: import obj}. Suppose \(A\) has \(L\) Jordan blocks with eigenvalue \(S\), then Lemma \ref{lemma: import ob} and \eqref{eq: eigenvector proj} imply
\begin{align}
     &\bs v(\bs \mu) = \sum_{r=1}^L \beta_r \bs v_r, \quad \beta_r = \bs u_r \cdot \bs \mu, \quad \bs u_r \cdot \bs v_r =1, \quad \text{ (Balanced) } \label{eq: import ob 3}
\end{align}
where the \((\bs u_r,\bs v_r)\) are left and right eigenvector pairs of \(A\) with eigenvalue \(S\). Moreover, by Lemma \ref{lemma: balanced frob}, w.l.o.g.\ we can take \(\bs u_1 = \bs a\), \(\beta_1 = \bs a \cdot \bs \mu>0\) (since the colour weights are positive). Importantly, this implies \(\bs v(\bs \mu)\neq 0\) for any choice of~\(\bs \mu\), which implies many of the limits seen in our results in Section \ref{Sec: main results} are non-zero for all choices of \(\bs \mu\). In the identity and irreducible cases, \eqref{eq: imp ob 1} has an even nicer form. In the irreducible case, \(\lambda_1\) is simple by Theorem \ref{theorem: PF}. Therefore, 
\begin{equation}
\label{eq: import ob 4}
    \bs v(\bs \mu) = \beta_1 \bs v_1, \quad \beta_1 = \bs u_1 \cdot \bs \mu, \quad \bs u_1 \cdot \bs v_1 = 1, \quad \text{ (Irreducible) }
\end{equation}
where \((\bs u_1, \bs v_1)\) is the positive left and right eigenvector pair of \(\lambda_1\). If one further assumes a balanced irreducible replacement structure, we have that \(\bs u_1=a\) and \(\lambda_1=S\) by Lemma \ref{lemma: balanced frob}. In the identity case, \(A\) is a diagonal matrix with a single eigenvalue \(S\), therefore \(P_{\lambda_1}=I\) and \(m_1=1\). This implies
\begin{equation}
\label{eq: import ob 5}
   \bs v(\bs \mu)= \bs \mu. \quad \text{ (Identity)}
\end{equation}
\section{Main Results}
\label{Sec: main results}
Let \(D\) denote the space of  right-continuous functions with left-hand limits (c\`adl\`ag). Unless stated otherwise, we assume these functions take values in \(\mathbb{C}^d\) equipped with the euclidean norm, and that \(D\) is equipped with the Skorohod \(J_1\)-topology. Recall the distance between two functions \(f,g\) in \(D[a,b]\) is 
\begin{equation*}
    d(f,g) = \inf_{\upsilon \in \Upsilon}\F\{\|\upsilon\|+\sup_{t \in [a,b]}\|f(t)-g(\upsilon(t))\|_2\R\},
\end{equation*}
where \(\Upsilon\) is the set of strictly increasing continuous maps from \([a,b]\) to itself with
\begin{equation*}
    \|\upsilon\| = \sup_{s\neq t}\F|\log \F(\frac{\upsilon(t)-\upsilon(s)}{t-s}\R)\R|.
\end{equation*}A sequence of processes \((\bs X_n)_{n\geq 1}\) converges to \(\bs X\) in \(D(-\infty,\infty)\), if, for all \(T \geq 0\), \((\bs X_n)_{n\geq 1}\) converges to \(\bs X\) in \(D[-T,T]\). For \(T \geq 0\), \((\bs X_n)_{n\geq 1}\) converges to \(\bs X\) in \(D(0,T]\), if, for all \(0<\varepsilon<T\), \((\bs X_n)_{n\geq 1}\) converges to \(\bs X\) in \(D[\varepsilon,T]\). We always use the variable \(t\) to index time in our functional limit results even if not stated. Note, for results in \(D[0,\infty)\), the converging sequence \((\bs X_n)_{n\geq 1}\) may not be well defined at all \(t\in [0,\infty)\) for all \(n\geq 1\), but for each \(t\in [0,\infty)\), \(\bs X_n(t)\) will be well defined for all \(n\) large enough. 
\\~\\
In Section \ref{Sec: main main result}, we give our main result on balanced GIC urns, Theorem \ref{theorem: main result simplified}. This is a functional limit theorem containing the behaviour of the asymptotic colour composition in each regime, and the size, shape, and scale of the highest order random fluctuations of the urn around its asymptotic colour composition. However, we omit the covariance functions of the limit processes until later, since they are derived in terms of the GIC MCBP limits. In Section \ref{Sec: results for the GIC MCBP}, we give a functional limit theorem for the asymptotic behaviour of the GIC MCBP. Lastly, in Section \ref{Sec: results for the polya urn}, we give some more results on the GIC urn: We give the covariance functions of the limits in Theorem \ref{theorem: main result simplified} in terms of processes defined in Section \ref{Sec: results for the GIC MCBP}. In the \(N=o(n)\) regime, we give a functional limit theorem for the balanced GIC urn when projected on to the Jordan spaces of \(A\), Theorem \ref{theorem: Main results discrete}. This is a more general result than the \(N=o(n)\) results presented in Theorem \ref{theorem: main result simplified}. In particular, Theorem \ref{theorem: Main results discrete} gives us access to the lower order random fluctuations seen in Jordan spaces that do not contribute to the highest order random fluctuations of the urn. Finally, we finish Section \ref{Sec: results for the polya urn} by giving some applications of our results.
\subsection{Our main functional limit theorem for balanced GIC P\'olya urns} 
\label{Sec: main main result}
\begin{theorem}
\label{theorem: main result simplified}
Let \(N=N(n)\) be a positive integer-valued sequence, such that \(N \rightarrow \infty\) as \(n \rightarrow \infty\). Let \(\bs \mu\) be a non-negative vector, such that \(\sum_{i=1}^d \mu_i=1\). Let \((\bs U_n)_{n \geq 1}\) be a sequence of balanced P\'olya urns with initial condition \(N\bs \mu\), and balanced replacement structure \((\bs a, (\bs \xi_i)_{i=1}^d)\), such that the amount of mass added per time step is \(S\in \mathbb{R}\setminus\{0\}\). Let \(\beta_1 = \bs a \cdot \bs \mu\),  \(\ell_1(n,t):=\log(1+Snt/\beta_1 N)\), and \(\ell_2(n,t) =\log(1+S(n/N)^t/\beta_1 )\).
\\~\\
\textbf{(IBD) Initial Ball Dominant:} Assume that \(n=o(N)\), and (A1)-(A2) hold. Then, as \(n \rightarrow \infty\),
    \begin{align}
  & n^{-1/2}(\bs U_n({\lfloor nt\rfloor})-N\e^{AS^{-1}\ell_1(n,t)}\bs \mu)\conindis \bs Y_1(t)\text{ in } D[0,\infty), \label{eq: colour comp IBD 2}\\
   &N^{-1}\bs U_n({\lfloor nt\rfloor})\coninprob \bs \mu\text{ in } D[0,\infty), \label{eq: colour comp IBD}
\end{align}
where \(\bs Y_1\) is a Brownian motion with correlated components.
\\~\\
\textbf{(TR) Transitional Regime:} Assume that \(n/N\rightarrow 1 \) as \(n \rightarrow \infty\), and (A1)-(A2) hold. Then, as \(n \rightarrow \infty\),
\begin{align}
  & N^{-1/2}(\bs U_n({\lfloor n t\rfloor})-N\e^{AS^{-1}\ell_1(n,t)}\bs \mu) \conindis \bs Y_2(t)\text{ in } D[0,b),   \label{eq: TR}\\
   & N^{-1}\bs U_n({\lfloor n t\rfloor}) \coninprob  \e^{AS^{-1}\log(1+St/\beta_1)}\bs \mu\text{ in } D[0,b), \label{eq: colour comp TR}
\end{align}
 where \(b = \infty\) if \(S>0\), and \(b = -\beta_1/S\) if \(S<0\). Furthermore, \(\bs Y_2\) is a mean-zero Gaussian process.
 \\~\\
 \textbf{(TSD) Time Step Dominant:} Assume that \(N=o(n)\) and (A1)-(A3) hold. Recall the definition of \(\bs v(\bs \mu)\) in \eqref{eq: imp ob 1}. Then, as \(n\rightarrow \infty\),
\begin{equation}
N^{-1}(n/N)^{-t}\bs U_n(\lfloor N(n/N)^t \rfloor)\coninprob  (S/\beta_1)\bs v(\bs \mu)\text{ in } D(0,\infty). \label{eq: colour comp TSD}
\end{equation}
\textbf{TSD Small Urns:} Assume that \(\Lambda_{\ell}= \{S\}\) with \(S\) simple, and that \(\Lambda_{c}=\emptyset \). Then, as \(n \rightarrow \infty\),
    \begin{equation}
                  n^{-1/2}(\bs U_n(\lfloor n t \rfloor)-N\e^{AS^{-1}\ell_1(n,t)}\bs \mu) \conindis t^{1/2} \bs Y_s(\log(t))\text{ in }  D(0,\infty), \label{eq: small time simp}
    \end{equation}
    where \(\bs Y_s\) is a mean-zero Ornstein-Uhlenbeck process. 
    \\~\\
    \textbf{TSD Critical Urns:} Assume that \(\Lambda_{\ell}=\{S\}\) with \(S\) simple, and that \(\Lambda_{c}=\{S/2\}\). Let \(m\) denote the size of the largest Jordan block across all Jordan blocks with eigenvalue S/2. Then, as \(n\rightarrow \infty\),
    \begin{align}
        & N^{-1/2}(n/N)^{-t/2}\log(n/N)^{-(m-1/2)}(\bs U_n(\lfloor N(n/N)^t \rfloor)-N\e^{AS^{-1}\ell_2(n,t)}\bs \mu)\conindis \bs Y_c(t)\text{ in } D(0,\infty),\label{eq: crit time simp}
    \end{align}
    where \(\bs Y_{c}\) is a mean-zero Gaussian processes.
      \\~\\
      \textbf{TSD Large Urns with \(S\) simple:} Assume that \(\Lambda_{\ell}\neq \{S\}\) and \(S\) is simple. Assume the eigenvalue of \(\Lambda_{\ell}\setminus\{S\}\) with largest real part denoted \(\lambda\) is real. Let \(m\) denote the size of the largest Jordan block across all Jordan blocks with eigenvalue \(\lambda\). Then, as \(n \rightarrow \infty\),
     \begin{align}
         &N^{-1/2}(n/N)^{-\lambda t/S}\log((n/N)^t)^{-(m-1)}(\bs U_n(\lfloor N(n/N)^t \rfloor)-N\e^{AS^{-1}\ell_2(n,t)}\bs \mu) \conindis \bs Z_{\ell}\text{ in } D(0,\infty), \label{eq: large time simp 1}
     \end{align}
 where \(\bs Z_{\ell}\) is a mean-zero Gaussian random variable.
    \\~\\
    \textbf{TSD Large Urns with \(S\) non-simple:} Assume that \(S\) non-simple. Then, as \(n \rightarrow \infty\),
     \begin{align}
         &N^{-1/2}(n/N)^{- t}(\bs U_n(\lfloor N(n/N)^t \rfloor)-N\e^{AS^{-1}\ell_2(n,t)}\bs \mu) \conindis \bs Z_{S}\text{ in } D(0,\infty), \label{eq: large time simp 2}
     \end{align}
 where \(\bs Z_{S}\) is a mean-zero Gaussian random variable.
\end{theorem}
\begin{remark}[Complex eigenvalues for the TSD Critical Urn and Large Urn cases]
\label{remark: comp eigen for c and l}
The real eigenvalue assumption in the TSD Critical Urn and Large Urn cases is not necessary. However, for complex eigenvalues, the behaviour seen is more complicated and results of the form \eqref{eq: crit time simp} and \eqref{eq: large time simp 1} are not possible. The dominant fluctuations of the urn are contained in Jordan projections of the urn on Jordan spaces whose eigenvalues are in \(\Lambda_c\) (Critical Urns) or whose eigenvalue(s) have the largest real part in \(\Lambda_{\ell}\setminus\{S\}\) (Large Urns). If any of these eigenvalues are complex, then we have a contribution to the dominant fluctuations from Jordan space conjugate pairs (recall \eqref{eq: conjugate block}). Suppose \(\lambda\) is one such complex eigenvalue. It will be shown in Section \ref{Sec: results for the polya urn} that the contribution of any Jordan projection with eigenvalue \(\lambda\) scales with the complex~\((n/N)^{\lambda t/S}\). By \eqref{eq: conjugate block}, the complex part of this contribution cancels with the complex part of the contribution coming from the conjugate Jordan space. This leaves an oscillatory real contribution with period \(2\pi S/\mathrm{Im}\lambda \log(n/N)\) in \(t\). Therefore, for Critical Urns, we expect the l.h.s.\ of \eqref{eq: crit time simp} at large \(n\) to be close in distribution to the sum of components of the form
\begin{equation}
  \mathrm{Re}\bs Y_{c,\lambda}(t) \cos(\mathrm{Im}\lambda \log(n/N)t/S)-\mathrm{Im}\bs Y_{c,\lambda}(t) \sin(\mathrm{Im}\lambda \log(n/N)t/S), \quad  t\geq 0, \label{eq: remark complex simple}
\end{equation}
where \(\lambda \in \Lambda_c\), and the \(\bs Y_{c,\lambda}\) are complex mean-zero Gaussian processes that take the place of \(\bs Y_c\). For Large Urns, we have the same result with \(\bs Y_{c,\lambda}\) replaced by complex mean-zero Gaussian random variables \(\bs Z_{\ell,\lambda}\), where \(\lambda\) is now an eigenvalue in \(\Lambda_{\ell}\setminus\{S\}\) of largest real part. 
\end{remark}
\begin{remark}[Asymptotic colour composition]
\label{Remark: ACC}
The results \eqref{eq: colour comp IBD}, \eqref{eq: colour comp TR}, and \eqref{eq: colour comp TSD} give the asymptotic behaviour of the colour composition in each regime. Notice in the TSD regime, when \(S\) is simple (such as for irreducible urns), \eqref{eq: import ob 4} tells us that the limit in \eqref{eq: colour comp TSD} is independent of \(\bs \mu\). Whereas, for \(S\) non-simple, by \eqref{eq: import ob 3}, there is a dependence on \(\bs \mu\).
\end{remark}
\begin{remark}
\label{remark: TR to TSD}
Following on from Remark \ref{Remark: ACC}, it is interesting to know how quickly the asymptotic colour composition in the TR regime converges to the asymptotic colour composition in the TSD regime. This is especially useful when \(S\) is simple, since it gives a quantitative measure on how fast the urn forgets its initial composition. By \eqref{eq: mat exp bound start} and Lemma \ref{lemma: import ob}, we see that the limit process in \eqref{eq: colour comp TR} can be written as 
\begin{equation*}
    \e^{AS^{-1}\log(1+St/\beta_1)}\bs \mu = (1+St/\beta_1)P_{S}\bs\mu +(I-P_{S}) \e^{AS^{-1}\log(1+St/\beta_1)}\bs \mu=(1+St/\beta_1)\bs v (\bs \mu) +(I-P_{S}) \e^{AS^{-1}\log(1+St/\beta_1)}\bs \mu,
\end{equation*}
where \(P_S=\sum_{J \in \mathcal{J}_S}P_J\). Let \(\lambda_2\) be the real part of the eigenvalue(s) of \(A\) with second largest real part, and suppose \(m_2\) is the size of the largest Jordan block across all of these eigenvalues. Then, the second term on the r.h.s.\ can be bounded by \eqref{eq: matrix exp bounds} to give
\begin{equation*}
    \|(I-P_{S}) \e^{AS^{-1}\log(1+St/\beta_1)}\bs \mu\|_2 \leq \cst (1+t)^{\lambda_2/S}\log(1+t)^{m_2-1}.
\end{equation*}
This implies the asymptotic colour composition in the TR regime converges to the asymptotic colour composition in the TSD regime at a polynomial rate in \(t\). Furthermore, the speed of this convergence increases with the relative spectral gap \((S-\lambda_2)/S\).
\end{remark}
\begin{remark}
In the fluctuation results, we see the normalizing terms are not the colour composition limits, but more complex processes involving these \(\ell_1\) and \(\ell_2\) functions. This is because, depending on how \(N\) scales with \(n\), there can be deterministic terms of order between the largest random fluctuations and the colour composition. These terms appear in the GIC case but not the FIC case, since our limits are arising due to an underlying CLT. This causes the random fluctuations to scale with \(N^{1/2}\), whereas the deterministic terms scale with \(N\). This additional scaling the deterministic terms have over the random fluctuations in the GIC case can be enough for lower order deterministic terms to dominate the random fluctuations. The heuristics behind these lower order deterministic terms is discussed further in Section \ref{Sec: Heuristics}.
\end{remark}
\begin{remark}
\label{remark: converge to mu} Theorem 3.1 still holds if the initial colour composition only converges to \(\bs \mu\) as \(n\rightarrow \infty\). The asymptotic colour composition results (\eqref{eq: colour comp IBD}, \eqref{eq: colour comp TR}, and \eqref{eq: colour comp TSD}) are unchanged. However, in the fluctuation results, one needs to replace~\(N\bs \mu\) with \(\bs X_n(0)\) on the l.h.s. Indeed, take the IBD result \eqref{eq: colour comp IBD}. If the initial colour composition does not approach \(\bs\mu\) at a rate faster than \(N^{-1}n^{1/2}\), then these fluctuations will be non-negligible in \eqref{eq: colour comp IBD 2}.
\end{remark}

\subsection{Results for the GIC MCBP}
\label{Sec: results for the GIC MCBP}
In the following two theorems we give a functional limit result for the GIC MCBP at three timescales. These correspond to the timescales discussed in Remark \ref{remark: bp timescales}. The first theorem focuses on the functional limit result, including, for the timescale corresponding to the TSD (\(N=o(n)\)) regime, the asymptotic behaviour of the MCBP when projected on to each Jordan space of \(A\). The second theorem tells us how the limits seen in each of the Jordan spaces depend on each other.
\begin{theorem}
\label{Theorem: Main continuous time}
Let \(N=N(n)\) be a positive integer-valued sequence, such that \(N \rightarrow \infty\) as \(n \rightarrow \infty\). Let \(\bs \mu\) be a non-negative vector, such that \(\sum_{i=1}^d\mu_i = 1\). Let \((\bs X_n)_{n\geq 1}\) be a sequence of MCBPs with initial condition \(N\bs \mu\) and replacement structure~\((\bs a, (\bs \xi_i)_{i=1}^d)\).
\\~\\
\textbf{(ST) Small Time:} Let \(\varepsilon=\varepsilon(n)\) be a positive sequence such that \(\varepsilon \rightarrow 0\) and \(N\varepsilon \rightarrow \infty\) as \(n \rightarrow \infty\). Assume that (A1)-(A2) hold. Then, as \(n \rightarrow \infty\),
    \begin{equation}
    \label{eq: improve 32}
         (N\varepsilon)^{-1/2}(\bs X_n(\varepsilon t)-N\mathrm{e}^{A\varepsilon t}\bs \mu)\conindis \bs W_1(t)  \text{ in }  D[0,\infty).
    \end{equation}
 \textbf{(CT) Critical Time:} Assume that (A1)-(A2) hold. Then, as \(n \rightarrow \infty\),
    \begin{equation}
    \label{eq: discrete 10}
        N^{-1/2}(\bs X_n(t)-N\mathrm{e}^{At}\bs \mu) \conindis \bs W_2(t) \text{ in }  D[0,\infty).
    \end{equation}
\textbf{(LT) Large Time:} Let \(\omega = \omega(n)\) be a positive sequence that tends to \(\infty\) as \(n \rightarrow \infty\). Recall the Perron-Frobenius eigenvalue \(\lambda_1\) of \(A\) with \(m_1\) denoting the size of the largest Jordan block across all Jordan blocks with eigenvalue \(\lambda_1\). Assume that (A1)-(A3) hold. Then, \eqref{eq: first order MCBP}-\eqref{eq: improve 8000} hold jointly. As \(n \rightarrow \infty\),
     \begin{equation}
     \label{eq: first order MCBP}
            N^{-1/2}(\omega t)^{-(m_1-1)}\mathrm{e}^{-\lambda_1 \omega t}(\bs X_n(\omega t)-N\e^{A\omega t}\bs \mu) \conindis   \sum_{J\in \mathcal{J}_{\lambda_1}}\frac{N_A^{m_1-1}\bs V_J}{(m_1-1)!} \text{ in } D(0,\infty).
        \end{equation}
\textbf{(LT\textsubscript{s}) Small Components:} Let \(P_s = \sum_{\lambda \in \Lambda_s}\sum_{J\in \mathcal{J}_{\lambda}}P_{J}\). Then, as \(n \rightarrow \infty\),
        \begin{equation}
        \label{eq: small component convergence}
            N^{-1/2}\omega^{-(m_1-1)/2}\mathrm{e}^{-\lambda_1(\omega+t)/2}P_s( \bs X_n(\omega+t)-N\mathrm{e}^{A(\omega+t)}\bs \mu) \conindis \bs W_{s}(t)  \text{ in }  D(-\infty,\infty).
        \end{equation}
        Let \(K=K(n) \rightarrow \infty\) as \(n \rightarrow \infty\). Then, as \(n \rightarrow \infty\),
        \begin{equation}
            \label{eq: small component convergence in probability}
           K^{-1} N^{-1/2}\omega^{-(2m_1-1)/4}\mathrm{e}^{-\lambda_1 \omega t/2}P_s( \bs X_n(\omega t)-N\mathrm{e}^{A \omega t}\bs \mu) \coninprob 0  \text{ in }  D[0,\infty).
        \end{equation}
 \textbf{(LT\textsubscript{c}) Critical Components:} Let \(\lambda \in \Lambda_c\), \(J\in \mathcal{J}_{\lambda}\) with size \(m\), and \(1 \leq \kappa \leq m\). Then, as \(n\rightarrow \infty\),
        \begin{equation}
            N^{-1/2}\omega^{-(m_1+2\kappa -2)/2}\e^{-\lambda \omega t}N_A^{m-\kappa}P_{J}(\bs X_n(\omega t)-N\e^{A\omega t}\bs \mu)\conindis \bs W_{J,\kappa}(t) \text{ in }  D[0,\infty).
            \label{eq: improve 40}
        \end{equation}
 \textbf{(LT\textsubscript{\(\ell\)}) Large Components:} Let \(\lambda \in \Lambda_{\ell}\), \(J\in \mathcal{J}_{\lambda}\) with size \(m\). Then, as \(n \rightarrow \infty\),
   \begin{equation}
             N^{-1/2} P_{J}(\e^{-A\omega t}\bs X_n(\omega t)-N\bs \mu)\conindis \bs V_{J}  \text{ in }  D(0,\infty).
             \label{eq: improve 80}
        \end{equation}
        Furthermore, by \eqref{eq: mat exp bound start},
        \begin{equation}
            N^{-1/2}\mathrm{e}^{-\lambda \omega t}P_{J}\left(\bs X_n(\omega t)-N\e^{A\omega t}\bs \mu\right) = N^{-1/2}\sum_{i=0}^{m-1}\frac{(\omega t)^i}{i!}N_A^iP_{J}(\e^{-A\omega t}\bs X_n(\omega t)-N\bs \mu). \label{eq: improve 30}
            \end{equation}
           This and \eqref{eq: improve 80} imply, as \(n\rightarrow \infty\), for \(1 \leq \kappa \leq m\),
            \begin{equation}
                \label{eq: improve 8000}
                N^{-1/2}(\omega t)^{-(\kappa-1)}\mathrm{e}^{-\lambda \omega t}N_A^{m-\kappa}P_{J}\left(\bs X_n(\omega t)-N\e^{A\omega t}\bs \mu\right)\conindis \frac{N_A^{m-1}\bs V_{J}}{(\kappa-1)!}  \text{ in }  D(0,\infty).
            \end{equation}
   \noindent The process \(\bs W_1\) is a  Brownian motion. The process \(\bs W_{s}\) is a mean-zero, real,  Ornstein-Uhlenbeck process. The processes~\(\bs W_2\) and \(\bs W_{J,\kappa} \) are mean-zero Gaussian processes. The random variable \(\bs V_{J}\) is a mean-zero Gaussian variable. Moreover, \( \bs W_{J,\kappa}\) and \(\bs V_{J}\) are real if the eigenvalue corresponding to the Jordan space is real and complex otherwise. 
    
    The covariance functions are as follows, for \(0 \leq t_1 \leq t_2 <\infty\) (or \(-\infty < t_1\leq t_2 <\infty\) for Small Components): 
    \begingroup
    \allowdisplaybreaks
    \begin{align*}
        &\Cov(\bs W_1(t_2),\bs W_1(t_1)) = t_1\sum_{i=1}^d \mu_i a_i \E[\bs \xi_i \bs \xi_i']  ,\\
        &\Cov(\bs W_2(t_2),\bs W_2(t_1)) = \sum_{i=1}^d \mathrm{e}^{A(t_2-t_1)}\int_{0}^{t_1}\mathrm{e}^{A( t_1-v)}\E[\bs \xi_i \bs \xi_i']\mathrm{e}^{A'(t_1-v)}a_i(\mathrm{e}^{Av}\bs \mu)_i\mathrm{d}v,\\
        &\Cov(\bs W_{s}(t_2),\bs W_{s}(t_1)) = \sum_{i=1}^d\mathrm{e}^{(A-\lambda_1/2) (t_2-t_1)}\int_{0}^{\infty}P_s\mathrm{e}^{Av}a_i v(\bs\mu)_i\E[\bs \xi_i \bs \xi_i']\e^{A'v}P_s'\mathrm{e}^{-\lambda_1 v}\mathrm{d}v.
        \end{align*}
        For each \(\lambda\in\Lambda_c\), \(J \in \mathcal{J}_{\lambda}\) with size \(m\), and \(1 \leq \kappa \leq m\),
        \begin{align*}
        & \Cov(\bs W_{J,\kappa}(t_2),\overline{\bs W}_{J,\kappa}(t_1))= \sum_{i=1}^dN_A^{m-1}P_{J}a_i v(\bs\mu)_i\E[\bs \xi_i \bs \xi_i'] P_{J}^*N_A'^{m-1}\int_{0}^{t_1}\frac{(t_1-v)^{\kappa-1}(t_2-v)^{\kappa-1}v^{m_1-1}}{(\kappa-1)!^2}\mathrm{d}v,\\
        & \Cov(\bs W_{J,\kappa}(t_2),\bs W_{J,\kappa}(t_1)) = 0 \quad \text{for \(\lambda\in \mathbb{C}\setminus\mathbb{R}\)},
        \end{align*}
        where \(P_{J}^* = \overline{P}_{J}'\). Lastly, for each \(J \in \cup_{\lambda \in \Lambda_{\ell}}\mathcal{J}_{\lambda}\),
        \begin{align*}
        &\Var(\bs V_{J}) = \sum_{i=1}^d  \int_{0}^{\infty}P_{J}e^{-Av}\E[\bs \xi_i \bs \xi_i']e^{-A'v}P_{J}'a_i(e^{Av}\bs \mu)_i\mathrm{d}v,\\
        &\Cov(\bs V_{J},\overline{\bs V}_{J})=\sum_{i=1}^d  \int_{0}^{\infty}P_{J}e^{-Av}\E[\bs \xi_i \bs \xi_i']e^{-A'v}P_{J}^*a_i(e^{Av}\bs \mu)_i\mathrm{d}v.
    \end{align*}
    \end{theorem}
        \begin{remark}
    The timescales chosen in Theorem \ref{Theorem: Main continuous time} match the order of the \(\tau_{n,n}\) as given in Remark 1.2. Indeed, if~\(n=o(N)\) we can take \(\varepsilon = n/N\), and if \(N=o(n)\) we can take \(\omega = \log(n/N)\).
    \end{remark}
    \begin{remark}
    Similarly to Remark \ref{remark: converge to mu}, we can take the initial colour composition converging to \(\bs \mu\) as long as we replace~\(N\bs \mu\) by \(\bs X_n(0)\) in the l.h.s.\ of the limit theorems. Importantly, the covariance functions are unaffected.
    \end{remark}
    \begin{remark}
    \label{remark: BM}
    In the LT\textsubscript{c} case, if one takes \(\kappa=1\), \(m_1=1\), we see that the covariance function of \(\bs W_{J,\kappa}\) is that of a Brownian motion.
    \end{remark}
    \begin{theorem}
\label{theorem: cont time 2}
Take the setting and assumptions of Theorem \ref{Theorem: Main continuous time} LT. The families
 \begin{align*}
      &\mathcal{W}_s := \{\bs W_s\}, \\
 & \mathcal{W}_{c,|\rho|} :=   \cup_{\lambda \in \{\rho,\overline{\rho}\}}  \cup_{J\in \mathcal{J}_{\lambda}}\cup_{\kappa=1}^{m_J}\{\bs W_{J,\kappa}\}, \text{ for different conjugate pairs \(\rho,\overline{\rho} \in \Lambda_c\)},\\
  &\mathcal{V}_{\ell}      := \cup_{\lambda \in \Lambda_{\ell}}\cup_{J\in \mathcal{J}_{\lambda}}\{\bs V_{J}\},
\end{align*}
    are pairwise independent, where \(m_J\) denotes the size of the Jordan block \(J\). Within dependent families we have the following covariance functions:
\\~\\
Let \(B:= \sum_{i=1}^d a_i v(\bs\mu)_i\E[\bs \xi_i\bs \xi_i']\). For \(0\leq t_1\leq t_2 <\infty\), \(\rho \in \Lambda_c\), and \(\bs W_{J_1,\kappa_1},\bs W_{J_2,\kappa_2} \in\mathcal{W}_{c,|\rho|}\) with \(J_1,J_2\) having eigenvalues \(\rho_1,\rho_2\in \{\rho,\overline{\rho}\}\) and sizes \(d_{1},d_{2}\) respectively,
\begin{align*}
 &\Cov(\bs W_{J_{2},\kappa_2}(t_2),\overline{\bs W}_{J_{1},\kappa_1}(t_1)) = \bs 1_{\{\rho_{1}=\rho_{2}\}} N_A^{d_{2}-1}P_{J_2}BP_{J_1}^*N_A'^{d_{1}-1}\int_{0}^{t_1}\frac{(t_1-v)^{\kappa_1-1}(t_2-v)^{\kappa_2-1}v^{m_1-1}}{(\kappa_2-1)!(\kappa_1-1)!}\mathrm{d}v,\\
  &\Cov(\bs W_{J_2,\kappa_2}(t_2),\bs W_{J_1,\kappa_1}(t_1))  = \bs 1_{\{\rho_{1}=\overline{\rho}_{2}\}} N_A^{d_{2}-1}P_{J_2}BP_{J_1}'N_A'^{d_{1}-1}\int_{0}^{t_1}\frac{(t_1-v)^{\kappa_1-1}(t_2-v)^{\kappa_2-1}v^{m_1-1}}{(\kappa_2-1)!(\kappa_1-1)!}\mathrm{d}v.
\end{align*}
For \(\bs V_{J_2},\bs V_{J_1}\in \mathcal{V}_{\ell}\),
\begin{align*}
   & \Cov(\bs V_{{J_2}},\overline{\bs V}_{J_1})=\sum_{i=1}^d  \int_{0}^{\infty}P_{J_2}\e^{-Av}\E[\bs \xi_i \bs \xi_i']\e^{-A'v}P_{J_1}^*a_i(\e^{Av}\bs \mu)_i\mathrm{d}v,\\
    &\Cov(\bs V_{J_2},\bs V_{J_1})=\sum_{i=1}^d \int_{0}^{\infty}P_{J_2}\e^{-Av}\E[\bs \xi_i \bs \xi_i']\e^{-A'v}P_{J_1}'a_i(\e^{Av}\bs \mu)_i\mathrm{d}v.
\end{align*}
    \endgroup
\end{theorem}
    \begin{remark}[Importance of \(N_A^{m-\kappa}\) in the LT\textsubscript{c} and LT\textsubscript{\(\ell\)} cases]
    \label{remark: Nilop}
    Recall the coefficient functions for the Jordan basis of~\(A\) \eqref{eq:def_alphas}. By linear independence of the Jordan basis, for any Jordan space \(J\), when we study the projection of the MCBP on to this space \((P_J \bs X_n)\), we are simply studying the joint behaviour of \(\alpha_{J,i}(\bs X_n)\) for \(1\leq i \leq m\). For critical and large Jordan spaces, as \(n\rightarrow \infty\), the \(\alpha_{J,i}(\bs X_n(\omega ))\) (and their fluctuations) are of decreasing orders of magnitude for \(1\leq i\leq m\). Therefore,~\(P_J \bs X_n\) is too rough an object to capture the full behaviour of the MCBP in this Jordan space, since only the behaviour of \(\alpha_{J,1}(\bs X_n(\omega ))\) will appear in the limit (i.e.\ we only see the asymptotic behaviour of the MCBP projected on the Jordan basis vector \(\bs v_{J,1}\)). By applying \(N_A^{m-\kappa}\) to \(P_J \bs X_n\), we remove the first \(m-\kappa\) coefficient functions from \(P_J \bs X_n\) (see \eqref{eq: coef func} and \eqref{eq:nilpotent operator}). Therefore, the dominant term becomes \(\alpha_{J,m-\kappa+1}(\bs X_n(\omega ))\), and we gain access to the asymptotic behaviour of the MCBP projected on the Jordan basis vector \(\bs v_{J,m-\kappa+1}\). This being asymptotically the \(\kappa\)th smallest out of all Jordan basis vector projections within \(J\). 
    
   However, by applying \(N_A^{m-\kappa}\) to \(P_J \bs X_n\), we have attached the coefficient function \(\alpha_{J,m-\kappa+i}\) to the Jordan basis vector~\(\bs v_{J,i}\) for \(1\leq i \leq \kappa\). Therefore, the limit we see in Theorem \ref{Theorem: Main continuous time} resides in the space \(\mathrm{span} \{\bs v_{J,1},...,\bs v_{J,\kappa}\}\) (in fact, the limit must be in \(\mathrm{span} \{\bs v_{J,1}\}\) by the previous paragraph), whereas the process this limit is describing originally lived in \(\mathrm{span}\{\bs v_{J,m-\kappa+1},...,\bs v_{J,m}\}\) before application of the nilpotent operator. We can account for this by applying the operator~\((N_A^{\dagger})^{m-\kappa}\) to the limit, where
\begin{equation*}
    N_A^{\dagger}\bs v_{J,m}=0,\quad N_A^{\dagger}\bs v_{J,i}=\bs v_{J,i+1}, \quad 1\leq i\leq m-1.
\end{equation*}
Indeed, applying \((N_A^{\dagger})^{m-\kappa}\) to \(N_A^{m-\kappa}P_J \bs X_n\) returns each coefficient function to its associated basis vector.
    \end{remark}
    \begin{remark}[Behaviour within each Jordan basis vector]
    \label{remark: within each bv}
   From Theorem \ref{Theorem: Main continuous time} LT one can recover the asymptotic behaviour of the MCBP projected on each Jordan basis vector of \(A\). For critical and large basis vectors, this was explained in Remark~\ref{remark: Nilop}. For small basis vectors, the asymptotic behaviour of each is within the same order of magnitude. Therefore, each small basis vector is represented in the limit \eqref{eq: small component convergence}, and can be individually recovered by applying the continuous mapping theorem with the basis vector's coefficient function as the map.  
    \end{remark}
    \begin{remark}[Importance of \eqref{eq: small component convergence in probability}]
    \label{Remark: Small importance}
    Since the fluctuations of the small component are on the timescale of \(\omega +t\), \(t \in (-\infty,\infty)\), a problem is posed when we want to show the limiting behaviour of the MCBP projected on the union of multiple Jordan spaces (for example, showing \eqref{eq: first order MCBP}). By Theorem \ref{Theorem: Main continuous time}, it is clear the fluctuations of the small component will be dominated by the fluctuations of any critical or large component. However, \eqref{eq: small component convergence} only implies this is true on the timescale of \(\omega +t\), \(t \in (-\infty,\infty)\), whereas we would like this to be true on the timescale of the critical and large component fluctuations being \(\omega t\), \(t \in [0,\infty)\). This is exactly what \eqref{eq: small component convergence in probability} gives us. Indeed, excluding small components, the smallest possible fluctuations occur from critical components when \(\kappa=1\); these being of order \(\omega^{1/2}\e^{\lambda_1\omega t/2}\). If we take \(K=\omega^{1/4}\) in \eqref{eq: small component convergence in probability}, we see the fluctuations of the small components are \(o_p(\omega^{1/2}\e^{\lambda_1\omega t/2})\), thus are dominated by the fluctuations of critical and large components over the timescale of interest.  
    \label{Remark: rmneed1}
    The notation \(\mathcal{O}_p,o_p\) is defined as \(\mathcal{O},o\) in probability.  
    \end{remark}
\subsection{Additional results for balanced GIC P\'olya urns}
\label{Sec: results for the polya urn}
In the first part of this section we state two theorems extending the results of Theorem \ref{theorem: main result simplified}. The first focuses on the TSD regime, where we give a functional limit theorem for the urn projected on to each Jordan space of \(A\). In the second theorem, we define the limits seen in the first theorem and Theorem \ref{theorem: main result simplified} in terms of the limits seen for the MCBP in Theorem \ref{Theorem: Main continuous time}. After, we go over some examples where these results can be applied.
\begin{theorem}
\label{theorem: Main results discrete}
Let \(N=N(n)\) be a positive integer-valued sequence, such that \(N \rightarrow \infty\) as \(n \rightarrow \infty\). Let \(\bs \mu\) be a non-negative vector, such that \(\sum_{i=1}^d \mu_i=1\). Let \((\bs U_n)_{n \geq 1}\) be a sequence of balanced P\'olya urns with initial condition \(N\bs \mu\), and balanced replacement structure \((\bs a, (\bs \xi_i)_{i=1}^d)\), such that the amount of mass added per time step is \(S\in \mathbb{R}\setminus\{0\}\). Assume that \(N=o(n)\) and (A1)-(A3) hold. Let \(\beta_1=\bs a \cdot \bs \mu\), \(\ell_1(n,t)=\log(1+S{ n t}/\beta_1 N),\) and \(\ell_2(n,t) =\log(1+Sn^t/ \beta_1 N^{t})\).
\\~\\
\textbf{(TSD\textsubscript{s}) Small Components:} Let \(P_s = \sum_{\lambda \in \Lambda_s}\sum_{J \in \mathcal{J}_{\lambda}}P_J\). Then, as \(n \rightarrow \infty\),
    \begin{equation}
    \label{eq: small urn flucts}
   n^{-1/2}P_s(\bs U_n(\lfloor n t \rfloor)-N\e^{AS^{-1}\ell_1(n,t)}\bs \mu) \conindis t^{1/2}\bs Y_s(\log(t))\text{ in } D(0,\infty),
    \end{equation}
    where \(\bs Y_s\) is the same Ornstein-Uhlenbeck process as in \eqref{eq: small time simp}.  Suppose \(K=K(n)\rightarrow \infty\) as \(n\rightarrow \infty\). Then, as \(n \rightarrow \infty\),
     \begin{equation}
     \label{eq: TSDs big con}
  N^{-1/2}(n/N)^{-t/2}\log(n/N)^{-1/4}P_s(\bs U_n(\lfloor N(n/N)^t\rfloor)-N\e^{AS^{-1}\ell_2(n,t)}\bs \mu) \coninprob 0 \text{ in } D(0,\infty).
    \end{equation}
   \textbf{(TSD\textsubscript{c}) Critical Components:} Let \(\lambda \in \Lambda_c\), \(J \in \mathcal{J}_{\lambda}\) with size \(m\), and \(1 \leq \kappa \leq m\). Then, as \(n \rightarrow \infty\),
    \begin{equation}
    \label{eq: critty comps}
  N^{-1/2}(n/N)^{-\lambda t/S} \log( n /N)^{-(\kappa -1/2)}N_A^{m-\kappa}P_{J}(\bs U_n(\lfloor N (n/N)^t\rfloor)-N\e^{AS^{-1}\ell_2(n,t)}\bs \mu)\conindis \bs Y_{J,\kappa}(t)\text{ in } D(0,\infty),
\end{equation}
where \(\bs Y_{J,\kappa}\) is a mean-zero Gaussian process, and is complex if \(J\) has complex eigenvalue. 
\\~\\
 \textbf{(TSD\textsubscript{\(\ell\)}) Large Components:} Let \(\lambda \in \Lambda_{\ell}\), \(J \in \mathcal{J}_{\lambda}\) with size \(m\), and \(1 \leq \kappa \leq m\). Then, as \(n \rightarrow \infty\),
\begin{equation}
   N^{-1/2}(n/N)^{-\lambda t/S} \log( (n/N)^t )^{-(\kappa -1)}N_A^{m-\kappa}P_J(\bs U_n(\lfloor N(n/N)^t\rfloor-N\e^{AS^{-1}\ell_2(n,t)}\bs \mu )\conindis  \frac{N_A^{m-1}\bs Z_J}{S^{\kappa-1}(\kappa-1)!}\text{ in } D(0,\infty),
\end{equation}
where \(\bs Z_J\) is a mean-zero Gaussian random variable, and is complex if \(J\) has complex eigenvalue.
\end{theorem}
\noindent Analogous versions of Remarks \ref{remark: converge to mu}, \ref{remark: Nilop}, \ref{remark: within each bv}, and \ref{Remark: Small importance} hold for Theorem \ref{theorem: Main results discrete}. In particular, the importance of Remark~\ref{Remark: Small importance} is now made clear, since \eqref{eq: TSDs big con} is necessary to recover the results of Theorem \ref{theorem: main result simplified} TSD Critical Urns and Large Urns from Theorem \ref{theorem: Main results discrete} and Slutsky's lemma. 

In the following theorem, we state the limit processes seen in Theorems \ref{theorem: main result simplified} and \ref{theorem: Main results discrete} in terms of the limits seen in Theorems \ref{Theorem: Main continuous time} and \ref{theorem: cont time 2}. This allows one to find the covariance functions for the urn limits by using the covariance functions for the MCBP limits.
\begin{theorem}
\label{theorem: main results discrete 2}
With reference to the limit processes in Theorems \ref{theorem: main result simplified}, \ref{Theorem: Main continuous time}, \ref{theorem: cont time 2}, and \ref{theorem: Main results discrete}:
\\~\\
 \textbf{IBD:} We have
    \begin{equation}
     \bs Y_1(t) \eqindis \beta_1^{-1/2}\bs W_1(t)- S^{-1}\beta_1 ^{-3/2} A\bs \mu \sum_{i=1}^d a_i W_1(t)_i, \quad t\geq 0.
     \label{eq: limit for IBD}
\end{equation}
This implies the process \(\bs Y_1\) is a Brownian motion with covariance function, for \(0\leq  t_1 \leq t_2\),
\begin{equation}
    \Cov (\bs Y_1(t_2), \bs  Y_1(t_1)) = t_1\F( \beta_1^{-1}\sum_{i=1}^d   a_i \mu_i \E[\bs \xi_i \bs \xi_i'] - \beta_1^{-2} \F(\sum_{i=1}^d   a_i \mu_i \E[\bs \xi_i]\R)\F(\sum_{i=1}^d   a_i \mu_i \E[\bs \xi_i]\R)'\R).
\end{equation}
 \textbf{TR:} We have
\begin{equation}
\label{eq: logo time}
    \bs Y_2(t) \eqindis \bs W_2(S^{-1}\log(1+St/\beta_1))- \frac{\e^{AS^{-1}\log(1+St/\beta_1)}A\bs \mu}{\beta_1 S+S^2t} \sum_{i=1}^da_i W_2(S^{-1}\log(1+S t/\beta_1))_i, \quad t\geq 0.
\end{equation}
 \textbf{TSD:} For the TSD regime, all equalities in distribution hold jointly, where by this we mean the joint convergence given in Theorem \ref{theorem: cont time 2} still holds for the urn limits.
 \\~\\
 \textbf{TSD\textsubscript{s}:} We have
\begin{equation}
     \label{eq: limit for TSDs}
    \bs Y_s(\log(t)) \eqindis (S/\beta_1)^{1/2}\bs W_s(\log(St/\beta_1 )/S)=: (S/\beta_1)^{1/2}\hat{\bs W}_s(\log(t)/S), \quad t>0,
\end{equation}
where \(\hat{\bs W}_s\) has the same distribution of \(\bs W_s\) since Ornstein-Uhlenbeck processes are stationary.
\\~\\
 \textbf{TSD\textsubscript{c}:} For \(\lambda \in \Lambda_c\), \(J\in \mathcal{J}_{\lambda}\) with size \(m\), and \(1\leq \kappa \leq m\), we have
\begin{equation}
     \label{eq: limit for TSDc}
    \bs Y_{J,\kappa}(t) \eqindis S^{-(\kappa-1/2)}(S/\beta_1)^{\lambda/S} \bs W_{J,\kappa}(t), \quad t>0.
\end{equation}
Furthermore, the limit in Theorem \ref{theorem: main result simplified} TSD Critical Urns is defined as follows. Let \(m_c\) be the size of the largest Jordan block across all Jordan blocks with eigenvalue in \(\Lambda_c\). For \(\lambda \in \Lambda_c\), let \(J_1,...,J_k\) be the list of Jordan blocks with eigenvalue \(\lambda\) and size \(m_c\) (it is possible for \(k=0\) for some \(\lambda \in \Lambda_c\) but not all). Then, as in \eqref{eq: remark complex simple},
\begin{equation}
    \bs Y_{c,\lambda}(t) \eqindis \sum_{i=1}^k  \bs Y_{J_i,m_c}(t), \quad t>0.
\end{equation}
In particular, if \(\Lambda_c=\{S/2\}\), then \eqref{eq: crit time simp} holds with \((\bs Y_{c}(t))_{t>0} = (\bs Y_{c,S/2}(t))_{t>0}\).
\\~\\
 \textbf{TSD\textsubscript{\(\ell\)}:} Let \(J_{\bs a}\) be the Jordan space corresponding to the left eigenvector \(\bs a\) with eigenvalue \(S\) (see Lemma \ref{lemma: balanced frob}). Let \(V_1= \bs V_{\bs a}\cdot \bs a\), where \(\bs V_{\bs a}\) is as in \eqref{eq: improve 80}. For \(\lambda \in \Lambda_{\ell}\), \(J\in \mathcal{J}_{\lambda}\), we have
\begin{equation}
\label{eq: large fluctuations depend}
\bs Z_J \eqindis (S/\beta_1)^{\lambda/S}(\bs V_{J}-(\beta_1 S)^{-1}V_1AP_J\bs \mu ).
\end{equation}
Furthermore, the limits in Theorem \ref{theorem: main result simplified} TSD Large Urns are defined as follows:
\\~\\
\textbf{Case of \(S\) simple:} Let \(\Lambda_{\ell,\mathrm{max}}\) be the set of eigenvalues in \(\Lambda_{\ell}\setminus\{S\}\) with largest real part. Let \(m_{\ell}\) be the size of the largest Jordan block across all Jordan blocks with eigenvalue in \(\Lambda_{\ell,\mathrm{max}}\). For \(\lambda \in \Lambda_{\ell,\mathrm{max}}\), let \(J_1,...,J_k\) be the list of Jordan blocks with eigenvalue \(\lambda\) and size \(m_{\ell}\). Then, as in Remark \ref{remark: comp eigen for c and l},
\begin{equation}
    \bs Z_{\ell, \lambda}\eqindis \sum_{i=1}^{k} \frac{N_A^{m_{\ell}-1}\bs Z_J}{S^{m_{\ell}-1}(m_{\ell}-1)!}.
\end{equation}
In particular, if \(\Lambda_{\ell,\mathrm{max}}=\{\lambda\}\) and \(\lambda\) is real, then \eqref{eq: large time simp 1} holds with \(\bs Z_{\ell}= \bs Z_{\ell, \lambda}\).
\\~\\
\textbf{Case of \(S\) non-simple:} We have
\begin{equation}
    \bs Z_S = \sum_{J \in \mathcal{J}_S}\bs Z_J.
\end{equation}
\end{theorem}
\begin{remark}
By taking the identity urn in Theorems \ref{theorem: main result simplified} and \ref{theorem: main results discrete 2}, we recover Corollary 1 of Borovkov \cite{Borov}. However, we note Borovkov \cite{Borov} is able to show the fluctuations converge in probability using additional properties of the identity urn.
\end{remark}
\begin{remark}
\label{remark: random walk}
In the IBD regime, the Brownian motion \(\bs Y_1\) seen in the limit is identical to the Brownian motion one would see by renormalizing the random walk with jump probabilities \((a_i\mu_i/\beta_1)_{i=1}^d\) and jump distribution \((\bs \xi_i)_{i=1}^d\). This is no surprise, since due to the urns overwhelming initial composition, the urns colour composition remains close to \(\bs\mu\) throughout the process, hence the urns behaviour is close to this random walk.
\end{remark}
\begin{remark}
\label{remark: TSD REMARK 27} 
In the TSD regime, it is interesting to know which of the limits in Theorem \ref{theorem: main results discrete 2} depend on the initial colour composition \(\bs \mu\). Since these limits correspond to the fluctuations of the urn, a dependence on \(\bs \mu\) implies the fluctuations depend on the early draws of the urn. Whereas, no dependence on \(\bs\mu\) implies the fluctuations depend on the asymptotic behaviour of the urn.
\\~\\
\textbf{TSD\textsubscript{\(\ell\)}:} For \(\lambda\in \Lambda_{\ell}\), \(J\in \mathcal{J}_{\lambda}\), it is clear that \(\bs Z_J\) depends on \(\bs \mu\). In \eqref{eq: large fluctuations depend} \(\bs \mu\) directly appears and appears through the variances of \(\bs V_J\) and \(V_1\) (See Theorem \ref{Theorem: Main continuous time}). One can check these don't cancel with each other. Therefore, the fluctuations of TSD Large urns depend on~\(\bs \mu\) whether \(S\) is simple or not.
\\~\\
\textbf{TSD\textsubscript{c}:}
By Theorem \ref{theorem: main results discrete 2} TSD\textsubscript{c} and Theorem \ref{Theorem: Main continuous time}, we have, for \(\lambda \in \Lambda_c\), \(J\in \mathcal{J}_{\lambda}\) with size \(m\), \(1\leq \kappa \leq m\), and \(0<t_1\leq t_2 <\infty\),
\begin{equation*}
    \Cov(\bs Y_{J,\kappa}(t_2),\overline{\bs Y}_{J,\kappa}(t_1))= (S/\beta_1)\sum_{i=1}^dN_A^{m-1}P_{J}a_i v(\bs\mu)_i\E[\bs \xi_i \bs \xi_i'] P_{J}^*N_A'^{m-1}\int_{0}^{t_1}\frac{(t_1-v)^{\kappa-1}(t_2-v)^{\kappa-1}v^{m_1-1}}{(\kappa-1)!^2}\mathrm{d}v.
\end{equation*}
In the case of \(S\) simple, we have, by \eqref{eq: import ob 3}, that \(\bs v(\bs \mu) = \beta_1 \bs v_1\). Therefore, if \(S\) is simple, the covariance function is independent of \(\bs\mu\), and so the TSD Critical Urn fluctuations are independent of \(\bs \mu\). If \(S\) is non-simple, then there is a dependence on~\(\bs \mu\) through \(\bs v(\bs\mu)\). However, one can argue that this is still only a dependence on the asymptotic colour composition, which in itself depends on \(\bs v(\bs \mu)\) (see \eqref{eq: colour comp TSD}).
\\~\\
\textbf{TSD\textsubscript{s}:} By Theorem \ref{theorem: main results discrete 2} TSD\textsubscript{s} and Theorem \ref{Theorem: Main continuous time}, for \(0< t_1\leq t_2\),
\begin{align*}
    &\Cov(\bs Y_s(\log(t_2)),\bs Y_s(\log (t_1))) = (S/\beta_1)\sum_{i=1}^d\mathrm{e}^{(A-\lambda_1/2) (t_2-t_1)}\int_{0}^{\infty}P_s\mathrm{e}^{Av}a_i v(\bs\mu)_i\E[\bs \xi_i \bs \xi_i']\e^{A'v}P_s'\mathrm{e}^{-\lambda_1 v}\mathrm{d}v.
\end{align*} 
We see the dependence on \(\bs\mu\) is identical to the TSD\textsubscript{c} case. Independent of \(\bs \mu\) if \(S\) is simple, and dependent on \(\bs \mu\) through \(\bs v(\bs \mu)\) if \(S\) is non-simple.

From these results, we see the fluctuations in the TSD regime depend on the initial composition when the relative spectral gap \((S-\lambda_1)/S\) is too small. The intuitive reason behind this will be discussed in our heuristics in Section \ref{Sec: Heuristics}.
\end{remark}
\begin{remark}
Our results do not give a speed of convergence with respect to the number of initial balls and time steps of the urn (\(N\) and \(n\)). This is an interesting open problem, since in application we would like to know when we have enough initial balls/time steps to approximate the behaviour of the urn by the results of this paper. Using additional results about the identity urn, in Borovkov \cite{Borov}, they show, for the identity urn, a rate of \(n^{-1/2}+N^{-1/2}\) for the speed of convergence of the fluctuations in Theorem \ref{theorem: main result simplified} to a Gaussian. In the general case, faster speeds in \(n\) have been shown for specific replacement rules for FIC Urns (see Hwang \cite{Hwang}). It is unclear how this translates to the GIC~case. It seems reasonable to see this faster rate in the TSD regime where the FIC case behaviour is relevant, whereas in the IBD and TR regimes, we expect to see the same rates as shown by Borovkov \cite{Borov}.
\end{remark}

\begin{remark}[Relaxation of assumption (A2)]
\label{remark: tenability}
Assumption (A2) can be relaxed to allow us to remove more than one ball of the colour drawn or colours other than the ball drawn. What is important is that the replacement structure and initial composition is such that the urn can never attempt to remove a ball that does not exist (i.e.\ go into a negative number of balls for some colour). This is often called a ``tenability assumption" in the literature. In fact, as long as this assumption holds, one can further take the replacement rule of the urn to be arbitrary real numbers instead of integers.

Under this assumption, as explained in Remark 4.2 of Janson \cite{Janson}, the continuous-time process can still be defined. It is no longer a MCBP, however the process still satisfies the strong Markov property and still satisfies an important martingale to be given in Lemma \ref{prop: martingale bp}. Furthermore, the process satisfies a weaker version of the branching property, where as long as one splits the initial composition into two tenable compositions, then they evolve independently from each other. These properties are enough to generalize the proofs of Section \ref{Sec: Prelim for proofs} onward to tenable urns. However, in the \(N=o(n)\) regime, we also need to assume the eigenvalue \(\lambda_1\) is real, since this no longer comes for free from the Perron-Frobenius theorem (Theorem \ref{theorem: PF}).
\end{remark}
\begin{example}[Friedman's urn in the TSD regime.]
Recall Friedman's urn defined in Example \ref{example: friedmans urn}. In this example, we focus on the urns behaviour in the TSD\textsubscript{\(\ell\)} case (\(\alpha > 3\gamma\)). (The TSD\textsubscript{s} and TSD\textsubscript{c} cases have Gaussian fluctuations in the FIC~case, this leads to similar behaviour in the GIC case. The behaviour for Friedman's urn in the FIC case can be found in Janson \cite{Janson} Example 3.27.) We have that the eigenvalues of \(A\), \(\lambda_1 = \alpha+\gamma\) and \(\lambda_2 = \alpha-\gamma\) are both simple. Therefore, Theorem \ref{theorem: main result simplified} TSD Large Urns with \(S\) simple implies
\begin{equation*}
  N^{-1/2}(n/N)^{-\frac{\alpha-\gamma}{\alpha+\gamma}t}(\bs U_n(\lfloor N(n/N)^t \rfloor)-N\e^{A(\alpha+\gamma)^{-1}\ell_2(n,t)}\bs \mu) \conindis \bs Z\text{ in } D(0,\infty)
  \end{equation*}
as \(n\rightarrow \infty\), where \(\bs Z\) is a real Gaussian random variable, \(S=\alpha+\gamma\), \(\lambda = \alpha-\gamma\), and \(\beta_1 =1\). Next, using \eqref{eq:identity} in the first line and Theorem \ref{theorem: main results discrete 2} TSD\textsubscript{\(\ell\)} in the second, we see
  \begin{align*}
      &\e^{A(\alpha+\gamma)^{-1}\ell_2(n,t)}=(P_{J_1}+P_{J_2})\e^{A(\alpha+\gamma)^{-1}\ell_2(n,t)}= (1+(\alpha+\gamma)(n/N)^t)P_{J_1}+(1+(\alpha+\gamma)(n/N)^t)^{\frac{\alpha-\gamma}{\alpha+\gamma}}P_{J_2},\\
      &\bs Z = (\alpha+\gamma)^{\frac{\alpha-\gamma}{\alpha+\gamma}}(\bs V_{J_2}-(\alpha+\gamma)^{-1}V_1AP_{J_2}\bs \mu).
  \end{align*}
  To calculate the variance of \(\bs Z\), we first use Theorems \ref{Theorem: Main continuous time} and \ref{theorem: cont time 2} to see
  \begin{align*}
      &\Var(\bs V_{J_2}) = \frac{(\alpha-\gamma)^2}{4(\alpha-3\gamma)}\left( {\begin{array}{cc}
    1 & -1 \\
    -1 & 1 \\
  \end{array} } \right) = \frac{(\alpha-\gamma)^2}{\alpha-3\gamma}\bs v_2 \bs v_2',\\
  & \Var(\bs V_{J_1}) = \Var( V_{1}\bs v_1) =  \frac{(\alpha+\gamma)^2}{4(\alpha-3\gamma)}\left( {\begin{array}{cc}
    1 & 1 \\
    1 & 1 \\
  \end{array} } \right)= \frac{(\alpha+\gamma)^2}{\alpha-3\gamma}\bs v_1 \bs v_1' , \\
  &\Cov(\bs V_{J_2},  \bs V_{J_1}) =\frac{(\alpha +\gamma)(\mu_1-\mu_2)}{4}\left( {\begin{array}{cc}
    1 & 1 \\
    -1 & -1 \\
  \end{array} } \right)  = (\alpha +\gamma)(\mu_1-\mu_2) \bs v_2 \bs v_1'.
  \end{align*}
 These calculations imply 
  \begin{align*}
      &\Cov(\bs V_{J_2}, (\alpha+\gamma)^{-1}V_1AP_{J_2}\bs \mu) = \frac{(\alpha-\gamma)(\mu_1-\mu_2)}{\alpha+\gamma}\Cov(\bs V_{J_2}, V_1\bs v_2) = (\alpha-\gamma)(\mu_1-\mu_2)^2\bs v_2 \bs v_2',\\
      &\Var((\alpha+\gamma)^{-1}V_1AP_{J_2}\bs \mu)) = \frac{(\alpha-\gamma)^2(\mu_1-\mu_2)^2}{(\alpha+\gamma)^2}\Var( V_1 \bs v_2) = \frac{(\alpha-\gamma)^2(\mu_1-\mu_2)^2}{\alpha -3\gamma }\bs v_2 \bs v_2'.
  \end{align*}
  Therefore, we see
  \begin{align*}
      \Var(\bs Z) &= (\alpha+\gamma)^{\frac{2(\alpha-\gamma)}{\alpha+\gamma}}\Var(\bs V_{J_2}-(\alpha+\gamma)^{-1}V_1AP_{J_2}\bs \mu)\\
      &= (\alpha+\gamma)^{\frac{2(\alpha-\gamma)}{\alpha+\gamma}}\frac{(\alpha-\gamma)^2(1+(\mu_1-\mu_2)^2)-2(\alpha-\gamma)(\alpha-3\gamma)(\mu_1-\mu_2)^2}{\alpha-3\gamma}\bs v_2 \bs v_2'\\
      &=(\alpha+\gamma)^{\frac{2(\alpha-\gamma)}{\alpha+\gamma}}\frac{(\alpha-\gamma)^2(1+(\mu_1-\mu_2)^2)-2(\alpha-\gamma)(\alpha-3\gamma)(\mu_1-\mu_2)^2}{2(\alpha-3\gamma)} \left({\begin{array}{cc}
    1 & 1 \\
    -1 & -1 \\
  \end{array} } \right).
  \end{align*}
\end{example}
\begin{example}[An application to matching problems]
\label{example: matching problem}
Random matching problems have been of significant interest since the seminal paper of Gale and Shapley \cite{Gale}. In these problems, one has a collection of \(d\) colours of points and looks to match these points according to some predetermined rule. One such rule has a nice interpretation in terms of an urn process. Let \((\bs U_n)_{n\geq 1}\) be a sequence of \(d\)-dimensional P\'olya urns with initial composition \((N/d)\bs 1\) (assume integer valued), and replacement structure \((\bs 1, (-\bs e_i)_{i=1}^d)\). Therefore, at each draw, we select a ball uniformly at random to be removed from the urn. The urn after every two draws can be seen as matching pairs of balls in the system, where each ball is equally likely to be chosen for matching. Theorem \ref{theorem: Main results discrete} TR gives us the asymptotic behaviour of the number of balls of each colour throughout the matching process as the number of initial balls to be matched tends to infinity. In this model, we have \(A=-I\), \(S=-1\), \(\beta_1 = 1\). Therefore, by Theorems \ref{theorem: main result simplified} TR and \ref{theorem: main results discrete 2} TR,
\begin{equation}
    n^{-1/2}(\bs U_{n}(\lfloor nt\rfloor) - d^{-1}n(1-t)\bs 1 )\conindis \bs W_2(-\log(1-t))-d^{-1}\sum_{i=1}^d W_2(-\log(1-t))_i\bs 1\text{ in } D[0,1) \label{eq: brownian bridge}
\end{equation}
as \(n\rightarrow \infty\). Using Theorem \ref{Theorem: Main continuous time}, one can check that the components of \(\hat{\bs W}_2(t):=d^{1/2}\bs W_2(-\log(1-t))\), \(t \in [0,1)\), are independent Brownian bridge processes. Indeed, a Brownian bridge makes sense here. By \eqref{eq: brownian bridge}, we have that the colour composition to first order remains a constant \(d^{-1}\bs 1\) throughout the process, therefore each colour should be close to equally likely to be removed from the urn. This suggests, if we look at a specific colour, we should see i.i.d.\ fluctuations that must eventually revert to 0 (since the urn will eventually empty) which naturally suggests a Brownian bridge. However, the colours are not independent. When we remove a colour, we forgo the opportunity to remove any of the other colours. This is what the second term in \eqref{eq: brownian bridge} represents (i.e.\ we cannot have too many of our Brownian bridges doing well, since the number of balls removed at any given time step is deterministic).
\end{example}

\section{Heuristic interpretation of the main results}
\label{Sec: Heuristics}
We can gain a good intuition of the behaviour seen in all three regimes in the GIC case by thinking about the following representation of the urn: For \(b>0\), let \(\bs U\) be a P\'olya urn with initial number of balls \(Nb\) and colour composition \(\bs \mu\), where we assume \(b\bs \mu\) is integer valued. Now suppose we take the initial \(Nb\) balls and place them into \(N\) total urns, where we place composition \(b\bs\mu\) into each urn. Denote these urns \(U_i\), \(1\leq i \leq N\). Then, we run the P\'olya urn process as usual, but we always note which of the \(N\) urns the ball drawn was chosen from, and we always add the set of new balls into the same urn from which the ball was chosen (i.e.\ we are keeping track of the progeny of each collection of \(b \bs \mu\) balls). For~\(1\leq i \leq N\), let \(D_i(n)\) be the number of times we have drawn from urn \(\bs U_i\) after time step \(n\) of \(\bs U\). Then, we have the following representation of the urn at time step \(n\)
\begin{equation}
    \bs U(n) = \sum_{i=1}^N \bs U_i(D_i(n)). \label{eq: heuristic}
\end{equation}
With this representation in mind, the first part of this heuristic is noting that the three regimes correspond to the \((\bs U_i)_{i=1}^N\) having seen, on average, 0 draws (IBD), a finite positive number of draws (TR), or infinitely many draws (TSD) by time step \(n\) as \(n\rightarrow \infty\). The second part of the heuristic is the fact that, if the original urn has balanced replacement structure, then conditionally on the \((D_i(n))_{i=1}^N\), the \((\bs U_i(D_i(n))_{i=1}^N\) are independent. To see this, notice that choosing the~\((\bs U_i)_{i=1}^N\) to be drawn from at time step \(n\) only depends on the mass inside each of the
 \((\bs U_i)_{i=1}^N\) at time step \(n-1\). Therefore, the~\((D_i(n))_{i=1}^N\) only depend on the mass of the \(((\bs U_i(D_i(m))_{m=0}^{n-1})_{i=1}^N\). 
Since in the balance case the total mass of the urn is a deterministic function of the number of draws of the urn, knowing \(D_i(n)\)
only tells us the number of draws of \(\bs U_i(D_i(n))\), but gives no information on the outcome of those draws. Of course, this does not imply the summands in \eqref{eq: heuristic} are indepedent, since they depend on each other through the \((D_i(n))_{i=1}^N\). However, since our limits in Section \ref{Sec: main results} are Gaussian, this heuristic suggests that as \(n\) gets large the \(D_i(n)\) behave ``independently enough" such that we see a CLT in \eqref{eq: heuristic}, where the behaviour of the summands depends on which regime we are in.
\begin{remark}
\label{remark: nonbal heuristic}
The heuristic \eqref{eq: heuristic} suggests the non-balanced case could display completely different behaviour to the balanced case. Indeed, in the
non-balanced case, if a particular \(D_i(n)\) is large, it is likely that \(\bs U_i(D_i(n))\) drew colours that added the most possible
mass to the urn, since this gives the urn a higher probability of being drawn in the future. This additional dependence may be enough to 
dominate the CLT behaviour we believe we are seeing in the balanced case. 
\end{remark}
\noindent We now discuss each regime in detail.
\subsection{Heuristics of the IBD regime}
\label{Sec: Heuristics IBD}
Firstly, we discuss the renormalizing term in \eqref{eq: colour comp IBD 2}, being \(n^{-1}N\e^{AS^{-1}\ell_1(n,t)}\). By taking a Taylor expansion of the logarithm around 1 and using the exponential power series, this term can be written as
\begin{equation*}
  n^{-1/2}N\bigg(\bs \mu + nt (A/\beta_1)\bs \mu+ \sum_{k=2}^{r}\frac{(nt)^{k}}{N^{k-1}}\sum_{m=1}^k c_{km}(S,\beta_1) A^m\bs \mu+\mathcal{O}\F((n/N)^{r+1}\R)\bigg),
\end{equation*}
for some real constants \(c_{km}(S,\beta_1)\), and any \(r \geq 2\). The terms can be seen as a LLNs for the r.h.s.\ of \eqref{eq: heuristic}. In particular, for~\(\ell\geq 0\), the terms involving \(A^{r}\), \(0\leq r\leq \ell\), come from summing over urns in (4.1) that have seen at most \(\ell\) draws. To see why this is the case, notice that the number of urns that receive 0 draws in \eqref{eq: heuristic} by time step \(n\) will be of order \(N\). Therefore, throughout the urn process we pick
an urn that has had 0 draws with probability almost 1. So, by time step \(n\), we expect to see order \(n\) urns that have received 1 draw. Next, notice that if we have order \(n\)
urns with 1 draw, then the probability we pick an urn that has received 1 draw is of order \(n/N\). Therefore, by time step \(n\), we expect to see order~\(n^2/N\) urns that have received 2 draws.
One can repeat this heuristic to see by time \(n\) the number of urns to receive \(k\) draws is expected to be on the order of \(n^{k+1}/N^k\). Interestingly, since it is possible to scale \(n\) and \(N\) such that
\(n^{k+1}/N^k\rightarrow \infty\) for all \(k\geq 0\), there can exist infinitely many urns in \eqref{eq: heuristic} that receive infinitely many draws by time step \(n\) as \(n\rightarrow \infty\). Furthermore, since this implies \(n^{k+1/2}/N^k\rightarrow \infty\) for all \(k\geq 0\), the deterministic trend terms of these infinite urns will dominate the fluctuations of the urn.
 
 The Brownian fluctuations in \eqref{eq: colour comp IBD 2} suggest we are seeing a CLT in \eqref{eq: heuristic} over all urns that have received exactly 1 draw. Indeed, the first draw of each urn is an i.i.d.\ random variable with probability \(a_i \mu_i/\beta_1\) of choosing colour \(i\). This agrees with the random walk analogy presented in Remark \ref{remark: random walk}.
\subsection{Heuristics of the TR regime}
\label{Sec: Heuristics TR}
Recall that in the TR regime the colour composition converges to a deterministic process (the limit process of \eqref{eq: colour comp TR} renormalized so that the components of the vector sum to one). This value can be loosely interpreted as the average colour composition of a FIC urn after \(t\) draws, which started with 1 initial ball, and whose initial colour composition and replacement structure is identical to the GIC urn. (Of course, this does not make formal sense, since the number of steps and initial composition would be fractional.) Heuristically, this follows since we are seeing a LLNs in \eqref{eq: heuristic}, where each of the \((\bs U_i)_{i=1}^d\) ``on average" have had \(t\) draws by time step \(\lfloor nt\rfloor\). 

The heuristic for the Gaussian fluctuations seen in \eqref{eq: TR} is similar. Now we are seeing a CLT in \eqref{eq: heuristic}, where again, ``on average" each urn has had \(t\) draws by time step \(\lfloor nt\rfloor\). For this heuristic to be true, the covariance function of the Gaussian process in \eqref{eq: TR} should be controlled by the variances of the urn at finite numbers of draws. Indeed this does occur, for example take the identity replacement structure: As shown by Borovkov \cite{Borov}, the fluctuations in \eqref{eq: TR} are the sum of an independent Gaussian random variable and Brownian motion. Also, it is well known (see Borovkov \cite{Borov}), that conditionally on the Dirichlet colour composition limit (Recall \(\bs V\) from Section \ref{Section: Introduction}: Identity urn), the draws of the identity urn are i.i.d.\ with probability \(V_i\) of drawing colour \(i\). Therefore, each urn in \eqref{eq: heuristic} has fluctuations coming from the randomness of the Dirichlet limit which appears as the Gaussian random variable, and the randomness of the ``on average" \(t\) i.i.d.\ draws of the urn which appears as the Brownian motion. Unlike the identity urn, for a general replacement structure, the urns local behaviour between draws is more complex and can significantly vary depending on the replacement structure chosen. Hence, unlike the identity urn, the Gaussian process in \eqref{eq: TR} does not have a nice general form.

\subsection{Heuristics of the TSD regime}
\label{Sec Heuristics TSD}
In the TSD regime, the urns in \eqref{eq: heuristic} have had ``on average" infinitely many draws as \(n\rightarrow \infty\). Therefore, it is of no surprise that the behaviour in the TSD regime is similar to the behaviour of the FIC urn. The intuitive reason behind the behaviour seen in both the FIC case and the TSD regime is as follows:
\\~\\
\textbf{Case of \(S\) simple:} First assume that the urn is irreducible. In this case, the colour composition converges in both the FIC case and the TSD regime to the right Perron-Frobenius eigenvector of \(A\), \(\bs v_1\), which is independent of the initial composition. This follows since even if one colour is lucky enough to dominate early, the irreducibility condition forces balls of all other colours to still be added through draws of the dominating colour. This pulls any colour composition towards the stable state \(\bs v_1\). 

If the urn is not irreducible, then by the Jordan normal form of a matrix, we can rearrange the colours such that \(A\) is block-lower triangular with irreducible matrices on the diagonal. On average, the colours corresponding to the irreducible matrix with eigenvalue \(S\) add more mass (balls if \(\bs a=1\)) to the urn than any of the other colours. This leads to them dominating in the limit no matter the early draws of the urn. Then, similarly to the irreducible case, this dominating set of irreducible colours pushes any colour composition towards \(\bs v_1\). Importantly, the only colours seen in the limit are these dominating colours, and any colours that can be added through the dominating set (i.e.\ the colour can appear in an urn with initial composition only containing the dominating colours). 

For the fluctuations, recall that we have three different cases depending on whether the real part \(\lambda_2\) of the eigenvalue(s) of~\(A\) with second largest real part is smaller, equal to, or larger than \(S/2\). Furthermore, when \(\lambda_2>S/2\), the fluctuations depend on the initial colour composition. Whereas, when \(\lambda_2\leq S/2\), the fluctuations depend on the asymptotic colour composition. The intuitive reason as to why we see this behaviour follows from a heuristic on the relative spectral gap~\((S-\lambda_2)/S\). By our previous discussion on the asymptotic colour composition, we know the urn wants to pull any colour composition towards~\(\bs v_1\). The relative spectral gap can be seen as the rate at which the urn is pulled towards~\(\bs v_1\), with a larger relative spectral gap corresponding to a faster rate. (Indeed, we saw this rate appear as the rate at which the TR colour composition converged to the TSD colour composition in Remark \ref{remark: TR to TSD}.) 

When \(\lambda_2>S/2\), the relative spectral gap is too small leading to the urn being pulled towards \(\bs v_1\) too slowly. This gives enough time for fluctuations from early draws of the urn to build up and dominate. Heuristically, we have that the rate at which the fluctuations diverge is larger than the rate at which the fluctuations are pulled towards \(\bs v_1\). This is why we see a non-Gaussian random variable dependent on the initial composition in the FIC case, and in the GIC case this random variable becomes Gaussian due to the CLT occurring in \eqref{eq: heuristic}. 

When \(\lambda_2=S/2\), the urn is pulled towards \(\bs v_1\) just fast enough such that the fluctuations no longer depend on the initial composition. Heuristically, we have that the rate at which the fluctuations diverge is the same as the rate at which the fluctuations are pulled towards \(\bs v_1\). Thus, given enough time, the fluctuations from the early draws of the process are pulled close enough to \(\bs v_1\) and are forgotten. This is akin to starting a Brownian motion at a large finite value - since you are equally likely to increase or decrease, you will inevitably hit 0 and forget your large initial condition. (It is now of no surprise that Brownian fluctuations are possible as shown in Remark \ref{remark: BM}.) Once the fluctuations are ``close enough" to~\(\bs v_1\), we enter an asymptotic regime. Here, the fluctuations occur due to the fact that although the urn is pulled towards \(\bs v_1\), this pull is random since the urn process is random. This leads to the creation of new fluctuations around \(\bs v_1\). Similarly to the fluctuations from the early draws of the urn, these new fluctuations attempt to diverge but are pulled back to \(\bs v_1\) at the same rate. Eventually, the new fluctuations get too close to \(\bs v_1\) and are forgotten after which the process repeats. 

 Lastly, when \(\lambda_2<S/2\), the rate at which the fluctuations diverge is now smaller than the rate at which the fluctuations are pulled towards \(\bs v_1\). This is why we see fluctuations over a much smaller timescale than in the \(\lambda_2=S/2\) case. These fluctuations have no chance to diverge, being pulled towards \(\bs v_1\) quickly and forgotten. This is why we see an Ornstein-Uhlenbeck process in the limit. Indeed, an Ornstein-Uhlenbeck process can be seen as a modification of a Brownian motion with mean-reverting behaviour towards some mean state, where the pull towards the mean state is larger the further the process strays from the mean state. Here, \(\bs v_1\) takes the role of our mean state in the Ornstein-Uhlenbeck process. Furthermore, a simple calculation shows that the correlation coefficient of the limit \(\bs Y_s(\log(\cdot))\) (see Theorem \ref{theorem: main results discrete 2} TSD\textsubscript{s}) at two time points \(0<t_1\leq t_2 <\infty\) decays like \((t_2/t_1)^{(\lambda_2-S/2)/S}\). This agrees with the idea that the larger the relative spectral gap, the stronger the pull to \(\bs v_1\), so the quicker we forget past fluctuations. 
\\~\\
\textbf{Case of \(S\) non-simple:} First consider the FIC identity urn. Recall for this urn, the colour composition converges a.s.\ to a Dirichlet random variable. Heuristically, this is because of the following: Since each ball only adds balls of its own colour, it is likely that colours that currently dominate the urn will also dominate the urn in the future. Also, when the urn has
few number of balls, any draw will result in a significant change in the colour composition. These imply that the first few draws
heavily decide asymptotic colour composition, while the later draw-by-draw fluctuations in the colour composition are
less significant and smooth out. Hence, we see a random limit dependent on the initial composition. In the GIC case, since the initial number of balls is already so large, the early draws of the urn have significantly less impact, hence we see the first moment of the Dirichlet random variable in the limit. For non-identity urns, the heuristic is identical, but now the random limit seen in the FIC case is not necessarily Dirichlet with mean given by the initial composition. To see this, take the Jordan normal form of \(A\). We have two or more irreducible matrices with eigenvalue \(S\) on the diagonal. Each one of these matrices (in particular the colours corresponding to the matrix) is analogous to one colour in the identity urn, where now the early draws of the urn dictate the proportion of balls each one of these irreducible matrices receive. The simplest non-identity example is if \(A\) is block-diagonal with two irreducible blocks. The the number of total balls across block~1 and block 2 evolves like an identity urn with two colours, and the colour composition within each block converges to the Perron-Frobenius right-eigenvector of the irreducible matrix of that block.

The fluctuations seen in the case of \(S\) non-simple are due to random fluctuations in the early draws of the process: Although the large initial composition stops the random behaviour seen at the level of the colour composition, the initial draws still have a significant enough impact to create the dominant fluctuations of the urn. 
\section{Preliminaries for the proofs of Theorems \ref{Theorem: Main continuous time} and \ref{theorem: cont time 2}.}
\label{Sec: Prelim for proofs}
In this section, we give a sketch proof of Theorem \ref{Theorem: Main continuous time} to introduce tools needed for the proofs of Theorems \ref{Theorem: Main continuous time} and \ref{theorem: cont time 2}. The functional convergence in Theorem \ref{Theorem: Main continuous time} will be shown using one of two methods: The first method uses naturally arising martingales in the MCBP setting along with a martingale functional CLT introduced in Janson \cite{Janson}. This approach is taken for all but the LT\textsubscript{s} case. In the LT\textsubscript{s} case, the martingale approach is not applicable due to the timescale of the fluctuations; instead, we show convergence of the finite dimensional distributions using a CLT under Lyapunov conditions along with showing tightness conditions for the process. The martingale method is sketched in Section \ref{Sec: Prelim mart}, and the other method is sketched in Section \ref{Sec: Prelim class}. Lastly, in Section \ref{Sec: prelim moments}, we introduce all necessary MCBP moment results needed for the assumptions of these methods. 
\subsection{Sketch proof of Theorem \ref{Theorem: Main continuous time}: ST, CT, LT\textsubscript{c}, and LT\textsubscript{\(\ell\)} (martingale functional CLT)}
\label{Sec: Prelim mart}
Firstly, we introduce the martingale satisfied by the MCBP.
\begin{lemma}
\label{prop: martingale bp}
Let \((\bs X(t))_{t\geq 0}\) be a MCBP with initial condition \(\bs X(0)\) and replacement structure \((\bs a, (\bs \xi_i)_{i=1}^d)\) such that (A2) holds and, for \(1\leq i,j\leq d\), \(\E[\xi_{ij}]<\infty\). Let \(\mathcal{F}_t=\sigma(\{\bs X(s):0\leq s \leq t\})\). Then, \((\bs Y(t))_{t\geq 0} = (\e^{-At}\bs X(t))_{t\geq 0}\) is a martingale with respect to \(\mathcal{F}_t\), and, for all \(t \geq 0\), 
\begin{equation}
\label{eq: CMBP First moment}
    \E[\bs X(t)]=\mathrm{e}^{At}\bs X(0).
\end{equation}
\end{lemma}
\begin{proof}
 This is a well known result. Equation \eqref{eq: CMBP First moment} is given in Chapter V Section 7 of \cite{Athreya}. The martingale property follows as a consequence of the strong Markov property \eqref{eq: strong markov prop}, and \eqref{eq: CMBP First moment}:
\begin{equation*}
    \E[\mathrm{e}^{-At}\bs X(t)|\mathcal{F}_s]=\E[\mathrm{e}^{-At}\widehat{\bs X}(t-s)|\mathcal{F}_s]=\mathrm{e}^{-As}\bs X(s),
\end{equation*}
where \(\widehat{\bs X}\) is an MCBP with replacement structure \((\bs a, (\bs \xi_i)_{i=1}^d)\) that only depends on \(\mathcal{F}_s\) through its initial condition~\(\bs X(s)\).
\end{proof}
\noindent 
From Lemma \ref{prop: martingale bp}, we can construct martingale sequences whose limits imply the limits of Theorem \ref{Theorem: Main continuous time}. For example, take the CT case in Theorem \ref{Theorem: Main continuous time}, by Lemma \ref{prop: martingale bp}, we have that, for each \(n\geq 1\), the process
\begin{equation}
    \bs Y_n(t) = N^{-1/2}(\e^{-At}\bs X_n(t) -N\bs\mu), \quad t\geq 0 \label{eq: critical time mart}
\end{equation}
is a martingale. If we can show a martingale functional CLT for this sequence of martingales, the limit of Theorem \ref{Theorem: Main continuous time} follows by multiplying by \(\e^{At}\). The other cases follow a similar idea. We now introduce the martingale functional CLT that will be used. To do so, we start by introducing relevant martingale theory. 

Let \(f: \mathbb{R}_{\geq 0} \rightarrow \mathbb{C}^d\) be a path on \(\mathbb{C}^d\). For \(T>0\), define a partition of size \(n\) on \([0,T]\) to be a real \(n\)-tuple \((x_1,...,x_{n})\), such that \(x_1 = 0\), \(x_{n}=T\), and \(x_1<x_2<...<x_n\). Let \(\Pi_n\) be the set of all partitions of size \(n\) on \([0,T]\); let \(\Pi=\bigcup_{n=2}^{\infty}\Pi_n\). We say that \(f\) has bounded variation on the time interval \([0,T]\) if
\begin{equation}
\label{eq:finite var}
   \sup_{n\geq 2} \sup_{ (x_0,...,x_{n}) \in \Pi}\sum_{i=0}^{n-1}\|f(x_{i+1})-f(x_i)\|_2 < \infty.
\end{equation}
The path \(f\) is of finite variation, if, for all \(T \geq 0\), \(f\) has bounded variation on \([0,T]\). A stochastic process \((\bs Y(t))_{t\geq 0}\) taking values in \(\mathbb{C}^d\) is said to have finite variation if the process has path-wise finite variation a.s.\ (there exists a set \(A\), such that~\(\mathbb{P}(A)=1\), and for all \(\omega \in A\), \(\bs Y(\cdot,\omega)\) is of finite variation). For any finite-variation stochastic processes \(\bs Y\) and \(\bs Z\), we define their quadratic covariation as
\begin{equation}
\label{eq:quadratic variation}
[\bs Y,\bs Z]_t = \sum_{0\leq s \leq t} \Delta \bs Y(s) \Delta \bs Z(s)', \quad t\geq 0,   
\end{equation}
where \(\Delta \bs Y(s) = \bs Y(s)-\bs Y(s-)\). The quadratic variation of \(\bs Y\) at time \(t \geq 0\) is defined as \([\bs Y,\overline{\bs Y}]_t\). The r.h.s.\ of (\ref{eq:quadratic variation}) is formally uncountable, however since we will only work with processes that have a countable number of jumps, the r.h.s.\ will have at most a countable number of non-zero entries. Suppose further that \(\bs Y\) and \(\bs Z\) are both martingales, then, for~\(t \geq 0\),
\begin{equation}
\label{eq: second moment martingale}
\E[\bs Y(t) \bs Z(t)'] = \E [[\bs Y,\bs Z]_t] + \E[\bs Y(0) \bs Z(0)'].    
\end{equation}
Both \eqref{eq:quadratic variation} and \eqref{eq: second moment martingale} extend to matrix-valued stochastic processes, where the process is of finite variation if (\ref{eq:finite var}) holds under the euclidean matrix norm (euclidean norm in \(d \times d\) dimensional space). Now we state the martingale functional CLT.
\begin{lemma}[Proposition 9.1 in \cite{Janson}]
\label{proposition: Janson 9.1}
For each \(n\geq 1\), let \((\bs M_n(t))_{t\geq 0}\) be a \(d\)-dimensional complex martingale on \([0,\infty)\) with \(\bs M_n(0)=0\). Assume that \(\Sigma_1(t)\) and \(\Sigma_2(t)\), \(t \geq 0\), are (deterministic) continuous matrix-valued functions, such that, for every fixed \(t\geq 0\),
\begin{enumerate}
    \item\([\bs M_n,\overline{\bs M}_n]_t \coninprob \Sigma_1(t) \quad \text{as }n \rightarrow \infty.\)
    \item\([\bs M_n,\bs M_n]_t \coninprob \Sigma_2(t) \quad \text{as }n \rightarrow \infty.\)
    \item \(\sup_n \E[\|\bs M_n(t)\|_2^2]<\infty.\)
\end{enumerate}
Then as \(n \rightarrow \infty\), \(\bs M_n \conindis \bs M\) in \(D[0,\infty)\), where \(\bs M\) is a continuous \(d\)-dimensional mean-zero Gaussian process with covariance matrices, for \(t_1,t_2 \geq 0\),
\begin{align*}
   & \E[\bs M(t_1) \overline{\bs M}(t_2)']=\Sigma_1(t_1 \wedge t_2),\quad \E[\bs M(t_1)\bs M(t_2)']=\Sigma_2(t_1\wedge t_2).
\end{align*}
\end{lemma}
\noindent Recall the sequence of CT martingales defined in \eqref{eq: critical time mart}. To apply Lemma \ref{proposition: Janson 9.1} to this sequence we need good control over its quadratic variation. Since the MCBP only jumps when a particle dies, the quadratic variation of \eqref{eq: critical time mart} satisfies
\begin{equation}
  [\bs Y_n,\overline{\bs Y}_n]_t = \sum_{i=1}^d\sum_{k:\tau_{ik}\leq t}\e^{A\tau_{ik}}\Delta \bs X_n(\tau_{ik})\Delta \bs X_n(\tau_{ik})'\e^{A'\tau_{ik}}, \label{eq: quad var r 1}
\end{equation}
where \((\tau_{ik})_{k\geq 1}\) are the ordered death times of particles of type \(i\). Sums of this form were studied in Janson \cite{Janson} to prove the asymptotic behaviour of irreducible FIC urns; they show these sums satisfy a certain matrix-valued martingale. This is presented in the following lemma which is a simple extension of Lemma 9.3 in \cite{Janson}.
\begin{lemma}
\label{lemma: Jan 9.3}
Fix \(i \in \{1,...,d\}\) and \(p\) a positive integer. For every \(j \in \{1,...,p\}\), let \(M_{1j}(t)\) and \(M_{2j}(t)\) be continuous complex matrix-valued functions
defined on \([0, \infty)\), and let \(N_{j}(\bs x)\) be a complex matrix-valued function defined on \(\mathbb{R}^d\).
Suppose further that the dimensions of \(M_{1j}(t), N_j(\bs x)\), and \(M_{2j}(t)\) are such that the product \(\prod_{j=1}^pM_{1j}(t)N_j(\bs x)M_{2j}(t)\) is defined. Let \((\bs X(t))_{t\geq 0}\) be a MCBP with replacement structure \((\bs a,(\bs \xi_i)_{i=1}^d)\). Let \((\tau_{ik})_{k\geq 1}\) be the ordered death times of particles of type \(i\) in \(\bs X\). Assume \(\mathbb{E}\F[\prod_{j=1}^pM_{1j}(t)N_j(\bs \xi_i)M_{2j}(t)\R]< \infty\). Let
\begin{equation}
    Z(t) := \sum_{k:\tau_{ik}\leq t}\prod_{j=1}^pM_{1j}(\tau_{ik})N_j(\Delta\bs X(\tau_{ik}))M_{2j}(\tau_{ik}) - \int_{0}^t\mathbb{E}\bigg[\prod_{j=1}^pM_{1j}(v)N_j(\bs \xi_i)M_{2j}(v)\bigg] a_i X(v)_i \mathrm{d}v, \quad t \geq 0. \label{eq: improve 1}
\end{equation}
Then, \(Z\) is a (matrix-valued) martingale; in particular 
\begin{align*}
    \E\bigg[\sum_{k:\tau_{ik}\leq t}\prod_{j=1}^pM_{1j}(\tau_{ik})N_j(\Delta\bs X(\tau_{ik}))M_{2j}(\tau_{ik})\bigg] &=\int_{0}^t\mathbb{E}\bigg[\prod_{j=1}^pM_{1j}(v)N_j(\bs \xi_i)M_{2j}(v)\bigg] \E[ a_i X
    (v)_i] \mathrm{d}v\\
    &=\int_{0}^t\mathbb{E}\bigg[\prod_{j=1}^pM_{1j}(v)N_j(\bs \xi_i)M_{2j}(v)\bigg](a_i \mathrm{e}^{Av}\bs X(0))_i \mathrm{d}v.
\end{align*}
\end{lemma}
\noindent The proof of this lemma immediately follows from the proof of Lemma 9.3 in \cite{Janson}. This lemma should not be understated. Firstly, by \eqref{eq: quad var r 1}, it gives us immediate access to the first moment of the quadratic variation of MCBP related martingales. Therefore, one can already construct a suitable limit for Conditions 1 and 2 in Lemma \ref{proposition: Janson 9.1}. Secondly, since the second term on the r.h.s.\ of \eqref{eq: improve 1} is continuous and of finite variation, the quadratic variation of \eqref{eq: improve 1} is also a sum of the form to apply Lemma \ref{lemma: Jan 9.3}. Indeed, the jumps of \(Z\) are exactly the summands in \eqref{eq: improve 1}, therefore the quadratic variation of \(Z\) is a sum with these summands squared. This, together with \eqref{eq: second moment martingale}, gives us access to the second moment of the quadratic variation of MCBP related martingales. Showing convergence to 0 of this second moment is how we are going to show convergence of the quadratic variation in Lemma \ref{proposition: Janson 9.1}. 
\subsection{Sketch Proof of Theorem \ref{Theorem: Main continuous time}: LT\textsubscript{s}}
\label{Sec: Prelim class}
The aim is to apply the following functional convergence result to the sequence of processes in \eqref{eq: small component convergence}.
\begin{theorem}[Theorem 13.5 of Billingsley \cite{Bill}]
\label{Theorem:Bill con}
Let \((\bs Y_n)_{n\geq 1}\) be a sequence of processes in \(D[a,b]\), and \(\bs Y\) a process in~\(D[a,b]\). Suppose the following conditions hold:
\begin{enumerate}
\item  
For all increasing sequences of times, \(a\leq t_1 \leq ... \leq t_k \leq b\), \((\bs Y_n(t_1),...,\bs Y_n(t_k)) \conindis (\bs Y(t_1),...,\bs Y(t_k))\) as \(n \rightarrow \infty\).
\item \(\bs Y(b)-\bs Y(b-\delta) \conindis 0\)
as \(\delta \rightarrow 0\).
\item For every \(n \geq 1\), \(a \leq t_1 \leq t_2 \leq t_3 \leq b\), \begin{equation*}
    \E \F [\| \bs Y_n(t_2)-\bs Y_n(t_1)\|_2^2 \| \bs Y_n(t_3)-\bs Y_n(t_2)\|_2^2\R] \leq \cst (t_3-t_1)^{3/2}.
\end{equation*}
\end{enumerate}
Then,
\begin{equation*}
    \bs Y_n(t) \conindis \bs Y(t)\text{ in } D[a,b].
\end{equation*}
\end{theorem}
\noindent We focus on the sketch proof of Condition 1. (Condition 2 holds since our limiting process is Gaussian with continuous covariance function. Condition 3 can be shown without much difficulty if one knows the behaviour of the moments of the MCBP which are given in Section \ref{Sec: prelim moments}.)  Take \(\bs X_n\) as in \eqref{eq: small component convergence}, by the branching property \eqref{equ:branching property},
\begin{equation*}
    \bs X_n(\omega+ t) \eqindis \sum_{i=1}^d\sum_{j=1}^{N\mu_i}\bs X_{nij}(\omega+ t), \quad t\geq 0,
\end{equation*}
where the \(\bs X_{nij}\) are independent branching processes with initial conditions \(\bs e_i\) and identical replacement structure to \(\bs X_n\). Since \(N\rightarrow \infty\) as \(n\rightarrow \infty\), we can show Condition 1 for a single time point (k=1) by proving a CLT holds for this sum at a fixed time \(t\). When \(k\geq 1\), the proof is more subtle: It is no longer enough to show a CLT holds at each of the two fixed time points, since this does not give us any information on how they depend on each other. However, the strong Markov property comes to our rescue, since this property along with the branching property implies 
\begin{equation*}
    \bs X_n(\omega+t_1 +t)-\bs X_n(\omega+t_1) \eqindis \sum_{i=1}^d\sum_{j=1}^{ X_n(\omega+t_1)_i}(\hat{\bs X}_{nij}(t)-\bs e_i),\quad t\geq 0,
\end{equation*}
where the \(\hat{\bs X}_{nij}\) are independent branching processes with initial conditions \(\bs e_i\), identical replacement structure to \(\bs X_n\), and are independent of the MCBP up to time \(\omega+t_1\). If we can show in some respect (in probability will be enough) that~\(\bs X_n(\omega+t_1)\) tends to \(\infty\) with \(n\) and does not fluctuate too much, then one should be able to obtain a CLT for this sum for fixed \(t\) with limit independent of \(\bs X_n(\omega +t_1)\) (since the fluctuations in \(\bs X_n(\omega+t_1)\) will have a zero contribution to the sum in the limit). Hence, obtaining how the limits of the CLTs depend on each other. For this method to work, we need a CLT for sums with a random number of summands. This motivates the following extension of Lyapunov's CLT to allow for a random number of summands.
\begin{lemma}[Lyapunov's multivariate CLT with random number of summands]
\label{Lemma: MCLT}
Let \((k_n)_{n\geq 0}\) be a random positive integer-valued sequence. Let \(((\boldsymbol{Y}_{n,k})_{k=1}^{k_n})_{n\geq 0}\) be a triangular array of random variables in \(\mathbb{C}^d\) that satisfy row-wise independence, and are independent of the \(k_n\) (the \(k_n\) only determine the number of random variables not their outcome). Assume that, as~\(n\rightarrow \infty\):
\begin{enumerate}
\item \(k_n \coninprob \infty\).
\item 
There exist deterministic matrices \(\Sigma_1, \Sigma_2 \in \mathbb{C}^{d\times d}\), such that \[
    \sum_{k=1}^{k_n} \mathrm{Cov}(\boldsymbol{Y}_{n,k},\boldsymbol{Y}_{n,k})  \coninprob \Sigma_1,
\quad 
    \sum_{k=1}^{k_n} \mathrm{Cov}(\boldsymbol{Y}_{n,k},\overline{\boldsymbol{Y}}_{n,k}) \coninprob \Sigma_2.
\]
\item There exists \(p > 2\), such that
\(
\sum_{k=1}^{k_n}\mathbb{E}[\|\boldsymbol{Y}_{n,k}-\mathbb{E}[\boldsymbol{Y}_{n,k}]\|_2^p] \coninprob  0.
\)
\end{enumerate}
Then, for any \(\bs W\)-continuity set \(B \subseteq \mathbb{C}^d\), the row sums \(\boldsymbol{S}_n = \sum_{k=1}^{k_n} \boldsymbol{Y}_{n,k}-\mathbb{E}[\boldsymbol{Y}_{n,k}]\) satisfy
    \begin{equation*}
        \mathbb{P}(\bs S_n \in B | k_n) \coninprob \mathbb{P}(\bs W \in B)
    \end{equation*}
    as \(n\rightarrow \infty\), where \(\boldsymbol{W}\) is Gaussian with 
\begin{equation*}
    \Var(\bs W)=\Sigma_1,
    \quad \Cov (\bs W, \overline{\bs W}) = \Sigma_2.
\end{equation*}
(Recall \(B\) is a \(\bs W\)-continuity set if \(\mathbb{P}(\bs W \in \partial B)=0\), where \(\partial B\) is the boundary of \(B\).) 
\end{lemma}
\begin{proof}
The case of a deterministic \(k_n\) is the standard Lyapunov CLT and can be found in \cite{Kolmogorov}. We use this result to extend to the case of random \(k_n\). Our approach is to split the row sums \(\bs S_n\) into two components. We let
\begin{align*}
    &\bs S_n = \bs S_{n1}+\bs S_{n2},\\
    &\bs S_{n1}=\sum_{k=1}^{\ell_n}\bs Y_{n,k}-\E[\bs Y_{n,k}], \quad \bs S_{n2}=\bs S_n - \bs S_{n1},
\end{align*}
where the \((\ell_n)_{n\geq 1}\) is a positive non-random integer-valued sequence that tends to infinity as \(n\rightarrow \infty\) to be defined later. We are going to show that the version of Lyapunov's CLT with deterministic number of summands can be applied to the~\((\bs S_{n1})_{n\geq 1}\) giving 
\begin{equation}
    \bs S_{n1} \conindis \bs W \label{eq: lyapunov 1}
\end{equation}
as \(n\rightarrow \infty\). We will also show, for any \(\varepsilon>0\), 
\begin{equation}
    \mathbb{P}(\|\bs S_{n2}\|_2<\varepsilon | k_n) \coninprob 1 \label{eq: lyapunov 2}
\end{equation} 
as \(n\rightarrow \infty\). For now, assume that \eqref{eq: lyapunov 1} and \eqref{eq: lyapunov 2} both hold. The lemma can be shown using these two results as follows. Let \(B\subseteq \mathbb{C}^d\) be a \(\bs W\)-continuity set. For \(\varepsilon>0\), define the sets 
\begin{align*}
    &\partial B_{\varepsilon} :=\{\bs x \in \mathbb{R}^d: \|\bs x - \bs y\|_2 \leq \varepsilon \text{ for some } \bs y \in \partial B\}, \quad B^+_{\varepsilon}= B \cup \partial B_{\varepsilon} ,\quad B^-_{\varepsilon}= B\setminus\partial B_{\varepsilon}.
\end{align*}
The set \(B^+_{\varepsilon}\) is an extension of \(B\), where we include the closed balls of radius \(\varepsilon\) at each of the boundary points of \(B\). The set \(B^-_{\varepsilon}\) is a retraction, where we remove these closed balls from \(B\). This implies, a.s.\ for any \(\varepsilon>0\),
\begin{equation}
   \mathbb{P}(\bs S_{n1} \in B^{-}_{\varepsilon})-\mathbb{P}(\|\bs S_{n2}\|_2 \geq \varepsilon |k_n) \leq \mathbb{P}(\bs S_n \in B |k_n) \leq  \mathbb{P}(\bs S_{n1} \in B^{+}_{\varepsilon})+\mathbb{P}(\|\bs S_{n2}\|_2 \geq \varepsilon |k_n), \label{eq: lyapunov 3}
\end{equation}
where we have used that \(\bs S_{n1}\) is independent of \(k_n\) to drop the conditioning. By construction, the set \(B^+_{\varepsilon}\) is closed and \(B^-_{\varepsilon} \) is open. Therefore, by the portmanteau theorem,
\begin{align}
    &\limsup_n \mathbb{P}(\bs S_{n1} \in B^{+}_{\varepsilon})\leq \mathbb{P}(\bs W \in B^{+}_{\varepsilon}),\nonumber \\
    & \liminf_n \mathbb{P}(\bs S_{n1} \in B^{-}_{\varepsilon})\geq  \mathbb{P}(\bs W \in B^{-}_{\varepsilon}). \label{eq: lyapunov 4}
\end{align}
Let \(B^{\degree}\) and \(B^{\top}\) denote the interior and closure of \(B\) respectively. We have \(B^{\degree} =\bigcup_{i=1}^{\infty}B^{-}_{2^{-i}}\), and \(B^{\top}=\bigcap_{i=1}^{\infty }B^{+}_{2^{-i}}\). This, continuity of the probability measure, and the fact that \(\mathbb{P}(\bs W \in \partial B)=0\) imply that, as \(\varepsilon \rightarrow 0\),
\begin{align}
     & \mathbb{P}(\bs W \in B^{-}_{\varepsilon}) \rightarrow \mathbb{P}(\bs W \in B),\nonumber\\
      &\mathbb{P}(\bs W \in B^{+}_{\varepsilon}) \rightarrow \mathbb{P}(\bs W \in B). \label{eq: lyapunov 5}
\end{align}
Using \eqref{eq: lyapunov 4} and \eqref{eq: lyapunov 5} in \eqref{eq: lyapunov 3} implies the existence of an \(\eta:= \eta(\varepsilon,n)>0\), such that \(\eta \rightarrow 0\) as \(\varepsilon\rightarrow 0\) and \(n\rightarrow \infty\), and
\begin{equation*}
     |\mathbb{P}(\bs S_n \in B |k_n) - \mathbb{P}(\bs W \in B)|\leq \eta + \mathbb{P}(\|\bs S_{n2}\|_2\leq \varepsilon|k_n)
\end{equation*}The lemma follows by this and \eqref{eq: lyapunov 2}. We are left to show \eqref{eq: lyapunov 1} and \eqref{eq: lyapunov 2}. First we show how to construct the \((\ell_n)_{n\geq 1}\) such that the summands in \((\bs S_{n1})_{n\geq 1}\) satisfy the conditions of Lemma \ref{Lemma: MCLT} for a deterministic number of summands. For~\(q \in \mathbb{Z}_{\geq 0}\), let \(\delta_q = 2^{-q}\). Define the events
\begin{align*}
    &E_{n,q}^{(1)}=\{k_n>\delta_q^{-1}\}, &&E_{n,q}^{(2)} = \F\{\F\|\sum_{k=1}^{k_n}\mathrm{Cov}(\boldsymbol{Y}_{n,k},\boldsymbol{Y}_{n,k}) - \Sigma_1\R\|_2<\delta_q\R\},\\
    &E_{n,q}^{(3)}=\F\{\F\|\sum_{k=1}^{k_n}\mathrm{Cov}(\boldsymbol{Y}_{n,k},\overline{\boldsymbol{Y}}_{n,k}) - \Sigma_2\R\|_2<\delta_q\R\}, &&E_{n,q}^{(4)}=\F\{\sum_{k=1}^{k_n}\mathbb{E}[\|\boldsymbol{Y}_{n,k}-\mathbb{E}[\boldsymbol{Y}_{n,k}]\|_2^p]<\delta_q\R\}.
\end{align*}Since Assumptions 1, 2 and 3 hold in probability, there exists some positive, finite, non-decreasing sequence \((m_{q})_{q\geq 0}\), such that, for each \(q\geq 0\), and all \(n \geq m_{q}\),
\begin{equation}
    \mathbb{P}(E_{n,q}^{(i)}) \geq 1-\delta_q, \quad i \in \{1,2,3,4\}. \label{eq: CLT events}
\end{equation}
Let \(E_{n,q}=\cap_{i=1}^4 E_{n,q}^{(i)}\). Let \((C_{n})_{n\geq 1}:= (E_{n,q_n})_{n\geq 1}\), where \(q_n = \max \{q : m_{q}\leq n\}\). Then by \eqref{eq: CLT events}, \(\mathbb{P}(C_n)\geq 1-4\delta_{q_n}\). Furthermore, since, as \(n\rightarrow \infty\), \(q_n\rightarrow \infty\), we have \(\mathbb{P}(C_{n})\rightarrow 1\) as \(n\rightarrow \infty\). Let \(\ell_n := \min \{k \in \mathbb{Z}_{\geq 1} : C_n\text{ occurs }\}\), and take~\(\ell_n=1\) if this set is empty (this corresponds to \(\mathbb{P}(C_n)=0\)). By definition, \((\ell_n)_{n\geq 1}\) tends to infinity as \(n\rightarrow \infty\) (\(\ell_n\) is at least~\(\delta_{q_n}^{-1}\), and \(\mathbb{P}(C_n)=0\) can hold for only finitely many \(n\)). Furthermore, since \(\{k_n=\ell_n\} \in C_n\), and \(q_n\rightarrow \infty\) as \(n\rightarrow \infty\), the row sums \((\bs S_{n1})_{n\geq 1}\) satisfy Lyapunov's conditions with covariance matrices \(\Sigma_1\), and \(\Sigma_2\). This implies \eqref{eq: lyapunov 1}. For \eqref{eq: lyapunov 2}, first note on \(C_n\), we have that \( \bs S_{n2}=\sum_{k=\ell_n+1}^{k_n}\bs Y_{n,k}-\E[\bs Y_{n,k}]\), with
\begin{equation*}
    \sum_{k=\ell_n+1}^{k_n}\Cov(\bs Y_{n,k},\overline{\bs Y}_{n,k}) \leq 2 \delta_{q_n}\rightarrow 0
\end{equation*}
as \(n\rightarrow \infty\). This implies, for all \(\varepsilon,\eta>0\), there exists an \(n^*\), such that, for all \(n>n^*\),
\begin{equation*}
    \mathbb{P}(\mathbb{P}(\|\bs S_{n2}\|_2 >\varepsilon |k_n) <\eta |C_n) = 1.
\end{equation*}
This and the fact that \(\mathbb{P}(C_n)\rightarrow 1\) as \(n\rightarrow \infty\) give \eqref{eq: lyapunov 2}.
\end{proof}
\subsection{Results on the second and fourth moments of the MCBP}
\label{Sec: prelim moments}
Many of the results we are going to use to prove Theorem \ref{Theorem: Main continuous time} have a moment assumption (e.g.\ Lemma \ref{proposition: Janson 9.1} Condition 3, Theorem \ref{Theorem:Bill con} Condition 3, and Lemma \ref{Lemma: MCLT} Conditions 2 and 3). In this section, we introduce all MCBP moment results we need to show these assumptions hold.
\begin{lemma}
\label{lemma: nolabel}
Let \((\bs X(t))_{t\geq 0}\) be a MCBP with initial condition \(\bs X(0)\) and replacement structure \((\bs a, (\bs \xi_i)_{i=1}^d)\), such that (A2) holds, and, for \(1\leq i,j \leq d\), \(\E[\xi_{ij}]<\infty\). Let \(M_1\in~ \mathbb{C}^{d\times d}\). Then, for every \(t\geq 0\),
\begin{align}
    &\E \F [\bs X(t)\bs X(t)'\R] = \sum_{i=1}^d \int_{0}^t \mathrm{e}^{A(t-v)}\E[\bs \xi_i \bs \xi_i']\mathrm{e}^{A'(t-v)} (a_i\mathrm{e}^{Av}\bs X(0))_i\mathrm{d}v+\mathrm{e}^{At}\bs X(0) \bs X(0)'\mathrm{e}^{A't}, \label{eq: second moment 1}\\
    &\E \F [\|M_1\bs X(t)\|_2^2\R] = \sum_{i=1}^d \int_{0}^t \E[\|M_1\mathrm{e}^{A(t-v)}\bs \xi_i \|_2^2] (a_i\mathrm{e}^{Av}\bs X(0))_i\mathrm{d}v+\|M_1\mathrm{e}^{At}\bs X(0)\|_2^2  \label{eq: second moment 2},\\
    & \Var(\bs X(t)) = \sum_{i=1}^d \int_{0}^t \mathrm{e}^{A(t-v)}\E[\bs \xi_i \bs \xi_i']\mathrm{e}^{A'(t-v)} (a_i\mathrm{e}^{Av}\bs X(0))_i\mathrm{d}v. \label{eq: MCBP var}
\end{align}
\end{lemma}
\begin{proof}
Equation (\ref{eq: second moment 1}) is shown in Section 9 of Janson \cite{Janson}. For a matrix \(A\), let \(\mathrm{tr}(A)\) be the trace of \(A\). For any \(\bs x \in \mathbb{C}^d\), we have that \(\mathrm{tr}(\bs x \bs x^*) = \|\bs x\|_2^2\). Since the trace is linear, (\ref{eq: second moment 2}) follows by taking the trace of \(M_1 \E \F [\bs X(t)\bs X(t)'\R] M_1^*\) and using (\ref{eq: second moment 1}). Lastly, \eqref{eq: MCBP var} follows from \eqref{eq: CMBP First moment} and \eqref{eq: second moment 1}.
\end{proof}

\noindent The final tool we need before proving Theorems \ref{Theorem: Main continuous time} and \ref{theorem: cont time 2} is a bound on the fourth moment of the MCBP. This is addressed in the following lemma whose proof is given in Appendix \ref{Appendix: proof of fourth moments}.
\begin{lemma}
\label{Lemma: fourth moments}
Let \((\bs X(t))_{t\geq 0}\) be a MCBP with initial condition \(\bs X(0)\) and replacement structure \((\bs a, (\bs \xi_i)_{i=1}^d)\), such that (A1)-(A3) are satisfied. Let \(\Xi_4 = \sum_{i=1}^d\E[\|\bs \xi_i\|_2^4]\).
\begin{enumerate}
    \item[] \textbf{Small Components:} Let \(P_s = \sum_{\lambda \in \Lambda_s}\sum_{J \in \mathcal{J}_{\lambda}}P_{J}\). Then, for all \(t\geq 0\),
    \begin{align*}
     & \E \F[\|  (1+t)^{-(m_1-1)/2}\mathrm{e}^{-\lambda_1t/2}P_s \bs X(t)\|_2^4\R]\leq \cst (1+ \Xi_4^{3/2})\|\bs X(0)\|^4_2, \\
      &\E \F[\|(1+t)^{-(m_1-1)/2}\mathrm{e}^{-\lambda_1t/2}P_s (\bs X(t)-\mathrm{e}^{At}\bs X(0))\|_2^4\R]\leq \cst (1+ \Xi_4^{3/2}) \|\bs X(0)\|^2_2(1\wedge t ).
    \end{align*}
        \item[] \textbf{Critical Components:} Let \(\lambda \in \Lambda_c\), \(J \in \mathcal{J}_{\lambda}\) with size \(m\), \(1\leq \kappa \leq m\). Then, for all \(t\geq 0\),
    \begin{align*}
      &\E \F[\|  (1+t)^{-(m_1+2\kappa-2)/2}\mathrm{e}^{-\lambda_1t/2}N_A^{m-\kappa}P_J \bs X(t) \|_2^4\R]\leq \cst (1+ \Xi_4^{3/2})\|\bs X(0)\|^4_2, \\
      &  \E \F[\|(1+t)^{-(m_1+2\kappa-2)/2}\mathrm{e}^{-\lambda_1t/2}N_A^{m-\kappa}P_J (\bs X(t)- \mathrm{e}^{At}\bs X(0))\|_2^4\R]\leq \cst (1+ \Xi_4^{3/2}) \|\bs X(0)\|^2_2(1\wedge t ).
    \end{align*}
            \item[] \textbf{Large Components:} Let \(P_{\ell}=\sum_{\lambda \in \Lambda_{\ell}}\sum_{J \in \mathcal{J}_{\lambda}}P_J\). Then, for all \(t\geq 0\),
    \begin{align*}
       & \E \F[\|P_{\ell}\mathrm{e}^{-At} \bs X(t)\|_2^4\R]\leq \cst (1+ \Xi_4^{3/2})\|\bs X(0)\|^4_2, \\
        & \E \F[\|P_{\ell}(\mathrm{e}^{-At} \bs X(t)-\bs X(0))\|_2^4\R]\leq \cst (1+ \Xi_4^{3/2}) \|\bs X(0)\|^2_2(1\wedge t ).
    \end{align*}
\end{enumerate}
\end{lemma}
\noindent
\begin{remark}

In Lemma \ref{Lemma: fourth moments}, we have given no indication of how the bound depends on \(\bs a\) or \(d\). From our method of proof, it is not difficult to show that one can at least get a bound of \(\cst d^2 \|\bs a\|_2^2\). 
\end{remark}
\noindent
Equations \eqref{eq:identity} and (\ref{eq: projection}) give the following immediate corollary of Lemma~\ref{Lemma: fourth moments}. For all \(t\geq 0\),
\begin{align}
   &\E\F[\|(1+t)^{-(m_1-1)}\mathrm{e}^{-\lambda_1t}\bs X(t)\|_2^4\R]\leq \cst (1+ \Xi_4^{3/2}) \|\bs X(0)\|^4_2, \nonumber\\
   &\E \F[\|(1+t)^{-(m_1-1)}\mathrm{e}^{-\lambda_1t}(\bs X(t)-\mathrm{e}^{At}\bs X(0))\|_2^4\R]\leq \cst (1+ \Xi_4^{3/2})\|\bs X(0)\|^2_2(1\wedge t ). \label{eq: fourth moment main}
\end{align}
\section{Proofs of Theorems \ref{Theorem: Main continuous time} and \ref{theorem: cont time 2}}
\label{Sec: proof of cont time}
We now have all the tools needed to show Theorems \ref{Theorem: Main continuous time} and \ref{theorem: cont time 2}. We start with the proof of Theorem \ref{Theorem: Main continuous time}.
\subsection{Proof of Theorem \ref{Theorem: Main continuous time} ST}
Let 
\begin{equation*}
    \bs Z_n(t) = (N\varepsilon)^{-1/2}(\e^{-A\varepsilon t}\bs X_n(\varepsilon t)- N\bs \mu), \quad t \geq 0.
\end{equation*}
 This is a martingale for each fixed \(n\) by Lemma \ref{prop: martingale bp}. We are going to apply Lemma \ref{proposition: Janson 9.1} to this sequence of martingales with limit \(\bs W_1\). We start by showing Condition 2 (Condition 1 can be skipped since the martingales are real). By definition of~\(\bs W_1\), we must show, for each \(t\geq 0\),
\begin{equation}
[\bs Z_n,\bs Z_n]_t \coninprob t\sum_{i=1}^d \mu_i a_i \E[\bs \xi_i \bs \xi_i'] \label{eq: Zn quad var}
\end{equation}
as \(n\rightarrow \infty\). Since a time-independent shift does not change the quadratic variation, we have
\begin{equation}
    [\bs Z_n,\bs Z_n]_t = (N\varepsilon)^{-1}\sum_{i=1}^d\sum_{k:\tau_{nik}\leq \varepsilon t} \e^{-A\tau_{nik}}\Delta\bs X_n(\tau_{nik})\Delta \bs X_n(\tau_{nik})' \e^{-A'\tau_{nik}}, \label{eq: Zn quad var 2}
\end{equation}
where the \((\tau_{nik})_{k\geq 1}\) are the ordered death times of particles of type \(i\) in \(\bs X_n\). If, in Lemma \ref{lemma: Jan 9.3}, we take \(p=1\), \(M_1(t)~=~M_2(t)'=\e^{-At}\), and \(N(\bs x) = \bs x \bs x'\), then the first term on the r.h.s.\ of \eqref{eq: improve 1} is equal to
\begin{equation*}
   (N\varepsilon)^{-1} \sum_{k:\tau_{nik}\leq \varepsilon t} \e^{-A\tau_{nik}}\Delta\bs X_n(\tau_{nik})\Delta \bs X_n(\tau_{nik})' \e^{-A'\tau_{nik}}.
\end{equation*}
This, Lemma \ref{lemma: Jan 9.3}, and \eqref{eq: Zn quad var 2} imply
\begin{equation}
    Y_n(t) =  [\bs Z_n,\bs Z_n]_t - (N\varepsilon)^{-1}\sum_{i=1}^d\int^{\varepsilon t}_0 \e^{-Av}\E[\bs \xi_i\bs \xi_i']\e^{-A'v}a_i X_n(v)_i \mathrm{d}v, \quad t\geq 0 \label{eq: Zn quad var 3}
\end{equation}
is a matrix-valued martingale for each \(n\). We prove \eqref{eq: Zn quad var} by showing, for each \(t\geq 0\),
\begin{align}
   &Y_n(t) \coninprob 0, \label{eq: Yn component}\\
   &U_n(t):=(N\varepsilon)^{-1}\sum_{i=1}^d\int^{\varepsilon t}_0 \e^{-Av}\E[\bs \xi_i\bs \xi_i']\e^{-A'v}a_i( X_n(v)_i-N(\e^{A v}\bs \mu)_i) \mathrm{d}v \coninprob 0, \label{eq: Un component}\\
   &\varepsilon^{-1}\sum_{i=1}^d\int^{\varepsilon t}_0 \e^{-Av}\E[\bs \xi_i\bs \xi_i']\e^{-A'v}(a_i\e^{A v}\bs \mu)_i \mathrm{d}v \rightarrow t\sum_{i=1}^d \mu_i a_i \E[\bs \xi_i \bs \xi_i'], \label{eq: other component}
\end{align}
as \(n\rightarrow \infty\). For the rest of the proof, fix \(t\geq 0\). We start by showing \eqref{eq: Yn component}. Since the second term of \(Y_n\) does not contribute to the quadratic variation (it is continuous and of finite variation), we have
\begin{equation*}
    [Y_n,Y_n]_t=\F[[\bs Z_n,\bs Z_n],[\bs Z_n,\bs Z_n]\R]_t =(N\varepsilon)^{-2}\sum_{i=1}^d\sum_{k:\tau_{nik}\leq \varepsilon t} (\e^{-A\tau_{nik}}\Delta\bs X_n(\tau_{nik})\Delta \bs X_n(\tau_{nik})' \e^{-A'\tau_{nik}})^2.
\end{equation*}
This, and Lemma \ref{lemma: Jan 9.3} imply
\begin{equation}
    \E [Y_n(t)Y_n(t)'] =  \E [Y_n,Y_n]_t =(N\varepsilon)^{-2}\sum_{i=1}^d\int^{\varepsilon t}_0 \F(\e^{-Av}\E[\bs \xi_i\bs \xi_i']\e^{-A'v}\R)^2a_i N(\e^{Av}\bs \mu)_i \mathrm{d}v, \label{eq: Yn component 1}
\end{equation}
where the first equality follows from \eqref{eq: second moment martingale} since \(Y_n\) is started from 0. Let \(\varepsilon_{\mathrm{max}}=\max \{\varepsilon(n): n\geq 1\}\) (this exists and is finite since \(\varepsilon \rightarrow 0\) as \(n\rightarrow \infty\)). By (A1), we have
\begin{equation*}
 (N\varepsilon)^{-2}\sum_{i=1}^d\int^{\varepsilon t}_0 \F\|\F(\e^{-Av}\E[\bs \xi_i\bs \xi_i']\e^{-A'v}\R)^2\R\|a_i N(\e^{Av}\bs \mu)_i \mathrm{d}v   \leq \cst (N\varepsilon)^{-2}\int^{\varepsilon t}_0 N \leq \cst (N\varepsilon)^{-1}\rightarrow 0
\end{equation*}
as \(n\rightarrow \infty\), where we have uniformly bounded the integrand by its maximum value in \([0,\varepsilon_{\mathrm{max}}t]\). This and \eqref{eq: Yn component 1} imply the second moment of \(\|Y_n(t)\|_2\) tends to \(0\) as \(n\rightarrow \infty\), so \eqref{eq: Yn component} holds. Next, we show \eqref{eq: Un component}. Firstly, by (A1),
\begin{equation*}
    \E[\|U_n(t)\|_2] \leq \cst (N\varepsilon)^{-1}\sum_{i=1}^d\int^{\varepsilon t}_0 \E[| X_n(v)_i-N(\e^{A v}\bs \mu)_i|] \mathrm{d}v,
\end{equation*}
where we have uniformly bound the components of the integrand not present on the r.h.s.\ by their maximum value in~\([0,\varepsilon_{max}t]\). By \eqref{eq: CMBP First moment} and Jensen's inequality, the integrand can be bounded by \(\sum_{i=1}^d \Var(X_n(v)_i)^{1/2}\). By \eqref{eq: MCBP var} and~(A1), this can be bounded by~\(N^{1/2}\cst\), where the bound follows by taking the maximum over all \(v\in [0,\varepsilon_{max}t]\). Therefore,
\begin{equation*}
    \E[\|U_n(t)\|_2] \leq \cst N^{1/2} \rightarrow 0
\end{equation*}
as \(n\rightarrow \infty\). This implies \eqref{eq: Un component}. Lastly, one sees \eqref{eq: other component} holds by the fact that, since \(\varepsilon \rightarrow 0\) as \(n\rightarrow \infty\), the exponential terms in the integrand all tend to 1 uniformly in \(v \in [0,\varepsilon t]\) as \(n\rightarrow \infty\). Thus Condition 2 holds. For Condition~3, by \eqref{eq: second moment 2} and \eqref{eq: CMBP First moment}, we have
\begin{equation*}
    \E[\|\bs Z_n(t)\|_2^2] = \varepsilon^{-1} \sum_{i=1}^d \int_{0}^{\varepsilon t} \E[\|\e^{A(t-v)}\bs \xi_i\|_2^2]a_i(\e^{Av}\bs \mu)_i \mathrm{d}v.
\end{equation*}
The r.h.s.\ is uniformly bounded by a constant for all \(n\), which is seen by bounding the integrand by its maximum in~\([0,\varepsilon_{\mathrm{max}t}]\). Therefore, Condition 3 of Lemma \ref{proposition: Janson 9.1} holds. Thus, Lemma \ref{proposition: Janson 9.1} implies
\begin{equation*}
    \bs Z_n(t) \conindis \bs W_1(t) \text{ in } D[0,\infty)
\end{equation*}as \(n\rightarrow \infty\) which is Theorem~\ref{Theorem: Main continuous time}~ST. 
\subsection{Proof of Theorem \ref{Theorem: Main continuous time} CT}
This proof is similar to the ST case and can be found in Appendix \ref{Appendix: proof of CT}. 
\subsection{Proof of Theorem \ref{Theorem: Main continuous time} LT\textsubscript{s}}
We first give the proof of \eqref{eq: small component convergence}, and then the proof of \eqref{eq: small component convergence in probability}.
\begin{proof}[Proof of Theorem \ref{Theorem: Main continuous time} LT\textsubscript{s} \eqref{eq: small component convergence}] Set
\begin{align*}
   & \bs Z_{n}(t) = N^{-1/2}\omega^{-(m_1-1)/2}\mathrm{e}^{-\lambda_1(\omega+t)/2}P_s( \bs X_n(\omega+t)-N\mathrm{e}^{A(\omega+t)}\bs \mu), \quad t\geq -\omega,
\end{align*}
and
\begin{equation*}
    \mathcal{F}_{t}^{(n)}=\sigma\{\bs X_n(r):0\leq r\leq \omega+t\}.
\end{equation*}
Fix an arbitrary \(T>0\). In the following, \(\cst\) can (and often will) depend on \(T\). We are going to apply Theorem \ref{Theorem:Bill con} to the sequence \(\bs Z_{n}\) with limit \(\bs W_{s}\) on the interval \([-T,T]\). (We assume \(n\) is large enough such that \(\omega > T\), so \(\bs Z_n(-T)\) is well defined.)
\\~\\
\textbf{Condition 1 of Theorem \ref{Theorem:Bill con}:} For Condition 1, we start by showing that the \(\bs Z_n(t)\) converge for some arbitrary time~\(-T\leq t\leq T\) (i.e.\ \(k=1\)), then extend to showing \(\bs Z_n\) converges jointly at two fixed time points (\(k=2\)), and lastly extend to a general \(k\) time points. Once we have the result for \(k\) time points, we show the random variables in the limit are equal in distribution to \(\bs W_s\) taken at the same \(k\) time points.
\\~\\
\underline{Convergence of \(\bs Z_n\) for \(k=1\):} Fix \(-T\leq t \leq T\). By the branching property \eqref{equ:branching property}, we have
\begin{align}
&\bs Z_{n}(t) \eqindis N^{-1/2} \sum_{i=1}^d\sum_{j=1}^{N\mu_i} \omega^{-(m_1-1)/2}\mathrm{e}^{-\lambda_1(\omega+t)/2} P_s(\bs X_{ijn}(\omega +t)-\mathrm{e}^{A(\omega+ t)}\bs e_i), \label{eq: apply branching prop ls} 
\end{align}
where the \(\bs X_{ijn}\) are independent MCBPs with initial conditions \(\bs e_i\), and identical replacement structures \((\bs a, (\bs \xi_i)_{i=1}^d)\). Therefore, by \eqref{eq: CMBP First moment}, the summands in the r.h.s.\ are a triangular array of mean-zero random variables index by \(n\). We are going to apply Lemma \ref{Lemma: MCLT} to this array. Condition 1 clearly holds so we start by showing Condition 2. Recall \(\bs v(\bs \mu)\) \eqref{eq: imp ob 1}. We are going to show 
\begin{align}
& N^{-1} \sum_{i=1}^d\sum_{j=1}^{N\mu_i} \Var(\omega^{-(m_1-1)/2}\mathrm{e}^{-\lambda_1(\omega+t)/2} P_s\bs X_{ijn}(\omega +t)) \rightarrow \sum_{k=1}^d\int_{0}^{\infty}P_s\mathrm{e}^{Av}a_kv(\bs\mu)_k\E[\bs \xi_k \bs \xi_k']\e^{A'v}P_s'\mathrm{e}^{-\lambda_1 v}\mathrm{d}v \label{eq: eq}
\end{align}
as \(n\rightarrow \infty\). We do this by showing each summand in \eqref{eq: apply branching prop ls} converges to the r.h.s.\ of \eqref{eq: eq} with \(\bs v(\bs \mu)\) replaced by \(\bs v(\bs e_i)\) and using that
\begin{equation*}
    \bs v(\bs\mu) = \sum_{i=}^d \mu_i \bs v(\bs e_i).
\end{equation*}By \eqref{eq: MCBP var}, for all \(i=1,...,d\), and \(j=1,...,N\mu_i\), we have
\begin{align}
   &\Var(\omega^{-(m_1-1)/2}\e^{-\lambda_1(\omega+t)/2}P_s\bs X_{ijn}(\omega+ t))\nonumber\\
   &\qquad\qquad\qquad\qquad\qquad= \omega^{-(m_1-1)}\sum_{k=1}^d \int_{0}^{\omega+t} P_s\mathrm{e}^{A(\omega+ t-u)}\E[\bs \xi_k \bs \xi_k']\mathrm{e}^{A'(\omega + t-u)}P_s' a_k(\mathrm{e}^{Au}\bs e_i)_k\e^{-\lambda_1(\omega+t)}\mathrm{d}u. \label{eq: small time variance ls}
\end{align}
The change of variables \(\omega +t-u =v \) in \eqref{eq: small time variance ls} gives 
\begin{equation}
 \Var(\omega^{-(m_1-1)/2}\e^{-\lambda_1(\omega+t)/2}P_s\bs X_{ijn}(\omega+ t))= \omega^{-(m_1-1)}\sum_{k=1}^d \int_{0}^{\omega+t} P_s\mathrm{e}^{Av}\E[\bs \xi_k \bs \xi_k']\mathrm{e}^{A'v}P_s' a_k(\mathrm{e}^{A(\omega+t-v)}\bs e_i)_k\e^{-\lambda_1(\omega+t)}\mathrm{d}v. \label{th1.1 ls 1}
\end{equation}
For \(1\leq i \leq d\), and \(t\geq 0\), let 
\begin{equation}
  \bs r_i(t) := \e^{-\lambda_1t}t^{-(m_1-1)}\mathrm{e}^{At}\bs e_i - \bs v(\bs e_i). \label{th1.1 ls 2}
\end{equation}
  By \eqref{eq: matrix exp bounds}, \(\bs r_i(t)\rightarrow 0\) as \(t\rightarrow \infty\). Using this definition in \eqref{th1.1 ls 1} gives
  \begin{align}
      &\Var(\omega^{-(m_1-1)/2}\e^{-\lambda_1(\omega+t)/2}P_s\bs X_{ijn}(\omega+ t))\nonumber\\
      &\qquad\qquad\qquad\qquad\qquad\qquad=\sum_{k=1}^d \int_{0}^{\omega+t} \frac{(\omega+t-v)^{m_1-1}}{\omega^{m_1-1}}P_s\mathrm{e}^{Av}\E[\bs \xi_k \bs \xi_k']\mathrm{e}^{A'v}P_s' a_k(\bs v(\bs e_i)+\bs r_i(\omega+t-v))_k\e^{-\lambda_1v}\mathrm{d}v \label{eq: Var bound}
  \end{align}We show this integral converges as \(n\rightarrow \infty\) using the following two results: Let \(\hat\lambda\) be the real part of the eigenvalue with largest real part in \(\Lambda_s\). Then, by definition, \(2\hat\lambda < \lambda_1\). Therefore, by \eqref{eq: matrix exp bounds}, there exists some \(\eta >0\), such that, for all \(t\geq 0\),
\begin{equation}
   \|\e^{-\lambda_1 v}P_s\mathrm{e}^{Av}\E[\bs \xi_k \bs \xi_k']\mathrm{e}^{A'v}\|_2\leq \cst \e^{-\eta v}. \label{eq: improve 17}
\end{equation}
Thus, for all \(n\), the ``tail" of the integral in \eqref{eq: Var bound} satisfies
\begin{equation}
    \sum_{k=1}^d \int_{C}^{\infty} \frac{(\omega+t-v)^{m_1-1}}{\omega^{m_1-1}}P_s\mathrm{e}^{Av}\E[\bs \xi_k \bs \xi_k']\mathrm{e}^{A'v}P_s' a_k(\bs v(\bs e_i)+\bs r_i(\omega+t-v))_k\e^{-\lambda_1v}\mathrm{d}v\leq \cst \int_{C}^{\infty}\e^{-\eta v}\mathrm{d}v \rightarrow 0 \label{eq: integral tail}
\end{equation}
as \(C\rightarrow \infty\). For the integral's ``head", notice that for any fixed \(M>0\), both
\begin{equation*}
    \frac{(\omega+t-v)^{m_1-1}}{\omega^{m_1-1}} \rightarrow 1, \text{ and } \bs r_i(\omega+t-v)\rightarrow 0
\end{equation*}
as \(n\rightarrow \infty\), uniformly for \(v\in [0,M]\). This and \eqref{eq: integral tail} imply
\begin{equation*}
    \Var(\omega^{-(m_1-1)/2}\e^{-\lambda_1(\omega+t)/2}P_s\bs X_{ijn}(\omega+ t)) \rightarrow  \sum_{k=1}^d\int_{0}^{\infty}P_s\mathrm{e}^{Av}a_kv(\bs e_i)_k\E[\bs \xi_k \bs \xi_k']\e^{A'v}P_s'\mathrm{e}^{-\lambda_1 v}\mathrm{d}v
\end{equation*}
as \(n\rightarrow \infty\). Thus, \eqref{eq: eq} holds which is Condition 2 of Lemma \ref{Lemma: MCLT}. Set
\begin{align}
    &\Sigma = \sum_{k=1}^d\int_{0}^{\infty}P_s\mathrm{e}^{Av}a_kv(\bs\mu)_k\E[\bs \xi_k \bs \xi_k']\e^{A'v}P_s'\mathrm{e}^{-\lambda_1 v}\mathrm{d}v. \label{eq: sigma thm 1.1}
\end{align}
For Condition 3 of Lemma \ref{Lemma: MCLT} we take \(p=4\). With this choice Condition 3 immediately follows by Lemma \ref{Lemma: fourth moments} for small components, since 
\begin{equation*}
    N^{-2} \sum_{i=1}^d\sum_{j=1}^{N\mu_i} \E\F[\F\|\omega^{-(m_1-1)/2}\mathrm{e}^{-\lambda_1(\omega+t)/2} P_s(\bs X_{ijn}(\omega +t)-\mathrm{e}^{A(\omega+ t)}\bs e_i)\R\|_2^4\R] \leq \cst N^{-1} \rightarrow 0
\end{equation*}
as \(n \rightarrow \infty\). Therefore, Lemma \ref{Lemma: MCLT} implies
\begin{align}
   &\bs  Z_{n}(t) \conindis{} \bs Y_{1} \sim \mathcal{N}(0, \Sigma)\label{eq: theorem 1.1 convergence 1}
\end{align}
as \(n \rightarrow \infty\).
\\~\\
\underline{Convergence of \(\bs Z_n\) for \(k=2\):} Fix \(-T\leq t_1 \leq t_2 \leq T\). Our plan is to show \(\bs Z_n(t_2) - \e^{(A-\lambda_1/2)(t_2-t_1)}\bs Z_n(t_1)\) converges independently of \(\bs Z_n(t_1)\). Once this is done, we use that
\begin{equation*}
    \bs Z_n(t_2) =  (\bs Z_n(t_2) - \e^{(A-\lambda_1/2)(t_2-t_1)}\bs Z_n(t_1))+\e^{(A-\lambda_1/2)(t_2-t_1)}\bs Z_n(t_1)
\end{equation*}
 and \eqref{eq: theorem 1.1 convergence 1} to obtain the joint convergence of \((\bs Z_n(t_1),\bs Z_n(t_2))\). By the branching property \eqref{equ:branching property} and the strong Markov property \eqref{eq: strong markov prop}, we have 
\begin{align}
    \bs Z_{n}(t_2) - &\e^{(A-\lambda_1/2)(t_2-t_1)}\bs Z_{n}(t_1)= N^{-1/2}\omega^{-(m_1-1)/2}e^{-\lambda_1(\omega+t_2)/2}P_s(\bs X_n(\omega+t_2) - \e^{A(t_2-t_1)}\bs X_n(\omega+ t_1)) \nonumber\\
    &\quad = N^{-1/2}\omega^{-(m_1-1)/2}\e^{-\lambda_1(\omega+t_1)/2} \sum_{i=1}^d\sum_{j=1}^{X_n(\omega+ t_1)_i}\e^{-\lambda_1(t_2-t_1)/2}P_s(\widehat{\bs X}_{ijn}(t_2- t_1)-\e^{A(t_2-t_1)} \bs e_i), \label{eq: rowsumstuff}
\end{align}
where the \(\widehat{\bs X}_{ijn}\) are independent branching processes with initial conditions \(\bs e_i\), identical replacement structures \((\bs a, (\bs \xi_i)_{i=1}^d)\), and are independent of \(\mathcal{F}_{t_1}^{(n)}\). We see, conditionally on \(\bs X_n(\omega+t_1)\), the summands in the r.h.s.\ of \eqref{eq: rowsumstuff} are a triangular array of mean-zero independent random variables independent of \(\bs X_n(\omega+t_1)\) indexed by \(n\). This is the setup required to apply Lemma \ref{Lemma: MCLT}, we are left to check the conditions hold. We show the three conditions by proving, as \(n\rightarrow \infty\):
\begin{enumerate}
    \item  \(N^{-1}\omega^{-(m_1-1)}\e^{-\lambda_1(\omega+t_1)}\bs X_n(\omega+t_1) \coninprob \bs v(\bs \mu)\).
    \item\( N^{-1}\omega^{-(m_1-1)}\e^{-\lambda_1(\omega+t_1)} \sum_{i=1}^d\sum_{j=1}^{X_n(\omega+ t_1)_i}\Var(\e^{-\lambda_1(t_2-t_1)/2}P_s\widehat{\bs X}_{ijn}(t_2- t_1))\coninprob  \Sigma(t_2-t_1) \),
    \begin{equation*}
       \Sigma(t)= \sum_{i=1}^d  v(\bs\mu)_i \Sigma_{i}(t), \quad\Sigma_{i}(t):= \sum_{k=1}^d \int_{0}^{t} P_s\mathrm{e}^{A(t-v)}\E[\bs \xi_k \bs \xi_k']\mathrm{e}^{A'(t-v)}P_s' a_k(\mathrm{e}^{Av}\bs e_i)_k\e^{-\lambda_1t}\mathrm{d}v.
    \end{equation*}
    \item \(N^{-2} \omega^{-2(m_1-1)}\e^{-2\lambda_1(\omega+t_1)}\sum_{i=1}^d\sum_{j=1}^{X_n(\omega+ t_1)_i}\E[\|\e^{-\lambda_1(t_2-t_1)/2}P_s(\widehat{\bs X}_{ijn}(t_2-t_1)-\mathrm{e}^{A (t_2-t_1)}\bs e_i)\|_2^4] \coninprob 0.\)
\end{enumerate}
Here, 1, 2, and 3 are Conditions 1, 2, and 3 of Lemma \ref{Lemma: MCLT} respectively, where in Condition 3 we have taken \(p=4\). Starting with Condition~1, by the branching property \eqref{equ:branching property} and Jensen's inequality in the first inequality, and \eqref{eq: fourth moment main} in the second, we have
\begin{align}
    \Var(N^{-1}\omega^{-(m_1-1)}\e^{-\lambda_1(\omega+t_1)}X_n(\omega+ t_1)_k) &\leq  N^{-2}\sum_{i=1}^d \sum_{j=1}^{N\mu_i}\E[\omega^{-4(m_1-1)}\e^{-4\lambda_1(\omega+t_1)} (X_{ijn}(\omega+ t_1)-\e^{A(\omega+t_1)}\bs e_i)_k^4]^{1/2} \nonumber \\
    &\leq \cst N^{-1}  \rightarrow 0 \label{eq: improve 8}
\end{align}
as \(n\rightarrow \infty\). Condition 1 follows from this, \eqref{eq: CMBP First moment}, and \eqref{th1.1 ls 2}. For Condition 2, we have by \eqref{eq: MCBP var}
\begin{equation*}
    N^{-1}\omega^{-(m_1-1)}\e^{-\lambda_1(\omega+t_1)} \sum_{i=1}^d\sum_{j=1}^{X_n(\omega+ t_1)_i}\Var(\e^{-\lambda_1(t_2-t_1)/2}P_s\widehat{\bs X}_{ijn}(t_2- t_1)) = N^{-1}\sum_{i=1}^d \omega^{-(m_1-1)}\e^{-\lambda_1(\omega+t_1)}X_n(\omega+t_1)_i \Sigma_{i}(t_2-t_1).
\end{equation*}
Condition 2 follows by applying Condition 1 to the r.h.s. Lastly, for Condition 3, by Lemma \ref{Lemma: fourth moments} for small components, we have
\begin{align}
 N^{-2} \omega^{-2(m_1-1)}\e^{-2\lambda_1(\omega+t_1)}\sum_{i=1}^d\sum_{j=1}^{X_n(\omega+ t_1)_i}&\E[\|\e^{-\lambda_1(t_2-t_1)/2}P_s(\widehat{\bs X}_{ijn}(t_2-t_1)-\mathrm{e}^{A (t_2-t_1)}\bs e_i)\|_2^4]\nonumber\\
 &\qquad\qquad\qquad\qquad\leq \cst N^{-2}  \sum_{i=1}^d \omega^{-2(m_1-1)}\e^{-2\lambda_1(\omega+t_1)}X_n(\omega+ t_1)_i.\label{eq: theorem 1.1 interim 3 ls}
\end{align}
Condition 3 follows by applying Condition 1 to the r.h.s.\ Therefore, by Lemma \ref{Lemma: MCLT}, for any \(\bs Y_2\)-continuity set \(B\subseteq \mathbb{R}^d\), we have as \(n \rightarrow \infty\)
\begin{equation*}
    \mathbb{P}\F(\bs Z_n(t_2) - \e^{(A-\lambda_1/2)(t_2-t_1)}\bs Z_n(t_1)\in B\M|\bs X_n(\omega+ t_1)\R)- \mathbb{P}\F(\bs Y_2 \in B\R) \coninprob 0,
\end{equation*}
where \(\bs Y_{2} \sim \mathcal{N}(0, \Sigma(t_2-t_1))\). Using the fact that the summands in \eqref{eq: rowsumstuff} are independent of \(\mathcal{F}_{t_1}^{(n)}\), we can extend this result to 
\begin{equation}
\label{eq: may need 5}
        \mathbb{P}\F(\bs Z_n(t_2) - \e^{(A-\lambda_1/2)(t_2-t_1)}\bs Z_n(t_1)\in B\M| \mathcal{F}_{t_1}^{(n)}\R)- \mathbb{P}\F(\bs Y_2 \in B\R) \coninprob 0
\end{equation}
as \(n \rightarrow \infty\). We are going to use this and \eqref{eq: theorem 1.1 convergence 1} to show, for any \(\bs Y_1,\bs Y_2\)-continuity sets \(B_1,B_2\subseteq \mathbb{R}^d\), 
\begin{align}
    &\mathbb{P}\F(\bs Z_n(t_1) \in B_1, \bs Z_n(t_2) - \e^{(A-\lambda_1/2)(t_2-t_1)}\bs Z_n(t_1) \in B_2 \R)& \nonumber\\
   & \hspace{4cm}=\mathbb{P}\F(\bs Z_n(t_1) \in B_1\R)\mathbb{P}\F(\bs Z_n(t_2) - \e^{(A-\lambda_1/2)(t_2-t_1)}\bs Z_n(t_1) \in B_2 \M| \bs Z_n(t_1) \in B_1 \R) \nonumber \\
  & \hspace{4cm}\rightarrow \mathbb{P}\F(\bs Y_{1} \in B_1\R) \mathbb{P}\F( \bs Y_{2} \in B_2\R) \label{eq: may need 10}
\end{align}
 as \(n\rightarrow \infty\). Firstly, by \eqref{eq: theorem 1.1 convergence 1},
\begin{equation}
\label{eq: may need 61}
\mathbb{P}\F(\bs Z_n(t_1) \in B_1\R)\rightarrow \mathbb{P}(\bs Y_{1} \in B_1)
\end{equation} as \(n \rightarrow \infty\). So, if \(\mathbb{P}(\bs Y_{1} \in B_1)=0\), then \eqref{eq: may need 10} holds by \eqref{eq: may need 61}. If \(\mathbb{P}(\bs Y_{1} \in B_1)\neq 0\), then \(\mathbb{P}\F(\bs Z_n(t_1) \in B_1\R) \not\rightarrow 0\) by \eqref{eq: may need 61}. This, the fact that \(\{\bs Z_n(t_1) \in B_1\}\subseteq \mathcal{F}_{t_1}^{(n)}\), and that \eqref{eq: may need 5} holds in probability implies  
\begin{equation*}
   \mathbb{P}\F(\bs Z_n(t_2) - \e^{(A-\lambda_1/2)(t_2-t_1)}\bs Z_n(t_1) \in B_2 \M| \bs Z_n(t_1) \in B_1 \R) \rightarrow \mathbb{P}\F(\bs Y_2 \in B_2\R)  
\end{equation*}
as \(n \rightarrow \infty\). Thus, \eqref{eq: may need 10} holds. The result \eqref{eq: may need 10} implies that jointly, as \(n \rightarrow \infty\),
\begin{align}
    &\bs Z_n(t_1) \conindis \bs Y_1,\nonumber\\
    &\bs Z_n(t_2)= \e^{(A-\lambda_1/2)(t_2-t_1)} \bs Z_n(t_1)+(\bs Z_n(t_2) - \e^{(A-\lambda_1/2)(t_2-t_1)} \bs Z_n(t_1))  \conindis \e^{(A-\lambda_1/2)(t_2-t_1)} \bs Y_1 + \bs Y_2, \label{eq: t1.1 second conv}
\end{align}
where \(\bs Y_1\) and \(\bs Y_2\) are independent of each other and have distributions \(\bs Y_1 \sim \mathcal{N}(0,\Sigma)\), \(\bs Y_2 \sim \mathcal{N}(0,\Sigma(t_2-t_1))).\)
\\~\\
\underline{Convergence of \(\bs Z_n\) for an arbitrary \(k\):} Let \(-T \leq t_1 \leq ...  \leq t_k \leq T\) be a sequence of fixed times. We take the sequences of random variables\[ \hat{\bs Z}_{n,1}:=\bs Z_n(t_1),\quad \hat{\bs Z}_{n,i}:=\bs Z_n(t_i) -\e^{(A-\lambda_1/2)(t_i-t_{i-1})}\bs Z_n(t_{i-1}),\quad i=2,...,k.\] By \eqref{eq: theorem 1.1 convergence 1}, we have, as \(n \rightarrow \infty\),
\begin{equation}
    \hat{\bs Z}_{n,1} \conindis \bs Y_1, \label{eq: th1.1 multitime 1}
\end{equation}
where \(\bs Y_1 \sim \mathcal{N}(0,\Sigma)\). By the arguments that showed \eqref{eq: may need 5}, for each \(i\in \{2,...,k\}\), and any \(\bs Y_i\)-continuity set \(B \in \mathbb{R}^d\),
\begin{equation}
     \mathbb{P}(\hat{\bs Z}_{n,i} \in B| \mathcal{F}^{(n)}_{t_{i-1}})-\mathbb{P}(\bs Y_{i} \in B) \coninprob 0
     \label{eq: th1.1 multitime 2}
\end{equation}
as \(n \rightarrow \infty\), where \(\bs Y_{i}\sim \mathcal{N}(0,\Sigma(t_i-t_{i-1}))\). By the same argument that gave \eqref{eq: may need 10}, we have, for any \(\bs Y_1,...,\bs Y_k\)~-continuity sets \(B_1,...,B_k \subseteq \mathbb{R}^d\), 
\begin{align*}
    \mathbb{P}(\hat{\bs Z}_{n,1} \in B_1,...,\hat{\bs Z}_{n,k}\in B_k) &= \mathbb{P}(\hat{\bs Z}_{n,1} \in B_1)\mathbb{P}(\hat{\bs Z}_{n,2}\in B_2| \hat{\bs Z}_{n,1} \in B_1)\times...\\
    &\quad\times \mathbb{P}(\hat{\bs Z}_{n,k} \in B_k|\cap_{i=1}^{k-1}\{\hat{\bs Z}_{n,i}\in B_{i}\})\\
    & \rightarrow \mathbb{P}(\bs Y_1 \in B_1)\times ...\times \mathbb{P}(\bs Y_{k}\in B_k)
\end{align*}
as \(n \rightarrow \infty\). (We use that, for \(j\in \{2,...,k\}\), \(\cap_{i=1}^{j-1}\{\hat{\bs Z}_{n,i}\in B_{i}\}\in \mathcal{F}_{t_{j-1}}^{(n)}\), along with \eqref{eq: th1.1 multitime 2}.) This implies, jointly, as \(n \rightarrow \infty\),
\begin{align*}
    &\bs Z_n(t_1) \conindis \bs Y_1\\
    &\bs Z_n(t_i) =\e^{(A-\lambda_1/2)(t_i-t_{1})}\bs Z_n(t_1)+ \sum_{j=2}^{i}\e^{(A-\lambda_1/2)(t_i-t_{j})}(\bs Z_n(t_j) - \e^{(A-\lambda_1/2)(t_j-t_{j-1})} \bs Z_n(t_{j-1})) \\
    &\quad\quad\quad\conindis e^{(A-\lambda_1/2)(t_i-t_1)}\bs Y_1+\sum_{j=2}^{i} e^{(A-\lambda_1/2)(t_i-t_j)}\bs Y_{j}, \quad i \in \{2,...,k\},
\end{align*}
where the \(\bs Y_1,\bs Y_{j}\), \(j \in \{2,...,k\}\) are pairwise independent, and, for \(j \in \{2,...,k\}\), \(\bs Y_1\sim \mathcal{N}(0,\Sigma)\), \(\bs Y_{j}\sim \mathcal{N}\F(0, \Sigma(t_j-t_{j-1})\R).\) To show Condition 1 of Theorem \ref{Theorem:Bill con}, we are left to check 
\begin{align*}
    &\bigg(\bs Y_1,...,e^{(A-\lambda_1/2)(t_k-t_1)}\bs Y_1+\sum_{j=2}^{k} e^{(A-\lambda_1/2)(t_i-t_j)}\bs Y_{j}\bigg) \eqindis (\bs W_s(t_1),...,\bs W_s(t_k)).
    \end{align*}
Since both sides are Gaussian, we only need to check the covariance matrices of each possible choice of two components are equal. For \(1\leq \ell \leq m \leq k\),
\begin{align}
& \Cov\bigg(\e^{(A-\lambda_1/2)(t_m-t_1)}\bs Y_1+\sum_{j=2}^m \e^{(A-\lambda_1/2)(t_m-t_j)}\bs Y_{j}, \bs \e^{(A-\lambda_1/2)(t_{\ell}-t_1)}\bs Y_1+\sum_{j=2}^{\ell}\e^{(A-\lambda_1/2)(t_{\ell}-t_j)}\bs Y_{j}\bigg) \nonumber
 \\ &= \e^{(A-\lambda_1/2)(t_{m}-t_{\ell})}\bigg(\Var(\e^{(A-\lambda_1/2)(t_{\ell}-t_1)}\bs Y_1)+\sum_{j=2}^{ \ell}\Var(\e^{(A-\lambda_1/2)(t_{\ell}-t_j)}\bs Y_{j})\bigg) \nonumber\\
    &=\e^{(A-\lambda_1/2)(t_{m}-t_{\ell})}\bigg(\e^{(A-\lambda_1/2)(t_{\ell}-t_1)}\Sigma \e^{(A'-\lambda_1/2)(t_{\ell}-t_1)}+\sum_{j=2}^{\ell}\e^{(A-\lambda_1/2)(t_{\ell}-t_j)}\Sigma(t_j-t_{j-1})\e^{(A'-\lambda_1/2)(t_{\ell}-t_j)}\bigg). \label{eq: covariance small 1}
\end{align}
The summands in the final line satisfy the following equalities.
\begin{align}
 \e^{(A-\lambda_1/2)(t_{\ell}-t_1)}\Sigma \e^{(A'-\lambda_1/2)(t_{\ell}-t_1)}&=  \sum_{k=1}^d\int_{0}^{\infty} P_s\mathrm{e}^{A(v+t_{\ell}-t_1)}a_kv(\bs\mu)_k\E[\bs \xi_k \bs \xi_k']\mathrm{e}^{A'(v+t_{\ell}-t_1)}P_s' \mathrm{e}^{-\lambda_1 (v+t_{\ell}-t_1)}\mathrm{d}v\nonumber\\
 &= \sum_{k=1}^d\int_{t_{\ell}-t_1}^{\infty} P_s\mathrm{e}^{Au}a_kv(\bs\mu)_k\E[\bs \xi_k \bs \xi_k']\mathrm{e}^{A'u}P_s' \mathrm{e}^{-\lambda_1 u}\mathrm{d}u, \label{eq: covariance small 2}
\end{align}
where we have used the change of variables \(u=v+t_{\ell}-t_1\). For \(j \in \{2,...,\ell\}\), 
\begin{align}
&\e^{(A-\lambda_1/2)(t_{\ell}-t_j)}\Sigma(t_j-t_{j-1})\e^{(A'-\lambda_1/2)(t_{\ell}-t_j)}\nonumber\\&=\sum_{k=1}^d \int_{0}^{t_j-t_{j-1}} P_s\mathrm{e}^{A(t_{\ell}-t_{j-1}-v)}\E[\bs \xi_k \bs \xi_k']\mathrm{e}^{A'(t_{\ell}-t_{j-1}-v)}P_s'  a_k(\mathrm{e}^{Av}\bs v(\bs\mu))_k\e^{-\lambda_1(t_{\ell}-t_{j-1})}\mathrm{d}v \nonumber \\
 &=\sum_{k=1}^d \int_{t_{\ell}-t_{j}}^{t_{\ell}-t_{j-1}} P_s\mathrm{e}^{Au}a_kv(\bs\mu)_k\E[\bs \xi_k \bs \xi_k']\mathrm{e}^{A'u}P_s' \mathrm{e}^{-\lambda_1 u}\mathrm{d}u, \label{eq: covariance small 3}
\end{align}
where in the final line, we have used the change of variables \(u = t_{\ell}-t_{p-1}-v\), and that \((\mathrm{e}^{Av}\bs v(\bs\mu))_k = \e^{\lambda_1 v}\bs v(\bs\mu)_k\). (Indeed, this follows since  \(\e^AN_A^{m_1-1}P_{\lambda_1}=\e^{\lambda_1}P_{\lambda_1}.\)) Using \eqref{eq: covariance small 2} and \eqref{eq: covariance small 3} in \eqref{eq: covariance small 1} gives
\begin{align*}
 &\Cov\bigg(\e^{(A-\lambda_1/2)(t_m-t_1)}\bs Y_1+\sum_{j=2}^m \e^{(A-\lambda_1/2)(t_m-t_j)}\bs Y_{j}, \bs \e^{(A-\lambda_1/2)(t_{\ell}-t_1)}Y_1+\sum_{j=2}^{\ell}\e^{(A-\lambda_1/2)(t_{\ell}-t_j)}\bs Y_{j}\bigg)\\
 &=\e^{(A-\lambda_1/2)(t_{\ell}-t_{m})}\Sigma =\Cov(\bs W_{s}(t_{m}),\bs W_{s}(t_{\ell})),
\end{align*}
thus Condition 1 of Theorem \ref{Theorem:Bill con} holds. \\~\\
\textbf{Condition 2 of Theorem \ref{Theorem:Bill con}:} Since \(\bs W_s\) is a Gaussian processes with continuous covariance function, Condition 2 holds. \\~\\
\textbf{Condition 3 of Theorem \ref{Theorem:Bill con}:} Fix an arbitrary \(n \geq 1\), \(-T \leq t_1 \leq t_2 \leq t_3 \leq T\). By the Cauchy-Schwarz inequality, we have that
\begin{align}
   \E  [\| \bs Z_n(t_2)-\bs Z_n(t_1)\|_2^2 \| \bs Z_n(t_3)-\bs Z_n(t_2)\|_2^2&]=\E \F [\| \bs Z_n(t_2)-\bs Z_n(t_1)\|_2^2 \E[\| \bs Z_n(t_3)-\bs Z_n(t_2)\|_2^2|\mathcal{F}_{t_2}^{(n)}]\R]\nonumber\\
    &\leq \E \F [\| \bs Z_n(t_2)-\bs Z_n(t_1)\|_2^4\R]^{1/2}\E\F[\E[\| \bs Z_n(t_3)-\bs Z_n(t_2)\|_2^2|\mathcal{F}_{t_2}^{(n)}]^2\R]^{1/2}.
    \label{eq: t1.1 tightness 0ss}
\end{align}
To prove Condition 3 of Theorem \ref{Theorem:Bill con}, we are going to show
\begin{align}
    &\E[\| \bs Z_n(t_2)-\bs Z_n(t_1)\|_2^4]^{1/2} \leq \cst (t_2-t_1)^{1/2} \label{eq: need bound crit 1},\\
    &\E\F[\E[\| \bs Z_n(t_3)-\bs Z_n(t_2)\|_2^2|\mathcal{F}_{t_2}^{(n)}]^{2}\R]^{1/2} \leq \cst (t_2-t_1), \label{eq: need bound crit 2}
\end{align}
which along with \eqref{eq: t1.1 tightness 0ss} implies the condition holds. We start by showing \eqref{eq: need bound crit 1}. By the triangle inequality, we have
\begin{align}
   \E \F [\| \bs Z_n(t_2)-\bs Z_n(t_1)\|_2^4\R] \leq  \E \F [\| \bs Z_n(t_2)-\e^{(A-\lambda_1/2)(t_2-t_1)}\bs Z_n(t_1)\|_2^4\R] +  \E \F [\| (\e^{(A-\lambda_1/2)(t_2-t_1)}-1)\bs Z_n(t_1)\|_2^4\R]. \label{eq: improve 9}
\end{align}
For the second term on the r.h.s., by Lemma \ref{Lemma: fourth moments} for small components, we have 
\begin{equation}
     \E \F [\| (\e^{(A-\lambda_1/2)(t_2-t_1)}-1)\bs Z_n(t_1)\|_2^4\R] \leq \cst (t_2-t_1)^4. \label{eq: improve 10}
\end{equation}
Next, by the strong Markov property \eqref{eq: strong markov prop}, we can apply Lemma \ref{Lemma: fourth moments} for small components conditionally on \(\mathcal{F}_{t_1}^{(n)}\) to the branching process \(\bs X_n(\omega+t_1+\cdot)\)  with initial condition \(\bs X_n(\omega +t_1)\). This gives the following bound for the first term on the r.h.s.\ of \eqref{eq: improve 9}
\begin{align*}
    &\E \F [\| \bs Z_n(t_2)-\e^{(A-\lambda_1/2)(t_2-t_1)}\bs Z_n(t_1)\|_2^4\R] \\
   & \qquad\qquad\qquad\qquad\qquad\qquad= N^{-2}\omega^{-2(m_1-1)}\e^{-2\lambda_1(\omega+t_2)}\E\F[\E\F[\| P_s(\bs X_n(\omega+t_2)- \e^{A(t_2-t_1)}\bs X_n(\omega+t_1))\|_2^4\M|\mathcal{F}_{t_1}\R]\R]\\
    &\qquad\qquad\qquad\qquad\qquad\qquad\leq \cst N^{-2}\omega^{-2(m_1-1)}\e^{-2\lambda_1(\omega+t_1)}\E[\|\bs X_n(\omega+t_1)\|_2^2](t_2-t_1).
\end{align*}
Then, applying Jensen's inequality followed by \eqref{eq: fourth moment main} to the r.h.s.\ gives
\begin{equation}
  \E \F [\| \bs Z_n(t_2)-\e^{(A-\lambda_1/2)(t_2-t_1)}\bs Z_n(t_1)\|_2^4\R] \leq \cst (t_2-t_1).
  \label{eq: tightness final 3}
\end{equation}
Using \eqref{eq: improve 10} and \eqref{eq: tightness final 3} in \eqref{eq: improve 9} gives \eqref{eq: need bound crit 1}. For \eqref{eq: need bound crit 2}, we first have
\begin{align*}
&\E\F[ \F(\bs Z_n(t_3)-\e^{(A-\lambda_1/2)(t_3-t_2)}\bs Z_n(t_2)\R)\bs Z_n(t_2)'\M|\mathcal{F}_{t_2}^{(n)}\R] = \E\F[ \Big(\bs Z_n(t_3)-\e^{(A-\lambda_1/2)(t_3-t_2)}\bs Z_n(t_2)\Big)\M|\mathcal{F}_{t_2}^{(n)}\R] \bs Z_n(t_2)'  \\
&\qquad\qquad\qquad\qquad=\omega^{-(m_1-1)/2}e^{-\lambda_1(\omega+t_3)/2}N^{-1/2}\sum_{i=1}^d \sum_{j=1}^{X_n(\omega+ t_2)_i}\mathbb{E}\F[\widehat{\bs X}_{ijn}(t_3-t_2)-\e^{A(t_3-t_2)}\bs e_i\R]\bs Z_n(t_2)' = 0,
\end{align*}
where the second equality follows from \eqref{eq: rowsumstuff}, and in the final equality the expectation vanishes by \eqref{eq: CMBP First moment}. This result implies the third equality in the following
\begin{align}
&\E[\| \bs Z_n(t_3)-\bs Z_n(t_2)\|_2^2|\mathcal{F}_{t_2}^{(n)}]= \mathrm{tr}\E[ (\bs Z_n(t_3)-\bs Z_n(t_2))(\bs Z_n(t_3)-\bs Z_n(t_2))'|\mathcal{F}_{t_2}^{(n)}] \nonumber\\
&=\mathrm{tr}\E\Big[ \Big(\bs Z_n(t_3)-\e^{(A-\lambda_1/2)(t_3-t_2)}\bs Z_n(t_2)+(\e^{(A-\lambda_1/2)(t_3-t_2)}-1)\bs Z_n(t_2)\Big) \nonumber\\
&\qquad\qquad\qquad\qquad\qquad\qquad\qquad\qquad\Big(\bs Z_n(t_3)-\e^{(A-\lambda_1/2)(t_3-t_2)}\bs Z_n(t_2)+(\e^{(A-\lambda_1/2)(t_3-t_2)}-1)\bs Z_n(t_2)\Big)'\Big|\mathcal{F}_{t_2}^{(n)}\Big]\nonumber\\
&= \mathrm{tr}\E\F[ \Big(\bs Z_n(t_3)-\e^{(A-\lambda_1/2)(t_3-t_2)}\bs Z_n(t_2)\Big) \Big(\bs Z_n(t_3)-\e^{(A-\lambda_1/2)(t_3-t_2)}\bs Z_n(t_2)\Big)'\M|\mathcal{F}_{t_2}^{(n)}\R]\nonumber\\
&\quad+\mathrm{tr}\F( (\e^{(A-\lambda_1/2)(t_3-t_2)}-1)\bs Z_n(t_2)\F((\e^{(A-\lambda_1/2)(t_3-t_2)}-1)\bs Z_n(t_2)\R)'\R). \label{eq: t1.1 tightness 1}
\end{align}
For the second term in the final equality on the r.h.s., we have
\begin{equation}
\mathrm{tr}\F( (\e^{(A-\lambda_1/2)(t_3-t_2)}-1)\bs Z_n(t_2)\F((\e^{(A-\lambda_1/2)(t_3-t_2)}-1)\bs Z_n(t_2)\R)'\R) \leq \cst  (t_3-t_2)^2 \mathrm{tr}(\bs Z_n(t_2) \bs Z_n(t_2)'). \label{eq: t1.1 tightness 2}
\end{equation}
For the first term in the final equality on the r.h.s.\ of \eqref{eq: t1.1 tightness 1}, we have
\begin{align}
  &\mathrm{tr}\E\F[ \Big(\bs Z_n(t_3)-\e^{(A-\lambda_1/2)(t_3-t_2)}\bs Z_n(t_2)\Big) \Big(\bs Z_n(t_3)-\e^{(A-\lambda_1/2)(t_3-t_2)}\bs Z_n(t_2)\Big)'\M|\mathcal{F}_{t_2}^{(n)}\R]\nonumber\\
  &=N^{-1}\omega^{-(m_1-1)}\e^{-\lambda_1(\omega+t_3)}\sum_{i=1}^d\sum_{j=1}^{X_n(\omega+ t_2)_i}\mathrm{tr}\E\F[\F(P_s(\widehat{\bs X}_{ijn}(t_3-t_2)-\e^{A(t_3-t_2)}\bs e_i)\R) \F(P_s(\widehat{\bs X}_{ij n}(t_3-t_2)-\e^{A(t_3-t_2)}\bs e_i)\R)'\R]\nonumber\\
  &= N^{-1}\omega^{-(m_1-1)}\e^{-\lambda_1(\omega+t_3)}\sum_{i=1}^d\sum_{j=1}^{X_n(\omega+ t_2)_i}\mathrm{tr}\Var\F(P_s(\widehat{\bs X}_{ij n}(t_3-t_2)-\e^{A(t_3-t_2)}\bs e_i)\R)\nonumber\\
  &\leq \cst N^{-1}\omega^{-(m_1-1)} e^{-\lambda_1(\omega+t_3)}\sum_{i=1}^d X_n(\omega+ t_2)_i (t_3-t_2), \label{eq: t1.1 tightness 3 smals} \end{align}
where in the second line we use \eqref{eq: rowsumstuff} (the cross terms vanish since they are the expectation of the product of two independent mean-zero random variables). In the final line, we use \eqref{eq: MCBP var}, where the constant arises by uniformly bounding the integrand in \eqref{eq: MCBP var} by its maximum over all choices of \(t_3-t_2 \in [0,2T]\). Using \eqref{eq: t1.1 tightness 2} and \eqref{eq: t1.1 tightness 3 smals} in \eqref{eq: t1.1 tightness 1} gives
\begin{align}
 \E\F[\E[\| \bs Z_n(t_3)-\bs Z_n(t_2)\|_2^2|\mathcal{F}_{t_2}^{(n)}]^2\R] &\leq \cst  (t_3-t_2)^4 \E[\mathrm{tr}(\bs Z_n(t_2) \bs Z_n(t_2)')^2]\nonumber\\
 &\quad+\cst N^{-2}(t_3-t_2)^2\omega^{-2(m_1-1)}e^{-2\lambda_1(\omega+t_3)}\E\F[\F(\sum_{i=1}^d X(\omega+ t_2)_i \R)^2\R]. \label{eq: t1.1 tightness 4}  
\end{align}
We bound the first term on the r.h.s.\ using Lemma \ref{Lemma: fourth moments} for small components, and the second term using Jensen's inequality (to get the fourth moment of the branching process) followed by \eqref{eq: fourth moment main}. This results in the bound 
\begin{equation}
 \E\F[\E[\| \bs Z_n(t_3)-\bs Z_n(t_2)\|_2^2|\mathcal{F}_{t_2}^{(n)}]^2\R]\leq  \cst ((t_3-t_2)^4+(t_3-t_2)^2) \leq \cst (t_3-t_2)^2,
\end{equation}
which is \eqref{eq: need bound crit 2}. Since \eqref{eq: need bound crit 1} and \eqref{eq: need bound crit 2} hold, Condition 3 of Theorem \ref{Theorem:Bill con} holds by \eqref{eq: t1.1 tightness 0ss}. Therefore, by Theorem \ref{Theorem:Bill con}, 
\begin{equation*}
    \bs Z_n(t) \conindis \bs W_{s}(t) \text{ in \(D[-T,T]\)}
    \end{equation*}
as \(n \rightarrow \infty\). Since \(T >0\) was arbitrary, the result holds in \(D(-\infty,\infty)\).
\end{proof}
\begin{proof}[Proof of Theorem \ref{Theorem: Main continuous time} LT\textsubscript{s} \eqref{eq: small component convergence in probability}]
For any integer \(T>0\), and \(n\geq 1\), we have that 
\begin{equation}
    \E\F[\sup_{0\leq t \leq T} \|N^{-1/2}\mathrm{e}^{-\lambda_1  t/2}P_s( \bs X_n( t)-N\mathrm{e}^{A  t}\bs \mu)\|^4_2\R] \leq \sum_{k=0}^{T-1}\E\F[\sup_{k\leq t \leq k+1} \|N^{-1/2}\mathrm{e}^{-\lambda_1  t/2}P_s( \bs X_n( t)-N\mathrm{e}^{A  t}\bs \mu)\|^4_2\R]. \label{eq: Bound 2}
\end{equation}
Also, for each fixed \(n\) and \(k\), the process \(N^{-1/2}\e^{A(k+1-t)}P_s(\bs X_n( t)-N\e^{At}\bs \mu)\), \(t\geq 0\), is a martingale by Lemma \ref{prop: martingale bp}. This and Doob's maximal inequality imply, for any \(n\geq 1\), and \(k\geq 0\),
\begin{align}
    \E\F[\sup_{k\leq t \leq k+1} \|N^{-1/2}\mathrm{e}^{-\lambda_1  t/2}P_s( \bs X_n( t)-N\mathrm{e}^{A  t}\bs \mu)\|^4_2\R] &\leq \cst \E\F[\sup_{k\leq t \leq k+1} \|N^{-1/2}\mathrm{e}^{-\lambda_1 k/2 }\e^{A(k+1-t)}P_s( \bs X_n( t)-N\mathrm{e}^{A  t}\bs \mu)\|^4_2\R] \nonumber\\
    &\leq \cst \E\F[ \|N^{-1/2}\mathrm{e}^{-\lambda_1 k/2 }P_s( \bs X_n( k+1)-N\mathrm{e}^{A (k+1)}\bs \mu)\|^4_2\R]\nonumber\\
    &\leq \cst (k+2)^{2(m_1-1)},\label{eq: Bound 3}
\end{align}
where in the first inequality we have used that \(\sup_{0\leq t \leq 1}\|\e^{At}\|_2\leq \cst\), and in the final inequality we have used Lemma \ref{Lemma: fourth moments} for small components. Fix \(T\geq 0\). Then, by \eqref{eq: Bound 2} and \eqref{eq: Bound 3}, we have
\begin{equation*}
     \E\F[\sup_{0\leq t \leq \omega T} \|N^{-1/2}\omega ^{-(2m_1-1)/4}K^{-1}\mathrm{e}^{-\lambda_1  t/2}P_s( \bs X_n( t)-N\mathrm{e}^{A  t}\bs \mu)\|^4_2\R] \leq \cst K^{-4}\omega^{-(2m_1-1)}\sum_{k=0}^{\omega T-1} (k+2)^{2(m_1-1)} \rightarrow 0
\end{equation*}
as \(n \rightarrow \infty\). This implies the result on \(D[0,T]\), and since \(T\) was arbitrary, the result holds on \(D[0,\infty)\).
\end{proof}
\subsection{Proof of Theorem \ref{Theorem: Main continuous time} LT\textsubscript{c}}
Instead of proving Theorem \ref{Theorem: Main continuous time} LT\textsubscript{c} directly, we prove a more general result on the functional convergence of linear combinations of different LT\textsubscript{c} components. We do so, since proving this more general result allows for an easy proof of the dependencies of different LT\textsubscript{c} components in Theorem \ref{theorem: cont time 2}. If we know how the linear combinations of these components behave we can apply a functional version of the Cr\'amer-Wold theorem (see Theorem \ref{theorem: cramer wold}) to get joint convergence. 
\begin{lemma}
 \label{Theorem: large time}
Take the setting of Theorem \ref{Theorem: Main continuous time} LT. Let \(p\geq 1\). For \(1\leq i \leq p\), let \(\lambda_{ci} \in \Lambda_c\), \(J_i\in \mathcal{J}_{\lambda_{ci}}\) with size \(m_{ci}\), and~\(1\leq \kappa_i\leq m_{ci}\). Also let \(d_1,...,d_p\in \mathbb{R}\). Then, as \(n \rightarrow \infty\),
        \begin{equation}
               N^{-1/2}\sum_{i=1}^p\omega^{-(m_1+2\kappa_i -2)/2}\e^{-\lambda_{ci} \omega t}d_iN_A^{m_{ci}-\kappa_i}P_{J_i}(\bs X_n(\omega t)-N\e^{A\omega t}\bs \mu) \conindis \bs W_{c}(t) \text{ in }  D[0,\infty).
        \end{equation}
Let \(B:= \sum_{i=1}^d a_i v(\bs\mu)_i\E[\bs \xi_i\bs \xi_i']\). The process \(\bs W_{c}\) is a mean-zero Gaussian processes with covariance function, for \(0 \leq t_1\leq t_2 <\infty\),
\begin{align}
   &\Cov(\bs W_{c}(t_2),\bs W_{c}(t_1))= \sum_{k,\ell=1}^{p}\bs 1_{\{\lambda_{ck}=\overline{\lambda}_{c\ell}\}}d_kd_{\ell}N_A^{m_{ck}-1}P_{J_k}B P_{J_{\ell}}'N_A'^{m_{c\ell}-1}\int_{0}^{t_1}\frac{(t_1-v)^{\kappa_{\ell}-1}(t_2-v)^{\kappa_{k}-1}v^{m_1-1}}{(\kappa_{k}-1)!(\kappa_{\ell}-1)!}\mathrm{d}v, \label{eq: covariance function of CT lem 1}
  \\
  &\Cov(\bs W_{c}(t_2),\overline{\bs W}_{c}(t_1))=\sum_{k,\ell=1}^{p}\bs 1_{\{\lambda_{ck}=\lambda_{c\ell}\}}d_kd_{\ell}N_A^{m_{ck}-1}P_{J_k}BP_{J_{\ell}}^*N_A'^{m_{c\ell}-1}\int_{0}^{t_1}\frac{(t_1-v)^{\kappa_{\ell}-1}(t_2-v)^{\kappa_{k}-1}v^{m_1-1}}{(\kappa_{k}-1)!(\kappa_{\ell}-1)!}\mathrm{d}v. \label{eq: covariance function of CT lem 2}
\end{align}
 \end{lemma}
\noindent To get Theorem \ref{Theorem: Main continuous time} LT\textsubscript{c} from Lemma \ref{Theorem: large time}  we take \(p=1\) and \(d_1=1\). We use the following intermediate result to prove Theorem \ref{Theorem: large time}.

\begin{lemma}
\label{lemma: intermediate lem 1}
Take the setting of Theorem \ref{Theorem: Main continuous time} LT. Let \(q \geq 1\). For \(1 \leq i \leq q\), let \(\lambda_{ci}\in \Lambda_c\), \(J_i \in \mathcal{J}_{\lambda_{ci}}\) with size \(m_{ci}\), \(1 \leq \kappa_i \leq m_{ci}\), and \(Q_{i}=N_A^{m_{ci}-\kappa_i}P_{J_i}\). Then, jointly, for all \(i \in \{1,...,q\}\), as \(n \rightarrow \infty\),
\begin{equation}
\label{eq: Convergence 2}
    N^{-1/2}\omega^{-(m_1+2\kappa_i-2)/2}Q_{i}(\e^{-A\omega t}\bs X_n(\omega t)-N\bs \mu) \conindis \bs W_i(t) \text{ in }D[0,\infty),
\end{equation}
where the \(\bs W_i\) are mean-zero Gaussian processes. Furthermore, for \(0 \leq t_1 \leq t_2 <\infty, 1 \leq k \leq \ell \leq q\),
\begin{align*}
    &\Cov(\bs W_{k}(t_2),\overline{\bs W}_{\ell}(t_1))=\bs 1_{\{\lambda_{ck}=\lambda_{c\ell}\}}F(t_1,\kappa_k,\kappa_{\ell}) N_A^{\kappa_{k}-1}Q_kBQ_{\ell}^*N_A'^{\kappa_{\ell}-1},\\
    &\Cov(\bs W_{k}(t_2),\bs W_{\ell}(t_1))=\bs 1_{\{\lambda_{ck}=\overline{\lambda}_{c\ell}\}}F(t_1,\kappa_k,\kappa_{\ell}) N_A^{\kappa_{k}-1}Q_kBQ_{\ell}'N_A'^{\kappa_{\ell}-1},
\end{align*}
where
\begin{align*}
  &F(t,\kappa_k,\kappa_{\ell})= (2-\kappa_{k}-\kappa_{\ell}-m_1)^{-1}((\kappa_{k}-1)!(\kappa_{\ell}-1)!)^{-1}(-t)^{\kappa_{k}+\kappa_{\ell}-1}t^{m_1-1}.  
\end{align*}
\end{lemma}
\noindent We leave the proof of this result until after the proof of Lemma \ref{Theorem: large time}.
\begin{proof}[Proof of Lemma \ref{Theorem: large time}]By \eqref{eq: projection}, we have
\begin{align}
   N^{-1/2}\sum_{i=1}^p\omega^{-(m_1+2\kappa_i -2)/2}&\e^{-\lambda_{ci} \omega t}d_iN_A^{m_{ci}-\kappa_i}P_{J_i}(\bs X_n(\omega t)-N\e^{A\omega t}\bs \mu) \nonumber \\
    &=N^{-1/2}\sum_{i=1}^{p}\sum_{j=0}^{\kappa_i-1}\frac{t^j}{j!}\omega^{-(m_1+2\kappa_i-2j-2)/2}d_iN_A^{m_{ci}-\kappa_i+j}P_{J_i} (\e^{-A\omega t}\bs X_n(\omega t)-N\bs \mu).\label{eq: improve 23}
\end{align}
The summands on the r.h.s.\ are of the form of the l.h.s.\ of \eqref{eq: Convergence 2}. Indeed, if, in Lemma \ref{lemma: intermediate lem 1}, we take \(q = \sum_{i=1}^p\kappa_i\), and for~\(1 \leq k \leq p\), \(1 \leq \ell \leq \kappa_k\), take \(Q_{\sum_{i=1}^{k-1}\kappa_i+\ell}=N_A^{m_{ck}-\kappa_k+\ell-1} P_{J_k}\). Then Lemma \ref{lemma: intermediate lem 1} gives, for \(1 \leq i \leq p\), \(0 \leq j \leq \kappa_i-1\), the joint convergence of
\begin{equation*}
 N^{-1/2}\omega^{-(m_1+2\kappa_i-2j-2)/2}N_A^{m_{ci}-\kappa_i+j}P_{J_i} (\e^{-A\omega t}\bs X_n(\omega t)-N\bs \mu)  \conindis \bs W_{i,j}(t)\text{ in } D[0,\infty)
\end{equation*}
as \(n \rightarrow \infty\), where, for \(1 \leq k,\ell \leq p\), \(0 \leq a \leq \kappa_k-1\), \(0 \leq b \leq \kappa_{\ell}-1\), \(0 \leq t_1 \leq t_2 <\infty\),
\begin{align*}
    &\Cov(\bs W_{k,a}(t_2),\overline{\bs W}_{\ell,b}(t_1))=\bs 1_{\{\lambda_{ck}=\lambda_{c\ell}\}}F(t_1,\kappa_k-a,\kappa_{\ell}-b) N_A^{m_{ck}-1}P_{J_k} B P_{J_{\ell}}^*N_A'^{m_{c\ell}-1},\\
    &\Cov(\bs W_{k,a}(t_2),\bs W_{\ell,b}(t_1))=\bs 1_{\{\lambda_{ck}=\overline{\lambda}_{c\ell}\}}F(t_1,\kappa_k-a,\kappa_{\ell}-b) N_A^{m_{ck}-1}P_{J_k}BP_{J_{\ell}}'N_A'^{m_{c\ell}-1}.
\end{align*}
Using this in \eqref{eq: improve 23} gives
\begin{equation*}
     N^{-1/2}\sum_{i=1}^p\omega^{-(m_1+2\kappa_i -2)/2}\e^{-\lambda_{ci} \omega t}d_iN_A^{m_{ci}-\kappa_i}P_{J_i}(\bs X_n(\omega t)-N\e^{A\omega t}\bs \mu) \conindis \bs W_c(t) \text{ in } D[0,\infty)
\end{equation*}
as \(n \rightarrow \infty\), where, for \(0\leq t_1 \leq t_2\),
\begin{align*}
&\Cov(\bs W_c(t_2),\overline{\bs W}_c(t_1)) = \sum_{k,\ell=1}^p\sum_{a=0}^{\kappa_{k}-1}\sum_{b=0}^{\kappa_{\ell}-1}\bs 1_{\{\lambda_{ck} = \overline{\lambda}_{c\ell}\}}\frac{t_1^{b}t_2^{a}}{a!b!}F(t_1,\kappa_k-a,\kappa_{\ell}-b) N_A^{m_{ck}-1}P_{J_k}BP_{J_{\ell}}^*N_A'^{m_{{c\ell}}-1},\\
   & \Cov(\bs W_c(t_2),\bs W_c(t_1)) = \sum_{k,{\ell}=1}^p\sum_{a=0}^{\kappa_{k}-1}\sum_{b=0}^{\kappa_{{\ell}}-1}\bs 1_{\{\lambda_{ck} = \overline{\lambda}_{c{\ell}}\}}\frac{t_1^{b}t_2^{a}}{a!b!}F(t_1,\kappa_k-a,\kappa_{\ell}-b) N_A^{m_{ck}-1}P_{J_k}BP_{\ell}'N_A'^{m_{{c\ell}}-1}.
\end{align*}
To show that these agree with the covariance functions given in \eqref{eq: covariance function of CT lem 1} and \eqref{eq: covariance function of CT lem 2}, we use that
\begin{equation*}
   (2-\kappa_{k}+a-\kappa_{\ell}+b-m_1)^{-1} (-t_1)^{\kappa_{k}-a+\kappa_{\ell}-b-1}t_1^{m_1-1}=\int_{0}^{t_1}(-v)^{\kappa_{k}-a+\kappa_{\ell}-b-2}v^{m_1-1}\mathrm{d}v.
\end{equation*}
Then, we apply the binomial theorem twice, first to the sum with index \(a\), then the sum with index \(b\).
\end{proof}
\noindent 
We return to proving Lemma \ref{lemma: intermediate lem 1}. To do so, we need the following consequence of the Cram\'er-Wold theorem.
\begin{theorem}
\label{theorem: cramer wold}
Let \(a < b\) be real. Let \((\bs X_{n1},...,\bs X_{nk})_{n\geq 1}\) be sequences of processes in \(D[a,b]\) that take values in \(\mathbb{C}^d\). Assume for some processes \(\bs W_1,...,\bs W_k\) in \(D[a,b]\), and all possible \(k\)-tuples \((c_1,...,c_k)\in \mathbb{R}^k\),
\begin{equation*}
    \sum_{i=1}^k c_i \bs X_{ni}(t) \conindis \sum_{i=1}^k c_i \bs W_i(t) \text{ in } D[a,b]
\end{equation*}
as \(n \rightarrow \infty\). Then,
\begin{equation*}
    (\bs X_{n1}(t),...,\bs X_{nk}(t)) \conindis (\bs W_1(t),...,\bs W_k(t)) \text{ in \(D[a,b]^k\)}
\end{equation*}
as \(n \rightarrow \infty\).
\end{theorem}
\begin{proof}
By assumption, we have, for each \(1\leq i \leq k\),
\begin{equation*}
    \bs X_{ni}(t) \conindis \bs W_{i}(t) \text{ in \(D[a,b]\)}.
\end{equation*}
Therefore, each of the \((\bs X_{ni})_{n\geq 1}\) are tight in \(D[a,b]\). This implies \( (\bs X_{n1},...,\bs X_{nk})_{n\geq 1}\) is tight in \(D[a,b]^k\). We are left to show convergence of the finite dimensional distributions. Let \(p\geq 1\) and \(a\leq t_1\leq ... \leq  t_p \leq b\). By assumption, we have convergence of the finite dimensional distributions of \(\sum_{i=1}^kc_i \bs X_{ni}\), therefore as \(n \rightarrow \infty\)
\begin{equation*}
    \F(\sum_{i=1}^k c_i \bs X_{ni}(t_1),...,\sum_{i=1}^k c_i \bs X_{ni}(t_p)\R) \conindis \F(\sum_{i=1}^k c_i \bs W_i(t_1),...,\sum_{i=1}^k c_i \bs W_i(t_p)\R).
\end{equation*}
The above implies, for all \(d_1,...,d_p \in \mathbb{R}\),
\begin{equation}
\label{eq: Convergence 3}
    \sum_{i=1}^k\sum_{j=1}^p d_j c_i \bs X_{ni}(t_j) \conindis \sum_{i=1}^k\sum_{j=1}^p d_j c_i \bs W_{i}(t_j)
\end{equation}
as \(n \rightarrow \infty\). Since the \((c_i)_{1\leq i \leq k}\) and \((d_j)_{1\leq j \leq p}\) were arbitrarily chosen, \eqref{eq: Convergence 3} is equivalent to the following result. For all \(pk\)-tuples, \((\hat{c}_1,...,\hat{c}_{pk})\in \mathbb{R}^{pk}\), we have
\begin{equation*}
    \sum_{i=1}^k\sum_{j=1}^p \hat{c
    }_{i+(j-1)k}\bs X_{ni}(t_j) \conindis  \sum_{i=1}^k\sum_{j=1}^p \hat{c
    }_{i+(j-1)k}\bs W_{i}(t_j)
\end{equation*}
as \(n \rightarrow \infty\). This is the required assumption to apply the Cram\'er-Wold theorem to give that
\begin{equation*}
 (\bs X_{n1}(t_1),....,\bs X_{nk}(t_1),...,\bs X_{n1}(t_p),...,\bs X_{nk}(t_p)) \conindis  (\bs W_{1}(t_1),....,\bs W_{k}(t_1),...,\bs W_{1}(t_p),...,\bs W_{k}(t_p))
\end{equation*}
as \(n \rightarrow \infty\). This is exactly convergence of the finite dimensional distributions, so since we have already shown tightness, functional convergence immediately follows.
\end{proof}
\begin{proof}[Proof of Lemma \ref{lemma: intermediate lem 1}]For arbitrary \(d_1,...,d_q \in \mathbb{R}\), let
\begin{equation}
\bs Z_n(t) = N^{-1/2}\sum_{k=1}^qd_k\omega^{-(m_1+2\kappa_k-2)/2}Q_k(\e^{-A\omega t}\bs X_n(\omega t)-N\bs \mu), \quad t\geq 0.\label{eq: int convergence}
\end{equation}
For ease of notation, let
\begin{equation}
    Q = \sum_{k=1}^qd_k\omega^{-(m_1+2\kappa_k-2)/2}Q_k. \label{eq: improve 15}
\end{equation}
We are going to show
\begin{equation}
    \bs Z_n(t) \conindis \sum_{k=1}^q d_k\bs W_k(t) \text{ in } D[0,\infty) \label{eq: need cramer apply}
\end{equation}
as \(n \rightarrow \infty\). Then, since the \((d_k)_{k=1}^q\) are arbitrary, the theorem will follow by \eqref{eq: need cramer apply} and Theorem \ref{theorem: cramer wold}. By Lemma \ref{prop: martingale bp}, for each \(n\), \(\bs Z_n\) is a martingale started from 0. To show \eqref{eq: need cramer apply}, we are going to apply Lemma \ref{proposition: Janson 9.1} to this sequence of martingales.
\\~\\
\underline{Condition 1 of Lemma \ref{proposition: Janson 9.1}:} We start by showing Condition 1. By definition, for any \(0\leq t_1 \leq t_2\),
\begin{equation*}
    \Cov\F(\sum_{k=1}^q d_k \bs W_k(t_2),\sum_{\ell=1}^q d_{\ell} \overline{\bs W}_{\ell}(t_1)\R) = \sum_{k,\ell=1}^q \bs 1_{\{\lambda_{ck}=\lambda_{c\ell}\}}d_kd_{\ell} F(t_1,\kappa_k,\kappa_{\ell})  N_A^{\kappa_{k}-1}Q_kB Q_{\ell}^*N_A'^{\kappa_{\ell}-1}.
\end{equation*}
Therefore, we need to show, for each \(t \geq 0\),
\begin{equation}
    [\bs Z_n,\overline{\bs Z}_n]_t \coninprob  \sum_{k,\ell=1}^q \bs 1_{\{\lambda_{ck}=\lambda_{c\ell}\}}d_kd_{\ell} F(t,\kappa_k,\kappa_{\ell})  N_A^{\kappa_{k}-1}Q_kB Q_{\ell}^*N_A'^{\kappa_{\ell}-1} \label{eq: quad cov need to show}
\end{equation}
as \(n \rightarrow \infty\). (Note, since \(F(\cdot, \kappa_k,\kappa_m)\) is continuous, the r.h.s.\ is continuous in \(t\), and thus a valid covariance function for Lemma~\ref{proposition: Janson 9.1}.) Since a time-independent shift does not change the quadratic covariation, for each \(n\), we have 
\begin{equation}
    [\bs Z_n,\overline{\bs Z}_n]_t = N^{-1}\sum_{i=1}^d\sum_{k:\tau_{nik}\leq t}Q\e^{-A\tau_{nik}}\Delta\bs X_n(\tau_{nik})\Delta \bs X_n(\tau_{nik})' \e^{-A'\tau_{nik}}Q^*, \quad t \geq 0, \label{eq: improve 12}
\end{equation}
where, for \(1 \leq i \leq d\), \((\tau_{nik})_{k\geq 1}\) are the ordered death times of particles of type \(i\) in \(\bs X_n\). If, in Lemma \ref{lemma: Jan 9.3}, we take \(p=1\), \(M_1(t) = M_2(t)^*= N^{-1/2}Q \mathrm{e}^{-At}\), and \(N(\bs x) = \bs x \bs x'\), then the first term on the r.h.s.\ of \eqref{eq: improve 1} is equal to
\[N^{-1}\sum_{k:\tau_{nik}\leq t}  Q \mathrm{e}^{-A\tau_{nik}}\Delta \bs X_n(\tau_{nik}) \Delta \bs X(\tau_{nik})'\mathrm{e}^{-A'\tau_{nik}}Q^*, \quad t\geq 0.\] 
This, Lemma \ref{lemma: Jan 9.3}, and \eqref{eq: improve 12} imply
\begin{equation}
    Y_n(t) := [\bs Z_n,\overline{\bs Z}_n]_t-N^{-1}\sum_{i=1}^d\int_{0}^{\omega t}Q\e^{-Av}\E[\bs \xi_i \bs \xi_i'] \e^{-A'v}Q^*a_iX_n(v)_i \mathrm{d}v, \quad t\geq 0 \label{eq: improve 14}
\end{equation}
is a matrix-valued martingale for each \(n\). By \eqref{eq: improve 14}, we can prove \eqref{eq: quad cov need to show} by showing, for each \(t \geq 0\),
\begin{align}
    & Y_n(t) \coninprob 0,\label{eq: showing quad variation 1}\\
    & U_n(t) :=N^{-1}\sum_{i=1}^d\int_{0}^{\omega t}Q\e^{-Av}\E[\bs \xi_i \bs \xi_i'] \e^{-A'v}Q^*a_i(X_n(v)_i-N(\e^{Av}\bs \mu)_i) \mathrm{d}v \coninprob 0,\label{eq: showing quad variation 2}\\
    & V_n(t):= \sum_{i=1}^d\int_{0}^{\omega t}Q\e^{-Av}\E[\bs \xi_i \bs \xi_i'] \e^{-A'v}Q^*a_i(\e^{Av}\bs \mu)_i \rightarrow \sum_{k,\ell=1}^d \bs 1_{\{\lambda_{ck}=\lambda_{c\ell}\}}d_kd_{\ell} F(t,\kappa_k,\kappa_{\ell})  N_A^{\kappa_{k}-1}Q_kB Q_{\ell}^*N_A'^{\kappa_{\ell}-1}\label{eq: showing quad variation 3}
\end{align}
as \(n \rightarrow \infty\). We start with \eqref{eq: showing quad variation 1}. Since \(Y_n(0) = 0\), by \eqref{eq: second moment martingale}, we have for \(t \geq 0\),
\begin{equation}
    \E[Y_n(t)\overline{Y}_n(t)']=\E[Y_n,\overline{Y}_n]_t. \label{eq: lcritical 11}
\end{equation}
Since the second term of \(Y_n\) does not contribute to its quadratic variation (since it is continuous and of finite variation), we have
\begin{equation}
  [Y_n,\overline{Y}_n]_t = \F[[\bs Z_n,\overline{\bs Z}_n],\overline{[\bs Z_n,\overline{\bs Z}_n]}\R]_t= N^{-2}\sum_{i=1}^d\sum_{k:\tau_{nik}\leq t}(Qe^{-A\tau_{nik}}\Delta \bs X_n(\tau_{nik}) \Delta \bs X_n(\tau_{nik})' e^{-A'\tau_{nik}}Q^*)^2, \quad t \geq 0,\label{eq: improve 13}
\end{equation}
where in the final equality, we have used that the jumps of the quadratic variation are hermitian. By the same argument that gave \eqref{eq: improve 14}, we have, for any \(t \geq 0\),
\begin{equation}
    \E[Y_n,\overline{Y}_n]_t=N^{-2}\sum_{i=1}^d\int_{0}^{\omega t}\E\F[\F(Qe^{-Av}\bs \xi_i \bs \xi_i' e^{-A'v}Q^*\R)^2\R]\mathbb{E}[a_iX_n(v)_i] \mathrm{d}v. \label{eq: lcritical 1}
\end{equation}
(In Lemma \ref{lemma: Jan 9.3}, we have taken \(p=2\), \(M_{11}(t)=M_{12}(t)=M_{21}(t)^*=M_{22}(t)^*=N^{-1/2}Q \mathrm{e}^{-At}\), \(N_1(\bs x)=N_2(\bs x)=\bs x \bs x'\).) By \eqref{eq: matrix exp bounds} and \eqref{eq: matrix exp bounds all} to bound the matrix exponential of \(A\), \eqref{eq: CMBP First moment} for the first moment of the branching process, and (A1), we get, for any \(1 \leq k,\ell \leq q\),
\begin{align*}
&N^{-2}\omega^{-2(m_1+\kappa_{k}+\kappa_{\ell}-2)}\sum_{i=1}^d\int_{0}^{\omega t}\E\F[\F\|\F(d_{k}Q_{k}e^{-Av}\bs \xi_i \bs \xi_i' e^{-A'v}d_{\ell}Q_{\ell}^*\R)^2\R\|_2\R]\mathbb{E}[a_iX_n(v)_i] \mathrm{d}v \\
&\qquad\qquad\qquad\qquad\qquad\qquad\qquad\qquad\qquad\qquad\qquad\leq \cst  N^{-1} \omega^{-2(m_1+\kappa_{k}+\kappa_{\ell}-2)}\int_{0}^{\omega t}\e^{-\lambda_1 v}v^{2(\kappa_{k}+\kappa_{\ell}-2)+(m_1-1)}\mathrm{d}v\\
&\qquad\qquad\qquad\qquad\qquad\qquad\qquad\qquad\qquad\qquad\qquad\leq \cst N^{-1}\omega^{-2(m_1+\kappa_{k}+\kappa_{\ell}-2)} \rightarrow 0
\end{align*}
as \(n \rightarrow \infty\). This and \eqref{eq: improve 15} imply 
\begin{equation*}
N^{-2}\sum_{i=1}^d\int_{0}^{\omega t}\E\F[\F\|\F(Qe^{-Av}\bs \xi_i \bs \xi_i' e^{-A'v}Q^*\R)^2\R\|_2\R]\mathbb{E}[a_iX_n(v)_i] \mathrm{d}v  \rightarrow 0
\end{equation*}
as \(n \rightarrow \infty\). Using this in \eqref{eq: lcritical 1}, and then in \eqref{eq: lcritical 11}, we get that the second moment of \(\|Y_n(t)\|_2\) converges to 0 as \(n \rightarrow \infty\). Thus, \eqref{eq: showing quad variation 1} holds. We move on to showing \eqref{eq: showing quad variation 2}. For \(1\leq k,\ell \leq q\), let
\begin{equation*}
    U_{nk\ell}(t) := N^{-1}\omega^{-(m_1+\kappa_{k}+\kappa_{\ell}-2)}\sum_{i=1}^d\int_{0}^{\omega t}d_{k}Q_ke^{-Av}\E[\bs \xi_i \bs \xi_i'] e^{-A'v}d_{\ell}Q_{\ell}^*a_i(X_n(v)_i-N(\e^{Av}\bs \mu)_i) \mathrm{d}v, \quad t \geq 0.
\end{equation*}
Using \eqref{eq: matrix exp bounds} to bound the matrix exponential of \(A\), we get for \(1 \leq k,\ell \leq q\), \(t \geq 0\),
\begin{equation}
    \E[\|U_{nk\ell}(t)\|_2]\leq \cst N^{-1}\omega^{-(m_1+\kappa_{k}+\kappa_{\ell}-2)}\int_{0}^{\omega t} v^{\kappa_{k}+\kappa_{\ell}-2}e^{-\lambda_1v}\E[\|\bs X_n(v)-Ne^{Av}\bs\mu\|_2]\mathrm{d}v. \label{eq: lcritical 2}
\end{equation}
We can bound the expectation in the integrand by \eqref{eq: fourth moment main} and Jensen's inequality, which gives, for all \(v \geq 0\),
\begin{equation*}
    \E[\|\bs X_n(v)-Ne^{Av}\bs\mu\|_2]\leq \E[\|\bs X_n(v)-Ne^{Av}\bs\mu\|_2^4]^{1/4}\leq \cst N^{1/2} (1+v)^{m_1-1}\e^{\lambda_1 v} .
\end{equation*}
Using this in \eqref{eq: lcritical 2} gives 
\begin{equation}
  \E[\|U_{nk\ell}(t)\|_2] \leq \cst_t N^{-1/2}  \rightarrow 0 \label{eq: Vij2}
\end{equation}
as \(n \rightarrow \infty\). This and the triangle inequality imply \(\E[\|U_n(t)\|_2]\rightarrow 0\) as \(n \rightarrow \infty\), therefore \eqref{eq: showing quad variation 2} holds. We move onto showing \eqref{eq: showing quad variation 3}. By \eqref{eq: projection}, for \(1 \leq k,\ell \leq q\), and \(v \geq 0\),
\begin{equation*}
   Q_k\e^{-Av}=\e^{-\lambda_{ck} v}\sum_{a=0}^{\kappa_{k}-1}\frac{(-vN_A)^{a}}{a!} Q_k, \quad \overline{Q}_{\ell}\e^{-Av}=\e^{-\overline{\lambda}_{c\ell} v}\sum_{b=0}^{\kappa_{\ell}-1}\frac{(-vN_A)^{b}}{b!}\overline{Q}_{\ell}.
\end{equation*}
This implies 
\begin{equation}
    V_{n}(t):= \sum_{k,\ell=1}^q \sum_{a=0}^{\kappa_k-1}\sum_{b=0}^{\kappa_{\ell}-1} V_{nk\ell ab}(t), \quad {t \geq 0},  \label{eq: improve 223}
\end{equation}
where 
\begin{equation*}
    V_{nk\ell ab}(t):= (a!b!)^{-1}\omega^{-(m_1+\kappa_{k}+\kappa_{\ell }-2)}d_kd_{\ell} \sum_{i=1}^d\int_{0}^{\omega t}(-v)^{a+b}\e^{-(\lambda_{ck}+\overline{\lambda}_{c\ell })v}N_A^{a}Q_k \E[\bs \xi_i \bs \xi_i']Q_{\ell }^*N_A'^ba_i(\e^{Av}\bs \mu)_i \mathrm{d}v, \quad {t \geq 0}.
\end{equation*}
The asymptotic behaviour of \(V_{nk\ell ab}\) as \(n\rightarrow \infty\) can be split into three cases: For \((a,b) \neq (\kappa_{k}-1,\kappa_{\ell}-1)\), and~\(t \geq 0\), we get
\begin{align}
\|V_{nk\ell ab}(t)\|_2 &\leq  \cst \omega^{-(m_1+\kappa_{k}+\kappa_{\ell}-2)}\sum_{i=1}^d\int_{0}^{\omega t} |(-v)^{a+b}(1+v)^{m_1-1}e^{(\lambda_1-\lambda_{ck}-\overline{\lambda}_{c\ell})v}| \mathrm{d}v  \\
& \leq \cst \omega^{-(m_1+\kappa_{k}+\kappa_{\ell}-2)}\int_{0}^{\omega t}(1+v)^{a+b+m_1-1}\mathrm{d}v \rightarrow 0 \label{eq: impove 222}
\end{align}
as \(n \rightarrow \infty\), where (A1) and \eqref{eq: matrix exp bounds} are used to bound the integrand in the first inequality. For \((a,b)=(\kappa_{k}-1,\kappa_{\ell}-1)\),  \(\lambda_{ck}\neq \lambda_{c\ell}\), and \(t \geq 0\), we get
\begin{equation}
 \|V_{nk\ell ab}(t)\|_2 \leq \cst \omega^{-(m_1+\kappa_{k}+\kappa_{\ell}-2)}\sum_{i=1}^d\F|\int_{0}^{\omega t} (-v)^{\kappa_{k}+\kappa_{\ell}-2}(1+v)^{m_1-1}e^{(\lambda_1-\lambda_{ck}-\overline{\lambda}_{c\ell})v}\mathrm{d} v\R| \rightarrow 0 \label{eq: impove 221}
\end{equation}
as \(n \rightarrow \infty\), where (A1) and \eqref{eq: matrix exp bounds} are used to bound the integrand. Note, the convergence holds since \(\lambda_1-\lambda_{ck}-\overline{\lambda}_{c\ell}\in \mathbb{C}/ \mathbb{R}\), so the integral is of order \(\omega^{m_1+\kappa_{k}+\kappa_{\ell}-3}\) as \(n\rightarrow \infty\). Lastly, take \((a,b)=(\kappa_{k}-1,\kappa_{\ell}-1)\) and \(\lambda_{ck} = \lambda_{c\ell}\). For \(1\leq i \leq d\), define~\(R_i(v)\) through the equation
\begin{equation}
    (\e^{Av}\bs \mu)_i =  \e^{\lambda_1v}v^{m_1-1} v(\bs \mu)_i (1+ R_i(v)), \quad v \geq 0.
\end{equation}
By \eqref{eq:identity} and \eqref{eq: matrix exp bounds}, for \(1\leq i \leq d \), \(R_i(v) \rightarrow 0\) as \(v\rightarrow \infty\). Thus, for \(t \geq 0\), we get
\begin{align}
V_{nk \ell ab}(t) &=  ((\kappa_k-1)!(\kappa_{\ell}-1)!)^{-1}\omega^{-(m_1+\kappa_{k}+\kappa_{{\ell}}-2)}d_kd_{\ell}\times...\\
&\quad ...\times  \sum_{i=1}^d\int_{0}^{\omega t}(-v)^{\kappa_k+\kappa_{\ell}-2}v^{m_1-1}N_A^{\kappa_k-1}Q_k a_i v(\bs \mu)_i\E[\bs \xi_i \bs \xi_i'] Q_{\ell}^*N_A'^{\kappa_{\ell}-1}(1+R_i(v)) \mathrm{d}v \nonumber \\
& \rightarrow d_kd_{\ell} F(t,\kappa_k,\kappa_{\ell}) N_A^{\kappa_{k}-1}Q_kB Q_{\ell}^*N_A'^{\kappa_{\ell}-1},  \label{eq: improve 22} 
\end{align}
as \(n \rightarrow \infty\). Using \eqref{eq: impove 222}, \eqref{eq: impove 221}, and \eqref{eq: improve 22} in \eqref{eq: improve 223} gives \eqref{eq: showing quad variation 3}. Therefore, \eqref{eq: quad cov need to show} holds which is Condition 1 of Lemma \ref{proposition: Janson 9.1}.
\\~\\
\underline{Condition 2 of Lemma \ref{proposition: Janson 9.1}:}
For this condition, we need to show, for each \(t \geq 0\),
\begin{equation*}
    [\bs Z_n,\bs Z_n]_t \coninprob  \sum_{k,\ell=1}^d \bs 1_{\{\lambda_{ck}=\overline{\lambda}_{c\ell}\}}d_kd_{\ell} F(t,\kappa_k,\kappa_{\ell}) N_A^{\kappa_{k}-1}Q_kB Q_{\ell}'N_A'^{\kappa_{\ell}-1}
\end{equation*}
as \(n \rightarrow \infty\). The arguments to show this are identical to the arguments used to show Condition 1 with \(\overline{\bs Z}_n\) and \(\overline{Q}\) replaced by~\(\bs Z_n\) and \(Q\). We omit the details, but note in \eqref{eq: improve 22} the non-zero terms come from when \(\lambda_{ck}=\overline{\lambda}_{c\ell}\) instead of \(\lambda_{ck} = \lambda_{c\ell}\). 
\\~\\
\underline{Condition 3 of Lemma \ref{proposition: Janson 9.1}:} This follows by Lemma \ref{Lemma: fourth moments} for critical components and Jensen's inequality, since for all \(n\geq 1\), and each \(t \geq 0\),
\begin{equation*}
  \E[\|\bs Z_n(t)\|_2^2]\leq \E[\|\bs Z_n(t)\|_2^4]^{1/2}\leq \cst_t.
\end{equation*}
Therefore, by Lemma \ref{proposition: Janson 9.1}, \eqref{eq: need cramer apply} holds. Thus, the lemma holds by Theorem \ref{theorem: cramer wold} as claimed.
\end{proof}
\subsection{Proof of Theorem \ref{Theorem: Main continuous time} LT\textsubscript{\(\ell\)}}
Since the Jordan projection matrices are continuous maps (as they are linear) and for any Jordan blocks \(J_1\neq J_2\), we have \(P_{J_1}P_{J_2}=0\), it suffices to show a functional convergence result for the MCBP under the projection \(P_{\ell}=\sum_{\lambda \in \Lambda_{\ell}}\sum_{J \in \mathcal{J}_{\lambda}}P_J\). From here, one can recover the behaviour for a single Jordan block by applying the continuous mapping theorem.
\begin{lemma}
\label{lemma: big time}
Take the setting of Theorem \ref{Theorem: Main continuous time} LT.  Let \(P_{\ell}=\sum_{\lambda \in \Lambda_{\ell}}\sum_{J \in \mathcal{J}_{\lambda}}P_J\). Then, as \(n \rightarrow \infty\),
        \begin{equation}
            N^{-1/2}P_{\ell}\left(\e^{-A\omega t}\bs X_n(\omega t)-N\bs \mu\right) \conindis \bs V_{\ell}  \text{ in }  D(0,\infty). \label{eq: improve 31}
        \end{equation}
        The random variable \(\bs V_{\ell}\) is mean-zero, real, and Gaussian. We have
\begin{align}
& \Var(\bs V_{\ell})=\sum_{i=1}^d  \int_{0}^{\infty}P_{\ell}\mathrm{e}^{-Av}\E[\bs \xi_i \bs \xi_i']\mathrm{e}^{-A'v}P_{\ell}'(\mathrm{e}^{Av}\bs \mu)_i\mathrm{d}v. \label{eq: improve 50}
\end{align}
\end{lemma}
\begin{proof}
Take \(0<a<1\leq b\) and let
\begin{align*}
&\bs Z_n(t) := N^{-1/2}P_{\ell}( \e^{-A\omega t}\bs X_n(\omega t)-\e^{-A\omega a }\bs X_n(\omega a )), \quad t\geq a,\\
& \bs Y_n :=N^{-1/2}P_{\ell}(\e^{-A\omega a } \bs X_n(\omega a) -N\bs \mu).
\end{align*}
We are going to show 
\begin{align}
    &\bs Z_n(t) \coninprob 0 \text{ in } D[a,b], \label{eq: slutsk 1}\\
    &\bs Y_n \conindis \bs V_{\ell}, \label{eq: slutsk 2}
\end{align}
as \(n \rightarrow \infty\). Then, the lemma will follow by combining \eqref{eq: slutsk 1} and \eqref{eq: slutsk 2} with Slutsky's lemma and sending \(a \downarrow 0\), \(b \uparrow \infty\). We start by showing \eqref{eq: slutsk 1}. By Lemma \ref{prop: martingale bp}, for each \(n\), the process \(\bs Z_n\) is a martingale initiated from 0. Hence, by Doob's martingale inequality,
\begin{equation}
    \mathbb{P}\F(\sup_{a \leq t \leq b}\|\bs Z_n(t)\|_2 > c\R) \leq \frac{\mathbb{E}[\|\bs Z_n(b)\|_2^2]}{c^2}. \label{eq: doobs inequality}
\end{equation}
To show \eqref{eq: slutsk 1}, we show that the r.h.s.\ of \eqref{eq: doobs inequality} tends to \(0\) for all \(c > 0\). By the strong Markov property \eqref{eq: strong markov prop}, we can apply Lemma \ref{Lemma: fourth moments} for large components to the branching process \((\bs X_n(\omega a + t))_{t\geq 0}\) conditionally on \(\bs X_n(\omega a)\). This gives
\begin{align*}
 \mathbb{E}[\|\bs Z_n(b)\|_2^2]& \leq \mathbb{E}[\E[\|\bs Z_n(b)\|_2^4|\bs X_n(\omega a)]^{1/2}]\\
 &= \E[\mathbb{E}[\|N^{-1/2}P_{\ell}( \e^{-A\omega t}\bs X_n(\omega t)-\e^{-A\omega a }\bs X_n(\omega a ))\|_2^4|\bs X_n(\omega a)]^{1/2}]\\
 &\leq \cst \|P_{\ell}\e^{-A\omega a}\|_2^2\E[\mathbb{E}[\|N^{-1/2}P_{\ell}( \e^{-A\omega (t-\alpha)}\bs X_n(\omega t)-\bs X_n(\omega a ))\|_2^4|\bs X_n(\omega a)]^{1/2}]\\
 &\leq \cst \|P_{\ell}\e^{-A\omega a}\|_2^2\E[N^{-1} \|\bs X_n(\omega a)\|_2]\\
 &\leq \cst \|P_{\ell}\e^{-A\omega a}\|_2^2\e^{\lambda_1 \omega a}(1+\omega a)^{m_1-1} \rightarrow 0
\end{align*}
as \(n \rightarrow \infty\). In the first line, we have used Jensen's inequality. In the fourth line, we have applied Lemma \ref{Lemma: fourth moments}. In the inequality in the final line, we have used Jensen's inequality followed by \eqref{eq: fourth moment main}, and in the convergence in the final line, we have used that, since \(P_{\ell}\) consists of only large projections, \eqref{eq: matrix exp bounds} implies \(\|P_{\ell}\e^{-A\omega a}\|_2\leq \cst \e^{-\lambda \omega a}\), where \(2\lambda >\lambda_1\). Therefore, for any~\(c > 0\), the r.h.s.\ of \eqref{eq: doobs inequality} tends to 0 as \(n \rightarrow \infty\). This implies \eqref{eq: slutsk 1} as claimed. We now show \eqref{eq: slutsk 2}. By the branching property \eqref{equ:branching property},
\begin{equation*}
    \bs Y_n= N^{-1/2}\sum_{i=1}^d\sum_{j=1}^{N\mu_i}P_{\ell}(\e^{-A\omega a } \bs X_{ijn}(\omega a)-\bs e_i),
\end{equation*}
where the \(\bs X_{ijn}\) are independent branching processes with initial conditions \(\bs e_i\), and identical replacement structures \((\bs \alpha, (\bs \xi_i)_{i=1}^d)\). The summands in the r.h.s.\ are a triangular array of mean-zero independent random variables indexed by \(n\). We are going to apply Lemma \ref{Lemma: MCLT} to this array. By \eqref{eq: MCBP var}, we have that, for all \(1\leq i \leq d,1\leq j \leq N\mu_i\),
\begin{equation*}
    \lim_{n \rightarrow \infty}\Var(P_{\ell}\e^{-A\omega a } \bs X_{ijn}(\omega a)) = \sum_{k=1}^d \int_{0}^{\infty}P_{\ell}\e^{-Av}\E[\bs \xi_k \bs \xi_k']\e^{-A'v}P_{\ell}'a_k( \e^{Av}\bs e_i)_k\mathrm{d}v.
\end{equation*}
Indeed the integral on the r.h.s.\ is convergent, since the eigenvalue \(\lambda\) of any large component satisfies \(\mathrm{Re}2\lambda >\lambda_1\), the integrand decays exponentially by \eqref{eq: matrix exp bounds}. Thus, Condition 2 of Lemma \ref{Lemma: MCLT} holds with
\begin{equation*}
   \lim_{n\rightarrow \infty} N^{-1}\sum_{i=1}^d\sum_{j=1}^{N\mu_i}\Var(P_{\ell}\e^{-A\omega a } \bs X_{ijn}(\omega a)) = \sum_{k=1}^d  \int_{0}^{\infty}P_{\ell}\e^{-Av}\E[\bs \xi_k \bs \xi_k']\e^{-A'v}P_{\ell}'a_k( \e^{Av}\bs \mu)_k\mathrm{d}v=\Var(\bs V_{\ell}).
\end{equation*}
We get Condition 2 of Lemma \ref{Lemma: MCLT} with \(p=4\) as a consequence of Lemma \ref{Lemma: fourth moments} for large components, since 
\begin{equation*}
    N^{-2}\sum_{i=1}^d \sum_{j=1}^{N\mu_i}\E[\|P_{\ell}(\e^{-A\omega a}\bs X_{ijn}(\omega a)-\bs e_i)\|_2^4]\leq \cst N^{-1}\rightarrow 0
\end{equation*}
as \(n \rightarrow \infty\). Therefore, by Lemma \ref{Lemma: MCLT}, \eqref{eq: slutsk 2} holds, which concludes the proof. 
\end{proof}
\begin{proof}[Proof of Theorem \ref{Theorem: Main continuous time} LT\textsubscript{\(\ell\)}]
Apply the continuous mapping theorem with function \(P_J\) to the result of Lemma \ref{lemma: big time}
\end{proof}
\subsection{Proof of \eqref{eq: first order MCBP}}

This immediately follows by applying Slutsky's lemma to Theorem \ref{Theorem: Main continuous time} LT\textsubscript{c} and LT\textsubscript{\(\ell\)}, \eqref{eq: small component convergence in probability}, and using \eqref{eq:identity}. The only fluctuations that do not vanish under the renormalization used is those coming from Jordan blocks with eigenvalue \(\lambda_1\) and size \(m_1\). 

\subsection{Proof of Theorem \ref{theorem: cont time 2}}
The proof of this theorem is given in Appendix \ref{Appendix: proof of covar} 
\section{Proof of Theorems \ref{theorem: main result simplified}, \ref{theorem: Main results discrete}, and \ref{theorem: main results discrete 2}}
\label{Sec: proof of discrete time}
In this section, we prove Theorems \ref{theorem: main result simplified}, \ref{theorem: Main results discrete}, and \ref{theorem: main results discrete 2} by applying the embedding \eqref{eq: branchurn} to Theorems \ref{Theorem: Main continuous time} and \ref{theorem: cont time 2}. To apply this embedding, we need the following result on random time-changes of stochastic processes.
\begin{theorem}[Section 14 of \cite{Bill}]
\label{theorem: bill 2}
Let \(a <b\), \(c<d\) be real numbers. Let \((\phi_n)_{n\geq 1}\) be a sequence of non-decreasing processes in \(D[a,b]\) that take values in the interval \([c,d]\). Also, let \((\bs Y_n)_{n\geq 1}\) be a sequence of processes in \(D[c,d]\) that take values in~\(\mathbb{C}^d\). Then, if \((\bs Y_n,\bs \phi_n)\conindis (\bs Y,\phi)\) in \(D[a,b]\times D[c,d]\) as \(n \rightarrow \infty\), where \(\bs Y\) is continuous a.s., we have that
\begin{equation*}
    \bs Y_n(\phi_n(t)) \conindis \bs Y(\phi(t)) \text{ in \(D[a,b]\)}.
\end{equation*}
\end{theorem}
\noindent We will use Kolmogorov's continuity theorem to check the a.s.\ continuity condition in Theorem \ref{theorem: bill 2}.
\begin{theorem}[Kolmogorov's continuity theorem]
\label{theorem: Kol cont}
Let \(a < b\) be real numbers. Let \(\bs Y\) be a random process in \(D[a,b]\) that takes values in \(\mathbb{C}^d\). If there exists positive constants \(\alpha,\beta,K\), such that, for all \(a \leq t_1 \leq t_2 \leq b\),
\begin{equation*}
    \E[\|\bs Y(t_2)-\bs Y(t_1)\|_2^{\alpha}] \leq K(t_2-t_1)^{1+\beta}.
\end{equation*}
Then \(\bs Y\) is a.s.\ continuous in \(D[a,b]\).
\end{theorem}
\noindent 
The remainder of this section is split up as follows: Firstly, we show the functional convergence results in Theorem \ref{theorem: main result simplified} IBD \& TR and Theorem \ref{theorem: Main results discrete}. In each case, we show the limit process satisfies the equality in distribution given in Theorem \ref{theorem: main results discrete 2}. Then, we show Theorem \ref{theorem: main result simplified} TSD using Theorem \ref{theorem: main results discrete 2} and Slutsky's lemma. 
\subsection{Proof of Theorems \ref{theorem: main result simplified} IBD \& TR, \ref{theorem: Main results discrete}, and \ref{theorem: main results discrete 2}}
 By \eqref{eq: branchurn}, there exists a sequence of branching processes \((\bs X_n)_{n\geq1}\), such that for each \(n\),
\begin{equation}
(\bs X_n(\tau_{n,i}))_{i\geq 0}\eqindis (\bs U_n(i))_{i\geq 0},
\label{eq: branchurn2}
\end{equation}
where the \((\tau_{n,i})_{i\geq 0}\) are the ordered death times of particles in \(\bs X_n\) (where we take \(\tau_0=0\)). The first part of the proof is to show the random time change \((\tau_{n,i})_{i\geq 0}\) converges in probability to some deterministic process in \(D\). Once this is shown, we can input this random time change, \eqref{eq: branchurn2}, and the MCBP results of Theorem \ref{Theorem: Main continuous time} into Theorem \ref{theorem: bill 2} to acquire the desired functional limit theorem for the urn. Furthermore, this limit is shown to satisfy the equality in distribution given in Theorem \ref{theorem: main results discrete 2}. Since the growth of the \((\tau_{n,i})_{i\geq 0}\) depends on how \(N\) scales with \(n\), each regime results in a different limit for the random time change \((\tau_{n,i})_{i\geq 0}\). Therefore, we need to prove each regime separately. Note, by Lemma \ref{lemma: import ob}, all Jordan blocks with eigenvalue \(S\) have size 1. In particular, this means \(m_1=1\) in Theorem \ref{Theorem: Main continuous time}, which should be kept in mind whenever we invoke this theorem in the following proofs.
\subsubsection{Proof of Theorems \ref{theorem: main result simplified} IBD and \ref{theorem: main results discrete 2} IBD}
\textbf{Convergence of the random time change:} We are going to show that
\begin{equation}
\label{eq: discrete 6}
     \phi_{n}(t) = {n^{-1}N\beta_1\tau_{n,\lfloor nt \rfloor}} \coninprob t \text{ in } D[0,\infty).
\end{equation}
as \(n \rightarrow \infty\). Take \(\varepsilon(n) = n/N\) in \eqref{eq: improve 32}. By taking the scalar product of \eqref{eq: improve 32} with \(\bs a\) and using Lemma \ref{lemma: balanced frob}, we get that
\begin{equation}
\label{eq: discrete 1}
    n^{-1/2}\F(\sum_{i=1}^d a_i X_n(nt/N)_i-N\e^{Snt/N}\beta_1 \R) \conindis \sum_{i=1}^d a_i W_1(t)_i\text{ in } D[0,\infty)
\end{equation}
as \(n \rightarrow \infty\), where we recall \(\beta_1 = \bs a \cdot \bs \mu\). Let \[B_n(t)=\# \{i:\tau_{n,i}\leq t\}\] be the number of deaths in \(\bs X_n\) up to time \(t\). Since \(\bs X_n\) has a balanced replacement structure, we have
\begin{equation}
\label{eq: discrete 2}
   \sum_{i=1}^d a_i X_n(nt/N)_i = N\beta_1 +SB_n(nt/N). 
\end{equation}
This and \eqref{eq: discrete 1} imply
\begin{equation}
\label{eq: discrete 3}
    n^{-1/2}\F(N\beta_1 +SB_n(nt/N)-N\e^{Snt/N}\beta_1\R) \conindis \sum_{i=1}^d a_i W_1(t)_i\text{ in } D[0,\infty)
\end{equation}
as \(n \rightarrow \infty\). This implies 
\begin{equation}
\label{eq: discrete 4}
    n^{-1}B_n(nt/N)-\beta_1 t  \coninprob 0\text{ in } D[0,\infty)
\end{equation}
as \(n \rightarrow \infty\). For \(T\geq 0\) and \(n\geq 1\), let \[A_{n,T}:=\{\tau_{n,\lfloor nT \rfloor}\leq n(T+1)/\beta_1 N\}.\] Since \(B_n(\tau_{n,i})=i\), by solving \eqref{eq: discrete 4} for \(t\) when \(B_n(nt/N)=\lfloor nT \rfloor\), we see that \(\mathbb{P}(A_{n,T})\rightarrow 1\) as \(n \rightarrow \infty\). This and the fact \eqref{eq: discrete 4} holds on the interval \([0,(T+1)/\beta_1]\), imply, on \(A_{n,T}\),
\begin{equation}
\label{eq: Convergence 4}
     n^{-1}B_n(\tau_{n,\lfloor nt \rfloor})-n^{-1}N\beta_1 \tau_{n,\lfloor nt \rfloor}\coninprob 0\text{ in } D[0,T]
\end{equation}
as \(n \rightarrow \infty\), where we have used, for any \(\varepsilon>0\),
\begin{align*}
& \mathbb{P}\F(\sup_{0\leq t \leq T}\|n^{-1}B_n(\tau_{n,\lfloor nt \rfloor})-n^{-1}N\beta_1 \tau_{n,\lfloor nt \rfloor}\|_2 \geq \varepsilon\M|A_{n,T}\R) \leq \mathbb{P}\F(\sup_{0\leq t \leq (T+1)/\beta_1}\|n^{-1}B_n(nt/N)-\beta_1 t  \|_2\geq \varepsilon\M |A_{n,T}\R)\\
 &\qquad\qquad\qquad\qquad\qquad\qquad\qquad\qquad\qquad\qquad\qquad\leq  \mathbb{P}(A_{n,T})^{-1}\mathbb{P}\F(\sup_{0\leq t \leq (T+1)/\beta_1}\|n^{-1}B_n(nt/N)-\beta_1 t  \|_2\geq \varepsilon\R)\rightarrow 0
\end{align*}
as \(n \rightarrow \infty\). Furthermore, since \(\mathbb{P}(A_{n,T}) \rightarrow 1\) as \(n \rightarrow \infty\), \eqref{eq: Convergence 4} holds without conditioning on \(A_{n,T}\). Moreover, since \(T\geq 0\) was arbitrary, \eqref{eq: Convergence 4} holds in \(D[0,\infty)\). Lastly, using that \(B_n(\tau_{n,\lfloor nt \rfloor})=\lfloor nt \rfloor\) in \eqref{eq: Convergence 4} gives \eqref{eq: discrete 6}.
\\~\\
\textbf{Functional convergence:} Take \(\varepsilon(n) = n/\beta_1N\) in \eqref{eq: improve 32}. This gives 
  \begin{equation}
  \label{eq: dis 1}
         \bs Y_n(t):=n^{-1/2}(\bs X_n(nt/\beta_1N)-N\mathrm{e}^{Ant/\beta_1N}\bs \mu)\conindis \beta_1 ^{-1/2}\bs W_1(t)\text{ in } D[0,\infty)
    \end{equation}
   as \(n \rightarrow \infty\). We are going to apply Theorem \ref{theorem: bill 2} to the \(\bs Y_n\) with time change \(\phi_n\). We need to show \(\bs W_1\) is continuous a.s., but this is immediate since \(\bs W_1\) is a Brownian motion. Let \(T \geq 0\). By definition of \(A_{n,T}\), for all \(t \in [0,T]\), we have that~\(\phi_n(t) \leq T+1\) a.s.\ on \(A_{n,T}\). Furthermore, since \(\mathbb{P}(A_{n,T})\rightarrow 1\) as \(n \rightarrow \infty\), we have that conditionally on \(A_{n,T}\) both \eqref{eq: discrete 6} holds in \(D[0,T]\) and \eqref{eq: dis 1} holds in \(D[0,T+1]\). Therefore, Theorem \ref{theorem: bill 2} implies on \(A_{n,T}\)
 \begin{equation}
 \label{eq: dis 2}
         \bs Y_n(\phi_n(t)):=n^{-1/2}(\bs U_n(\lfloor nt \rfloor )-N\mathrm{e}^{A\tau_{n,\lfloor nt \rfloor }}\bs \mu)\conindis  \beta_1^{-1/2} \bs W_1(t)\text{ in } D[0,T]
    \end{equation}
    as \(n \rightarrow \infty\). Again, since \(\mathbb{P}(A_{n,T})\rightarrow 1\) as \(n \rightarrow \infty\), this holds without conditioning on \(A_{n,T}\). Also, since \(T\geq 0\) was arbitrary, \eqref{eq: dis 2} holds in \(D[0,\infty)\). By \eqref{eq: dis 2}, the proof is complete once we have shown
    \begin{equation}
    \label{eq: dis 7}
        n^{-1/2}N(\mathrm{e}^{A\tau_{n,\lfloor nt \rfloor }}-\mathrm{e}^{AS^{-1}\log(1+Snt/\beta_1 N)})\conindis -A\beta_1^{-3/2} S^{-1}\sum_{i=1}^d a_i W_1(t)_i \text{ in } D[0,\infty)
    \end{equation}
as \(n \rightarrow \infty\). Indeed, this and \eqref{eq: dis 2} gives the fluctuation result \eqref{eq: colour comp IBD 2}, and the asymptotic colour composition \eqref{eq: colour comp IBD} follows from \eqref{eq: colour comp IBD 2}. By taking the scalar product of \eqref{eq: dis 2} with \(\bs a\) and using Lemma \ref{lemma: balanced frob}, we get
\begin{equation}
\label{eq: dis 3}
    n^{-1/2}N (1 +(\beta_1 N)^{-1}Snt-\e^{S\tau_{n,\lfloor nt \rfloor }}) \conindis \beta_1^{-3/2}\sum_{i=1}^d  a_i W_1(t)_i\text{ in }  D[0,\infty)
\end{equation}
as \(n \rightarrow \infty\), where the first two terms in the brackets appear because of the balanced property. Taking the Taylor expansion of the logarithm around \(1+(\beta_1 N)^{-1}Snt\) in the second term on the l.h.s.\ in the following gives
\begin{align}
\log(1+Snt/\beta_1 N)- \log (\e^{S\tau_{n,\lfloor nt \rfloor }}) =&  \frac{1}{1+(\beta_1 N)^{-1}Snt}( 1+(\beta_1 N)^{-1}Snt-\e^{S\tau_{n,\lfloor nt \rfloor }})\nonumber\\
&+\mathcal{O}\F((1+(\beta_1 N)^{-1}Snt-\e^{S\tau_{n,\lfloor nt \rfloor }})^2\R).\label{eq: dis 4}
\end{align}
The \(\mathcal{O}\)-term is \(\mathcal{O}_p(N^{-2}n)\) by \eqref{eq: dis 3}. Moreover, since \(N^{-1}n^{1/2}=o(1)\) (as \(n = o(N)\)), the \(\mathcal{O}\)-term is \(o_p(n^{-1/2}N^{-1})\).  This, \eqref{eq: dis 3}, \eqref{eq: dis 4}, and the fact that, for all \(t \geq 0\), \((\beta_1 N)^{-1}Snt \rightarrow 0\) as \(n \rightarrow \infty\), imply
\begin{equation}
\label{eq: dis 5}
     n^{-1/2}N (\log(1+Snt/\beta_1N)-S\tau_{n,\lfloor nt \rfloor }) \conindis \beta_1^{-3/2}\sum_{i=1}^d a_i W_1(t)_i \text{ in } D[0,\infty)   
\end{equation}
as \(n \rightarrow \infty\). We multiply the terms within the brackets of \eqref{eq: dis 5} by \(-AS^{-1}\) and take their exponential. By the same argument given to show \eqref{eq: dis 5}, only the first two terms in the exponential power series do not disappear in probability, this implies
\begin{equation}
\label{eq: dis 6}
     n^{-1/2}N(\mathrm{exp}(A\tau_{n,\lfloor nt \rfloor }-AS^{-1}\log(1+Snt/\beta_1N))-1) \conindis -A\beta_1^{-3/2} S^{-1}\sum_{i=1}^d a_i W_1(t)_i \text{ in } D[0,\infty)
\end{equation}
as \(n \rightarrow \infty\). Lastly, multiplying \eqref{eq: dis 6} by \(\e^{AS^{-1}\log(1+Snt/\beta_1N)}\), which for all \(t\geq 0\) converges to 1 (since \(n=o(N)\)), gives \eqref{eq: dis 7}.
\\~\\
The proofs for the other regimes follow a similar structure and can be found in Appendix \ref{Appendix: proof of discrete time}
\subsection{Proof of Theorem \ref{theorem: main result simplified} TSD} Theorem \ref{theorem: main result simplified} TSD immediately follows by applying \eqref{eq:identity} to the l.h.s.\ of the converging sequences in Theorem \ref{theorem: main result simplified} TSD, then using Slutsky's lemma along with the functional limits of Theorem \ref{theorem: Main results discrete}. Indeed, it is clear from Theorem \ref{theorem: Main results discrete} that 
\begin{equation*}
    \text{small components \(<\) critical components \(<\) large components,}
\end{equation*}
here we use \(<\) to mean the fluctuations in Theorem \ref{theorem: Main results discrete} TSD are of lower order as \(n\rightarrow \infty\). One can take this picture further and note that critical components are of lower order than other critical components with a larger Jordan block size. Similarly, large components are of lower order than other large components with eigenvalues of larger real part, or eigenvalues of equal real part and larger Jordan block size. Furthermore, because of \eqref{eq: TSDs big con}, small components are of lower order than critical and large components on the critical and large component time scale. The one caveat to check is that there are no fluctuations coming from the Jordan space with eigenvalue \(S\) when \(S\) is simple. Since by the hierarchy we have just presented, these would be the dominant fluctuations of the urn, which contradicts Theorem \ref{theorem: Main results discrete} TSD. Since the urn is balanced, \(S\) has right-eigenvector \(\bs a\) by Lemma \ref{lemma: balanced frob}. Again, since the urn is balanced, we have, for any \(n,m\geq 0\),
\begin{equation}
   \bs a' \bs U_n(m) = (N\beta_1+Sm). \label{eq:balanced jordan proj}
\end{equation}
Since \(S\) is simple, by \eqref{eq: eigenvector proj}, we have that the Jordan block \(J\) with eigenvalue \(S\) has projection matrix \(P_J=\bs v_1 \bs a\), where \(\bs v_1\) is the left-eigenvector of \(S\) and is normalized such that \(\bs v_1 \cdot \bs a =1\). This and \eqref{eq:balanced jordan proj} imply, for any \(t\geq 0\),
\begin{align*}
  & P_J (\bs U_n(\lfloor n t \rfloor)-N\e^{AS^{-1}\ell_1(n,t)}\bs \mu) = (N\beta_1 +S\lfloor n t \rfloor -N\beta_1\e^{\ell_1(n,t)})\bs v_1=(\lfloor n t \rfloor- n t)S\bs v_1,\\
   &P_J (\bs U_n(\lfloor N(n/N)^t \rfloor)-N\e^{AS^{-1}\ell_2(n,t)}\bs \mu) = (N\beta_1 +S\lfloor N(n/N)^t \rfloor -N\beta_1\e^{\ell_2(n,t)})\bs v_1= (\lfloor N(n/N)^t \rfloor- N(n/N)^t)S\bs v_1.
\end{align*}
Thus, the fluctuations in this Jordan space are bounded by \(S\bs v_1\) for all \(n\). This is of smaller order than any possible small/critical/large component, so they vanish in the limit. Lastly, note that the equalities in distribution given in Theorem \ref{theorem: main results discrete 2} TSD\textsubscript{c} and TSD\textsubscript{\(\ell\)} for the limits in Theorem \ref{theorem: main result simplified} TSD critical \& large urns also hold from the above argument. We are simply summing up the limits in Theorem \ref{theorem: Main results discrete} for all components with the largest Jordan block size (critical urns) or the largest real eigenvalue and Jordan block size (large urns). Note, it is fine that we only have convergence in distribution, since we know how the limits behave jointly by Theorem \ref{theorem: cont time 2}.
\\~\\
\textbf{Acknowledgement.} The author is extremely grateful to Dr C\'ecile Mailler for many conversations and much advice, which without the work in this paper would not have been possible.

The author is supported by a scholarship from the EPSRC Centre for Doctoral Training in Statistical Applied Mathematics at Bath (SAMBa), under the project EP/S022945/1.
\bibliographystyle{plain}
\bibliography{bibo.bib}

\begin{thebibliography}{10}

\bibitem{AthreyaandKarlin}
Krishna~B. Athreya and Samuel Karlin.
\newblock {Embedding of Urn Schemes into Continuous Time Markov Branching
  Processes and Related Limit Theorems}.
\newblock {\em The Annals of Mathematical Statistics}, 39(6):1801 -- 1817,
  1968.

\bibitem{Athreya}
Krishna~B. Athreya and Peter~E. Ney.
\newblock {\em Branching Processes}.
\newblock Springer-Verlag Berlin Heidelberg, 1972.

\bibitem{Bill}
Patrick Billingsley.
\newblock {\em Probability and Measure}.
\newblock John Wiley and Sons, second edition, 1986.

\bibitem{blackwell}
David Blackwell and James~B. MacQueen.
\newblock Ferguson distributions via {P}\'olya urn schemes.
\newblock {\em Ann. Statist.}, 1(2):353--355, 03 1973.

\bibitem{Borov}
Konstantin {Borovkov}.
\newblock Gaussian process approximations for multicolor {P}{\'o}lya urn
  models.
\newblock {\em arXiv e-prints}, page arXiv:1912.09665, December 2019.

\bibitem{Polya}
F.~Eggenberger and G.~Pólya.
\newblock Über die statistik verketteter vorgänge.
\newblock {\em ZAMM - Journal of Applied Mathematics and Mechanics /
  Zeitschrift für Angewandte Mathematik und Mechanik}, 3(4):279--289, 1923.

\bibitem{Friedman}
Bernard Friedman.
\newblock A simple urn model.
\newblock {\em Communications on Pure and Applied Mathematics}, 2(1):59--70,
  1949.

\bibitem{Gale}
D.~Gale and L.~S. Shapley.
\newblock College admissions and the stability of marriage.
\newblock {\em The American Mathematical Monthly}, 69(1):9--15, 1962.

\bibitem{Kolmogorov}
B.~V. Gnedenko and A.~N. Kolmogorov.
\newblock {\em Limit Distributions for Sums of Independent Random Variables}.
\newblock Cambridge, Mass: Addison-Wesley Pub. Co, 1954.

\bibitem{Hwang}
Hsien-Kuei Hwang.
\newblock Second phase changes in random m-ary search trees and generalized
  quicksort: Convergence rates.
\newblock {\em The Annals of Probability}, 31(2):609--629, 2003.

\bibitem{Janson}
Svante Janson.
\newblock Functional limit theorems for multitype branching processes and
  generalized {P}\'olya urns.
\newblock {\em Stochastic Processes and their Applications}, 110(2):177 -- 245,
  2004.

\bibitem{PouyandJan}
Svante Janson and Nicolas Pouyanne.
\newblock Moment convergence of balanced pólya processes.
\newblock {\em Electron. J. Probab.}, 23:13 pp., 2018.

\bibitem{Kesten}
H.~Kesten and B.~P. Stigum.
\newblock {Additional Limit Theorems for Indecomposable Multidimensional
  Galton-Watson Processes}.
\newblock {\em The Annals of Mathematical Statistics}, 37(6):1463 -- 1481,
  1966.

\bibitem{Markov2006}
A.A. Markov.
\newblock Extension of the law of large numbers to quantities, depending on
  each other.
\newblock {\em Izv. fizm.-mat. obsch. Kazanskom univ}, 1906.

\bibitem{muller}
Noela Müller.
\newblock Central limit theorem analogues for multicolour urn models, 2019.

\bibitem{Pouyanne}
Nicolas Pouyanne.
\newblock An algebraic approach to pólya processes.
\newblock {\em Ann. Inst. H. Poincaré Probab. Statist.}, 44(2):293--323, 04
  2008.

\end{thebibliography}
\appendix 
\section{Appendix}
\subsection{Proof of Lemma \ref{Lemma: fourth moments}}
\label{Appendix: proof of fourth moments}
To prove Lemma \ref{Lemma: fourth moments}, we need the following consequence of the Burkholder-Davis-Gundy inequality.
\begin{lemma}
\label{Lemma: BDG}
Let \((\bs M(t))_{t\geq 0}\) be a complex d-dimensional martingale in \(D[0,\infty)\). Assume that \(\bs M(0)=0\). Then, for any~\(k \geq 1\), \(t \geq 0\), 
\begin{equation*}
    \E\F[\|\bs M(t)\|_2^{2^{k+1}}\R] \leq \cst_k\E\F[ \mathrm{tr}\F([\bs M,\overline{\bs M}]_t^{2^k}\R)\R].
\end{equation*}
\end{lemma}
\begin{proof}
Firstly, by the equivalence of the euclidean norm and \(\infty\)-norm, we have, for \(t \geq 0\),
\begin{equation}
  \E\F[\|\bs M(t)\|_2^{2^{k+1}}\R]\leq \cst_k   \E\F[\F(\max_{i = 1,...,d}{ |M_i(t)|}\R)^{2^{k+1}}\R]. \label{eq: almost d 1}
\end{equation}
Next, by the Burkholder-Davis-Gundy inequality in one dimension, we have, for \(i=1,...,d\), \(t \geq 0\),
\begin{equation*}
    \E\F[|M_i(t)|^{2^{k+1}}\R]\leq \cst_k \E\F[[M_i,\overline{M}_i]_t^{2^k}\R].
\end{equation*}
This implies
\begin{equation}
\label{eq: Lemma 1.7 interim 1}
    \E\F[\F(\max_{i = 1,...,d}{ |M_i(t)|}\R)^{2^{k+1}}\R]\leq \sum_{i=1}^d \E\F[|M_i(t)|^{2^{k+1}}\R] \leq \cst_k \sum_{i=1}^d\E\F[[M_i,\overline{M}_i]_t^{2^k}\R].
\end{equation}
For any hermitian matrix \(A\) (\(A=A^*\)), we have that \((A^2)_{ii} = (AA^*)_{ii}= \|\bs A_i\|^2_2\geq ( A_{ii})^2\), where \(\bs A_i\) is the \(i\)th row of \(A\). Since the quadratic variation raised to any positive integer power is hermitian, repeated use of \((A^2)_{ii}\geq (A_{ii})^2\) gives, for~\(t \geq 0\),
\begin{equation}
\label{eq: Lemma 1.7 interim 2}
    \sum_{i=1}^d\E\F[[M_i,\overline{M}_i]_t^{2^k}\R]\leq \sum_{i=1}^d\E\F[([\bs M,\overline{\bs M}]_t^2)_{ii}^{2^{k-1}}\R]\leq ... \leq \E \F[\mathrm{tr}\F([\bs M,\overline{\bs M}]_t^{2^k}\R)\R].
\end{equation}
The lemma follows by using (\ref{eq: Lemma 1.7 interim 2}) in (\ref{eq: Lemma 1.7 interim 1}) and \eqref{eq: almost d 1}.
\end{proof}
\begin{proof}[Proof of Lemma \ref{Lemma: fourth moments}] For \(t \geq 0\), let \(\bs Y(t) = \mathrm{e}^{-At}\bs X(t)\). This is a martingale by Lemma \ref{prop: martingale bp}. This implies, for any fixed~\(r~\geq~0\), the following processes are martingales started at 0
\begin{align}
    &\bs M_{s,r}(t) = (1+r)^{-(m_1-1)/2}\mathrm{e}^{(A-\lambda_1/2)r}P_s( \bs Y(t)-\bs Y(0)), \quad t\geq 0, \nonumber \\
    &\bs M_{c,\kappa,r}(t) = (1+r)^{-(m_1+2\kappa-2)/2}\mathrm{e}^{(A-\lambda_1/2)r}N_A^{m-\kappa}P_J( \bs Y(t)-\bs Y(0)), \quad t \geq 0,\nonumber \\
    &\bs M_{\ell}(t)=P_{\ell}( \bs Y(t)-\bs Y(0)), \quad t \geq 0. \label{eq:components}
\end{align}
We see the random variables in the l.h.s.\ of Lemma \ref{Lemma: fourth moments} are \(\bs M_{\ell}(t)\), \(\bs M_{c,\kappa,t}(t)\), and \(\bs M_{s,t}(t)\) for some fixed \(t \geq 0\). For brevity, we state the following arguments for \(\bs M_{\ell}(t)\) only. Almost identical arguments hold for \(\bs M_{s,t}(t)\) and \(\bs M_{c,\kappa,t}(t)\). To recover the arguments for small or critical components, one needs to replace
\begin{align*}
    &\bs M_{\ell}\text{ with \(\bs M_{s,t}\) or \(\bs M_{c,\kappa,t}\), and}\\
    &\text{\(P_{\ell}\) with \((1+t)^{-(m_1-1)/2}\mathrm{e}^{(A-\lambda_1/2)t}P_s\) or \((1+t)^{-(m_1+2\kappa-2)/2}\mathrm{e}^{(A-\lambda_1/2)t}N_A^{m-\kappa}P_J.\)}
\end{align*}If a step for the small or critical component case is not clear from its large component counterpart we add additional justification. 

Since \(\bs M_{\ell}\) is a martingale, for any \(t \geq 0\), we apply Lemma \ref{Lemma: BDG} to get
\begin{align}
\E \F[\|P_{\ell}(\bs Y(t)-\bs Y(0))\|_2^4\R] &=  \E\F[ \|\bs M_{\ell}(t)\|_2^4\R]\leq \cst \E \F[ \mathrm{tr}\F([\bs M_{\ell},\overline{\bs M}_{\ell}]_t^{2}\R)\R]. \label{eq:largecomp}
\end{align}
(Since the moments we are interested in are for some fixed \(t \geq 0\), we get the same bound for small and critical components using the martingales \(\bs M_{t,r}\) and \(\bs M_{c,k,t}\) evaluated at time \(t\).) Applying the triangle inequality to the l.h.s.\ of this bound gives, for any \(t \geq 0\),
\begin{equation}
\label{eq:largecomp2}
 \E \F[\|P_{\ell} \bs Y(t)\|_2^4\R] \leq \cst \E \F[ \mathrm{tr}\F([\bs M_{\ell},\overline{\bs M}_{\ell}]_t^{2}\R)\R] +\cst  \| \bs X(0)\|_2^4.
\end{equation} 
(To get the final term on the r.h.s.\ for small or critical components, we use, for any \(t\geq 0\), \[\|(1+t)^{-(m_1-1)/2}\e^{(A-\lambda_1/2)t}P_s\|_2 \leq \cst \text{ or }\|(1+t)^{-(m_1+2\kappa-2)/2}\e^{(A-\lambda_1/2)t}N_A^{m-\kappa}P_J\|_2 \leq \cst,\] which follows by \eqref{eq: matrix exp bounds}.) By Equations (\ref{eq:largecomp}), and (\ref{eq:largecomp2}), the lemma is proven once we show
\begin{equation}
\label{eq: mightneed 1}
    \E \F[ \mathrm{tr}\F([\bs M_{\ell},\overline{\bs M}_{\ell}]_t^{2}\R)\R] \leq  \cst (1+ \Xi_4^{3/2})\|\bs X(0)\|^2_2(1\wedge t ).
\end{equation}For \(1 \leq i \leq d\), let \((\tau_{ik})_{k\geq 1}\) be the ordered death times of particles of type \(i\) in \(\bs X\). Since a time-independent shift does not change the quadratic variation, we have 
\begin{equation}
 [\bs  M_{\ell},\overline{\bs M}_{\ell}]_t =  [P_{\ell}\bs Y  ,\overline{P_{\ell}\bs Y}]_t = \sum_{i=1}^d \sum_{k:\tau_{ik}\leq t}  P_{\ell} \mathrm{e}^{-A\tau_{ik}}\Delta \bs X(\tau_{ik}) \Delta \bs X(\tau_{ik})'(\overline{P}_{\ell} \mathrm{e}^{-A\tau_{ik}})', \quad t\geq 0, \label{eq: lemma 1.6 interim6}
\end{equation}
where in the second equality we have used that \(\overline{\bs Y}=\bs Y\), and that the jumps of \(\bs Y\) correspond to the particle deaths in \(\bs X\). If, in Lemma \ref{lemma: Jan 9.3}, we take \(p=1\), \(M_1(t) = P_{\ell} \mathrm{e}^{-At}\), \(N(\bs x) = \bs x \bs x'\), and \(M_2(t)=(\overline{P}_{\ell} \mathrm{e}^{-At})'\), then the first term on the r.h.s.\ of \eqref{eq: improve 1} is equal to
\[\sum_{k:\tau_{ik}\leq t}  P_{\ell} \mathrm{e}^{-A\tau_{ik}}\Delta \bs X(\tau_{ik}) \Delta \bs X(\tau_{ik})'(\overline{P}_{\ell} \mathrm{e}^{-A\tau_{ik}})', \quad t\geq 0.\]
This, Lemma \ref{lemma: Jan 9.3}, and \eqref{eq: lemma 1.6 interim6} imply
\begin{align}
  Z_{\ell}(t) := [\bs M_{\ell},\overline{\bs M}_{\ell}]_t-\sum_{i=1}^d \int_{0}^t P_{\ell}\mathrm{e}^{-Av}\E[\bs \xi_i \bs \xi_i'](\overline{P}_{\ell}\mathrm{e}^{-Av})'a_iX(v)_i\mathrm{d}v, \quad t\geq 0 \label{eq:jan 9.3 mart}
\end{align}
is a matrix-valued martingale. This martingale consists of a jump processes (first term) and a continuous process of finite variation (second term). The continuous finite-variation term does not contribute to the quadratic variation, so the quadratic variations of \( Z_{\ell}\) and \([\bs M_{\ell},\overline{\bs M}_{\ell}]\) are equal a.s. Furthermore, since the summands of the quadratic variation are hermitian and correspond to the jumps of the process, the quadratic variation of \([\bs M_{\ell},\overline{\bs M}_{\ell}]\) is the r.h.s.\ of \eqref{eq: lemma 1.6 interim6} with the summands squared. These two results imply a.s.\
\begin{equation}
\label{eq:quad var trick}
[  Z_{\ell},  \overline{Z}_{\ell}]_t  =\sum_{i=1}^d \sum_{k:\tau_{ik}\leq t}  (P_{\ell} \mathrm{e}^{-A\tau_{ik}}\Delta \bs X(\tau_{ik}) \Delta \bs X(\tau_{ik})'(\overline{P}_{\ell} \mathrm{e}^{-A\tau_{ik}})')^2,  \quad t \geq 0.
\end{equation}
If, in Lemma \ref{lemma: Jan 9.3}, we take \(p=2\),  \(M_{11}(t)=M_{12}(t) = P_{\ell} \mathrm{e}^{-At}\), \(N_1(\bs x)=N_2(\bs x) = \bs x \bs x'\), and \(M_{21}(t)=M_{22}(t)=(\overline{P}_{\ell} \mathrm{e}^{-At})'\), then the first term on the r.h.s.\ of \eqref{eq: improve 1} is equal to
\begin{equation*}
    \sum_{k:\tau_{ik}\leq t}  (P_{\ell} \mathrm{e}^{-A\tau_{ik}}\Delta \bs X(\tau_{ik}) \Delta \bs X(\tau_{ik})'(\overline{P}_{\ell} \mathrm{e}^{-A\tau_{ik}})')^2, \quad t\geq 0.
\end{equation*} This, Lemma \ref{lemma: Jan 9.3}, and \eqref{eq:quad var trick} imply, for all \(t \geq 0\),
\begin{equation}
\label{eq:jan 9.3 aplic}
  \E [  Z_{\ell},  \overline{Z}_{\ell}]_t= \sum_{i=1}^d \int_{0}^t \E\F[\F(P_{\ell}\mathrm{e}^{-Av}\bs \xi_i \bs \xi_i'(\overline{P}_{\ell}\mathrm{e}^{-Av})'\R)^2\R](a_i\mathrm{e}^{Av}\bs X(0))_i\mathrm{d}v. 
\end{equation}
For ease of notation, let
\begin{align}
    & U_{\ell}(t) = \sum_{i=1}^d \int_{0}^t P_{\ell}\mathrm{e}^{-Av}\E[\bs \xi_i \bs \xi_i'](\overline{P}_{\ell}\mathrm{e}^{-Av})'a_iX(v)_i\mathrm{d}v, \quad t\geq 0,\label{eq: U part} \\
    & V_{\ell}(t)=\sum_{i=1}^d \int_{0}^t \E\F[\F(P_{\ell}\mathrm{e}^{-Av}\bs \xi_i \bs \xi_i'(\overline{P}_{\ell}\mathrm{e}^{-Av})'\R)^2\R](a_i\mathrm{e}^{Av}\bs X(0))_i\mathrm{d}v, \quad t\geq0. \label{eq: V part}
\end{align}
Since \(Z_{\ell}\) is hermitian with \(Z_{\ell}(0)=0\), equations (\ref{eq: second moment martingale}) and (\ref{eq:jan 9.3 aplic}) imply, for \(t \geq 0\),
\begin{equation}
\E[Z_{\ell}(t)^2]=\E[Z_{\ell}(t)\overline{Z}_{\ell}(t)]
= \E [  Z_{\ell},  \overline{Z}_{\ell}]_t = V_{\ell}(t). \label{eq:lemma 1.6 interim}
\end{equation}
Using (\ref{eq:jan 9.3 mart}) and \eqref{eq: U part} in the l.h.s., then rearranging, gives
\begin{align}
\nonumber
    \E\F[[\bs M_{\ell},\overline{\bs M}_{\ell}]_t^2\R]&= V_{\ell}(t)+\E[[\bs M_{\ell},\overline{\bs M}_{\ell}]_t U_{\ell}(t)]+\E[ U_{\ell}(t)[\bs M_{\ell},\overline{\bs M}_{\ell}]_t]-\E[ U_{\ell}(t)^2]\\
    &=V_{\ell}(t)+\E[ Z_{\ell}(t) U_{\ell}(t)]+\E[ U_{\ell}(t) Z_{\ell}(t)]+\E[ U_{\ell}(t)^2].\label{eq:innerproductmatrixuse}
\end{align}
Define the matrix inner product \(\langle A,B\rangle_{\mathrm{tr}} = \mathrm{tr}(AB^*)\) and associated norm \(\|A\|_{\mathrm{tr}}=\mathrm{tr}(AA^*)^{1/2}\). Since, for all \(t\geq0\), \(Z_{\ell}(t)\) and \(U_{\ell}(t)\) are hermitian, applying the trace to both sides of (\ref{eq:innerproductmatrixuse}) and using that the trace is linear gives
\begin{equation}
\label{eq:lemma1.6interim2}
   \E\F[ \mathrm{tr}([\bs M_{\ell},\overline{\bs M}_{\ell}]_t^2)\R] = \mathrm{tr}(V_{\ell}(t))+ \E[ \langle Z_{\ell}(t), U_{\ell}(t)\rangle_{\mathrm{tr}}]+\E[\langle U_{\ell}(t) ,Z_{\ell}(t)\rangle_{\mathrm{tr}}]+\E[\mathrm{tr}( U_{\ell}(t)^2)].
\end{equation}
By the Cauchy-Schwarz inequality, we have
\begin{align}
    \E[ \langle Z_{\ell}(t), U_{\ell}(t)\rangle_{\mathrm{tr}}] &\leq \E[\|U_{\ell}(t)\|_{\mathrm{tr}}^2]^{1/2}\E[\|Z_{\ell}(t)\|_{\mathrm{tr}}^2]^{1/2}\nonumber\\
    &=\E[\mathrm{tr}(U_{\ell}(t)^2)]^{1/2}\E[\mathrm{tr}(Z_{\ell}(t)^2)]^{1/2}\nonumber\\
    &=\E[\mathrm{tr}(U_{\ell}(t)^2)]^{1/2}\mathrm{tr}(V_{\ell}(t))^{1/2},
    \label{eq:lemma1.6interim3}
\end{align}
where in the first equality we use that \(U_{\ell}(t)\) and \(Z_{\ell}(t)\) are hermitian, and in the second equality we use (\ref{eq:lemma 1.6 interim}). Using (\ref{eq:lemma1.6interim3}) in (\ref{eq:lemma1.6interim2}) gives
\begin{align}
    \E\F[ \mathrm{tr}([\bs M_{\ell},\overline{\bs M}_{\ell}]_t^2)\R] &\leq \mathrm{tr}(V_{\ell}(t))+2\E[\mathrm{tr}(U_{\ell}(t)^2)]^{1/2}\mathrm{tr}(V_{\ell})^{1/2}+\E[\mathrm{tr}(U_{\ell}(t)^2)]\\
     &\leq  3\mathrm{tr}(V_{\ell}(t))+3\E[\mathrm{tr}(U_{\ell}(t)^2)].\label{eq:Lemma 1.6 int3}
\end{align}
By using this in \eqref{eq: mightneed 1}, we see the lemma is proved once we show, for every \(t\geq 0\),
\begin{align}
    &\mathrm{tr}(V_{\ell}(t)) \leq  \cst   (1+ \Xi_4^{3/2}) \|\bs X(0)\|_2^2(1\wedge t),\nonumber\\
    &\E[\mathrm{tr}(U_{\ell}(t)^2)] \leq \cst   (1+ \Xi_4^{3/2}) \|\bs X(0)\|_2^2(1\wedge t) . \label{eq: may need 2}
\end{align}
From here, we treat the small, critical, and large component cases separately. 
\\~\\
\textbf{Large Components:} Set \(\hat{\lambda}\) to be the real part of the eigenvalue(s) in \(\Lambda_{\ell}\) with smallest real part. Set \(\hat{\kappa}\) to be the size of the largest Jordan block across all Jordan blocks with eigenvalue in \(\Lambda_{\ell}\). By (\ref{eq: matrix exp bounds}), and (\ref{eq: matrix exp bounds all}), for \(t \geq 0\),
\begin{align}
    \mathrm{tr}(V_{\ell}(t)) &=\sum_{i=1}^d \int_{0}^t \mathrm{tr}\E\F[\F(P_{\ell}\mathrm{e}^{-Av}\bs \xi_i \bs \xi_i'(\overline{P}_{\ell}\mathrm{e}^{-Av})'\R)^2\R](a_i\mathrm{e}^{Av}\bs X(0))_i\mathrm{d}v \nonumber \\
    &\leq \cst \Xi_4\|\bs X(0)\|_2\int_{0}^t(1+v)^{4\hat{\kappa}+m_1-5}\mathrm{e}^{-4\hat{\lambda} v+\lambda_1v}\mathrm{d}v. \label{eq:lemma1.6interim4}
\end{align}
Since \(\hat\lambda\) is the real part of some large eigenvalue, we have that \(2\hat \lambda > \lambda_1\), so the integrand on the r.h.s.\ of (\ref{eq:lemma1.6interim4}) is bounded and decays exponentially to 0 as \(v \rightarrow \infty\). Therefore, for any \(t \geq 0\),
\begin{equation}
\label{eq:Lemma 1.6 int2}
    \mathrm{tr}(V_{\ell}(t)) \leq \cst  \Xi_4 \|\bs X(0)\|_2(1\wedge t) \leq  \cst   (1+ \Xi_4^{3/2}) \|\bs X(0)\|_2^2(1\wedge t).
\end{equation}
Again by (\ref{eq: matrix exp bounds}), for any \(t \geq 0\) and \(\varepsilon>0\),
\begin{align}
   \E[\mathrm{tr}(U_{\ell}(t)^2)] &= \E\F[\mathrm{tr}\F(\F(  \sum_{i=1}^d \int_{0}^t P_{\ell}\mathrm{e}^{-Av}\E[\bs \xi_i \bs \xi_i'](\overline{P}_{\ell}\mathrm{e}^{-Av})'a_iX(v)_i\mathrm{d}v\R)^2\R)\R]\nonumber \\
   & \leq \cst \Xi_4 \sum_{i=1}^d\E\F[\F(\int_{0}^t (1+v)^{2(\hat\kappa-1)}\mathrm{e}^{-2\hat\lambda v}a_iX(v)_i\mathrm{d}v\R)^2\R]\nonumber\\
    &\leq \cst \Xi_4  \sum_{i=1}^d\int_{0}^t\mathrm{e}^{-4\varepsilon v}\mathrm{d}v\E\F[ \int_{0}^t (1+v)^{4(\hat\kappa-1)}\mathrm{e}^{-4(\hat\lambda-\varepsilon) v}a_i^2X(v)_i^2\mathrm{d}v\R]\nonumber\\
 &\leq  \cst \Xi_4  \sum_{i=1}^d\int_{0}^t\mathrm{e}^{-4\varepsilon v}\mathrm{d}v\int_{0}^t (1+v)^{4(\hat\kappa-1)}\mathrm{e}^{-4(\hat\lambda-\varepsilon) v}\E[X(v)_i^2]\mathrm{d}v, \label{eq:Lemma 1.6 int1}
\end{align}
where in the second line we have used by Jensen's inequality (for \(1 \leq i \leq d\), \(\|\mathbb{E}[\bs \xi_i \bs \xi_i']\|_2^2 \leq \Xi_4\)), and in the penultimate line we have used the Cauchy-Schwarz inequality for \(L^2\) bounded functions. To bound the expectation in \eqref{eq:Lemma 1.6 int1}, we take \(M_1=I\) in (\ref{eq: second moment 2}) and apply \eqref{eq: matrix exp bounds all}. This gives, for \(t \geq 0\),
\begin{align}
    \mathbb{E}[\|\bs X(t)\|_2^2] &= \sum_{i=1}^d \int_{0}^t \E[\|\mathrm{e}^{A(t-v)}\bs \xi_i\|_2^2] (a_i\mathrm{e}^{Av}\bs X(0))_i\mathrm{d}v+\|\mathrm{e}^{At}\bs X(0) \|_2^2 \nonumber \\
    &\leq \cst \sum_{i=1}^d \E[\| \bs \xi_i\|_2^2] \|\bs X(0)\|_2\e^{2\lambda_1t}(1+t)^{2(m_1-1)}\int_{0}^t(1+v)^{4(m_1-1)}\e^{-\lambda_1v} \mathrm{d}v+\cst(1+t)^{2(m_1-1)}\mathrm{e}^{2\lambda_1t}\|\bs X(0)\|_2^2
    \nonumber \\&\leq \cst  (1+ \Xi_4^{1/2}) (1+t)^{2(m_1-1)}\mathrm{e}^{2\lambda_1t}\|\bs X(0)\|_2^2, \label{eq: final d 2}
\end{align}
where in the second line we have used \eqref{eq: matrix exp bounds} to bound the matrix exponential of \(A\), and in the final line we have used Jensen's inequality to get \(\sum_{i=1}^d \E[\| \bs \xi_i\|_2^2] \leq \Xi_4^{1/2}\). Using this bound in \eqref{eq:Lemma 1.6 int1} gives, for~\(t\geq 0\),
\begin{equation}
  \E[\mathrm{tr}(U_{\ell}(t)^2)]   \leq \cst  (1+ \Xi_4^{3/2}) \|\bs X(0)\|_2^2\int_{0}^t\mathrm{e}^{-4\varepsilon v}\mathrm{d}v\int_{0}^t (1+v)^{4\hat\kappa+2m_1-6}\mathrm{e}^{(4\varepsilon-4\hat\lambda+2\lambda_1) v}\mathrm{d}v. \label{eq:1.6 cauchy schwarz bound}
\end{equation} 
 Since \(2\hat\lambda > \lambda_1\), we can take \(\varepsilon>0\) small enough such that \(2(\hat\lambda-\varepsilon) > \lambda_1\). With this choice, the second integrand on the r.h.s.\ of (\ref{eq:1.6 cauchy schwarz bound}) is bounded and decays exponentially to 0 as \(v\rightarrow \infty\). The first integral on the r.h.s.\ of (\ref{eq:1.6 cauchy schwarz bound}) is bounded by~\((4\varepsilon)^{-1}\) for all \(t\). Hence, the r.h.s.\ of \eqref{eq:1.6 cauchy schwarz bound} is \(\mathcal{O}(1)\) as \(t \rightarrow \infty\) and \(\mathcal{O}(t)\) as \(t \rightarrow 0\). Therefore, for \(t \geq 0\),
\begin{equation*}
    \E[\mathrm{tr}(U_{\ell}(t)^2)]\leq \cst (1+ \Xi_4^{3/2})   \|\bs X(0)\|_2^2(1\wedge t).
\end{equation*}
This and \eqref{eq:Lemma 1.6 int2} is \eqref{eq: may need 2} which gives the lemma for large components as claimed.
\\~\\
\textbf{Small and Critical Components:} We need to show \eqref{eq: may need 2} for
\begin{align*}
    &\text{\(\mathrm{tr}(V_{s,t}(t))\) and \(\E[\mathrm{tr}(U_{s,t}(t)^2)]\) (small components)} \quad \text{\(\mathrm{tr}(V_{c,\kappa,t}(t))\) and \(\E[\mathrm{tr}(U_{c,\kappa,t}(t)^2)]\) (critical components)}.
\end{align*}
Recall that these are defined by \eqref{eq: U part} and \eqref{eq: V part} with \( P_{\ell}\) replaced by
\begin{equation*}
    \text{\((1+t)^{-(m_1-1)/2}\mathrm{e}^{(A-\lambda_1/2)t}P_s\) (small components)\(\quad (1+t)^{-(m_1+2\kappa-1)/2}\mathrm{e}^{(A-\lambda_1/2)t}N_A^{m-\kappa}P_J\) (critical components).}
\end{equation*} Let \(\hat\lambda\) be the real part of the eigenvalue(s) with largest real part in \(\Lambda_s\). Let \(\hat\kappa\) be the size of the largest Jordan block across all Jordan blocks with eigenvalue in \(\Lambda_s\). Then, by \eqref{eq: matrix exp bounds}, for \(t \geq 0\),
\begin{align}
    \mathrm{tr}(V_{s,t}(t)) &=(1+t)^{-2(m_1-1)}\mathrm{e}^{-2\lambda_1 t}\sum_{i=1}^d \int_{0}^t \mathrm{tr}\E\F[\F(P_{s}\mathrm{e}^{A(t-v)}\bs \xi_i \bs \xi_i'(\overline{P}_{s}\mathrm{e}^{A(t-v)})'\R)^2\R](a_i\mathrm{e}^{Av}\bs X(0))_i\mathrm{d}v \nonumber \\
    &\leq \cst \Xi_4 \|\bs X(0)\|_2(1+t)^{-2(m_1-1)}\mathrm{e}^{-2\lambda_1 t}\int_{0}^t(1+(t-v))^{4(\hat\kappa-1)}(1+v)^{m_1-1}\mathrm{e}^{4\hat\lambda (t-v)+\lambda_1 v}\mathrm{d}v\nonumber\\
    &\leq\cst \Xi_4  \|\bs X(0)\|_2\int_{0}^t(1+(t-v))^{4(\hat\kappa-1)}\mathrm{e}^{4\hat\lambda (t-v)-2\lambda_1 (t-v)}\mathrm{d}v \nonumber\\
    &=\cst \Xi_4  \|\bs X(0)\|_2\int_{0}^t(1+u)^{4(\hat\kappa-1)}\mathrm{e}^{4\hat\lambda u-2\lambda_1 u}\mathrm{d}u\label{eq:smallcompbound1},
\end{align}
where in the final line we have used the change of variables \(t-v=u\). Since \(\hat\lambda\) is the real part of some eigenvalue in \(\Lambda_s\), we have that \(2\hat\lambda <\lambda_1\), therefore the integrand in (\ref{eq:smallcompbound1}) decays exponentially to 0 as \(u \rightarrow \infty\). Thus, the r.h.s.\ of (\ref{eq:smallcompbound1}) is \(\mathcal{O}(1)\) as \( t \rightarrow \infty\) and \(\mathcal{O}(t)\) as \(t \rightarrow 0\). This implies, for all \(t\geq0\),
\begin{equation}
  \mathrm{tr}(V_{s,t}(t)) \leq \cst \Xi_4  \|\bs X(0)\|_2(1\wedge t). \label{eq: V bound}  
\end{equation}
Also by \((\ref{eq: matrix exp bounds})\), for \(t\geq  0\),
\begin{align}
  \mathrm{tr}(V_{c,\kappa,t}(t)) &=(1+t)^{-2(m_1+2\kappa-2)}\mathrm{e}^{-2\lambda_1 t}\sum_{i=1}^d \int_{0}^t \mathrm{tr}\E\F[\F(N_A^{m-\kappa}P_J\mathrm{e}^{A(t-v)}\bs \xi_i \bs \xi_i'(N_A^{m-\kappa}\overline{P}_{J}\mathrm{e}^{A(t-v)})'\R)^2\R](a_i\mathrm{e}^{Av}\bs X(0))_i\mathrm{d}v \nonumber \\
  &\leq \cst \Xi_4  \|\bs X(0)\|_2 (1+t)^{-2(m_1+2\kappa-2)}\mathrm{e}^{-2\lambda_1 t}\int_{0}^t(1+(t-v))^{4(\kappa-1)}(1+v)^{m_1-1}\mathrm{e}^{2\lambda_1 (t-v)+\lambda_1 v}\mathrm{d}v,\nonumber \\
    &\leq \cst \Xi_4  \|\bs X(0)\|_2 (1+t)^{-4(\kappa-1/2)}\int_{0}^t(1+u)^{4(\kappa-1)}\mathrm{d}u.  \label{eq:critcompbound1}
\end{align}
The bound on the r.h.s.\ tends to 0 as \(t \rightarrow \infty\), and is \(\mathcal{O}(t)\) as \(t\rightarrow 0\). Thus, for all \(t\geq 0\),
\begin{align}
    &\mathrm{tr}(V_{c,\kappa,t}(t)) \leq \cst \Xi_4  \|\bs X(0)\|_2(1 \wedge t).
    \label{eq: may need 4}
\end{align}
Next, by (\ref{eq: matrix exp bounds}), for \(t \geq 0\),
\begin{align}
   \E[\mathrm{tr}(U_{s,t}(t)^2)] &= (1+t)^{-2(m_1-1)}\mathrm{e}^{-2\lambda_1 t}\E\F[\mathrm{tr}\F(\F(  \sum_{i=1}^d \int_{0}^t P_{s}\mathrm{e}^{A(t-v)}\E[\bs \xi_i \bs \xi_i'](\overline{P}_{s}\mathrm{e}^{A(t-v)})'a_iX(v)_i\mathrm{d}v\R)^2\R)\R]\nonumber\\
   &\leq \cst \Xi_4  (1+t)^{-2(m_1-1)}\mathrm{e}^{-2\lambda_1 t}\sum_{i=1}^d \E\F[\F(\int_{0}^t (1+(t-v))^{2(\hat\kappa-1)}\mathrm{e}^{2\hat\lambda(t- v)}X(v)_i\mathrm{d}v\R)^2\R],\label{eq: bound small}
   \end{align}
   and
   \begin{align}
\E[\mathrm{tr}(U_{c,\kappa,t}(t)^2)] &= (1+t)^{-2(m_1+2\kappa-2)}\mathrm{e}^{-2\lambda_1 t}\E\F[\mathrm{tr}\F(\F(  \sum_{i=1}^d \int_{0}^t N_A^{m-\kappa}P_J\mathrm{e}^{A(t-v)}\E[\bs \xi_i \bs \xi_i'](N_A^{m-\kappa}\overline{P}_{J}\mathrm{e}^{A(t-v)})'X(v)_i\mathrm{d}v\R)^2\R)\R]\nonumber\\
   &\leq \cst \Xi_4  (1+t)^{-2(m_1+2\kappa-2)}\mathrm{e}^{-2\lambda_1 t}\sum_{i=1}^d\E\F[\F( \int_{0}^t (1+(t-v))^{2(\kappa-1)}\mathrm{e}^{\lambda_1(t- v)}X(v)_i\mathrm{d}v\R)^2\R].\label{eq: bound crit}
\end{align}
Similarly to the large component case, we apply the Cauchy-Schwarz inequality to the expectations in the r.h.s.\ of \eqref{eq: bound small} and \eqref{eq: bound crit}, then use \eqref{eq: final d 2} in the resulting bounds. This gives, for any \(\varepsilon >0\), and \(t \geq 0\),
\begin{align}
 & (1+t)^{-2(m_1-1)}\mathrm{e}^{-2\lambda_1 t}\sum_{i=1}^d\E\F[\F(\int_{0}^t (1+(t-v))^{2(\hat\kappa-1)}\mathrm{e}^{2\hat\lambda v}X(t-v)_i\mathrm{d}v\R)^2\R]\nonumber\\
 &\quad\quad\quad\quad\quad\quad\quad\quad\quad\quad\quad\quad\quad\quad\leq \cst  (1+ \Xi_4^{1/2}) \|\bs X(0)\|_2^2 \int_{0}^{t}\mathrm{e}^{-2\varepsilon v}\mathrm{d}v\int_{0}^t (1+v)^{4(\hat\kappa-1)}\mathrm{e}^{4\hat\lambda v-2(\lambda_1-\varepsilon)v}\mathrm{d}v,  \label{eq:smallstuff}
 \end{align}
 and
 \begin{align}
 &(1+t)^{-2(m_1+2\kappa-2)}\mathrm{e}^{-2\lambda_1 t}\sum_{i=1}^d\E\F[\F( \int_{0}^t (1+(t-v))^{2(\kappa-1)}\mathrm{e}^{\lambda_1(t- v)}X(v)_i\mathrm{d}v\R)^2\R]\nonumber\\
 &\quad\quad\quad\quad\quad\quad\quad\quad\quad\quad\quad\quad\quad\quad\quad\quad\quad\quad\quad\leq \cst  (1+ \Xi_4^{1/2}) \|\bs X(0)\|_2^2(1+t)^{-4(\kappa-1/2)}\int_{0}^t\mathrm{d}v\int_{0}^t(1+v)^{4(\kappa-1)}\mathrm{d}v\nonumber\\
 &\quad\quad\quad\quad\quad\quad\quad\quad\quad\quad\quad\quad\quad\quad\quad\quad\quad\quad\quad \leq \cst (1+ \Xi_4^{1/2}) \|\bs X(0)\|_2^2 (1+t)^{-1}t.\label{eq:criticalstuff}
\end{align}
Since \(2\hat\lambda < \lambda_1\), we can choose \(\varepsilon\) small enough such that \(2\hat \lambda < \lambda_1-\varepsilon\). With this choice, the first integral on the r.h.s.\ of (\ref{eq:smallstuff}) is bounded by \((2\varepsilon)^{-1}\), and the second integrand decays exponentially to 0 as \(v \rightarrow \infty\). Thus, \eqref{eq:smallstuff} is \(\mathcal{O}(1)\) as \( t \rightarrow \infty\) and \(\mathcal{O}(t)\) as \(t \rightarrow 0\). We see this also holds for (\ref{eq:criticalstuff}). Therefore, if we use \eqref{eq:smallstuff} in \eqref{eq: bound small} and \eqref{eq:criticalstuff} in \eqref{eq: bound crit}, we get, for all \(t\geq 0\),
\begin{align*}
  &\E[\mathrm{tr}(U_{s,t}(t)^2)] \leq  \cst (1+\Xi_4^{3/2}) \|\bs X(0)\|_2^2(1 \wedge t).\\
  &\E[\mathrm{tr}(U_{c,\kappa,t}(t)^2)] \leq  \cst (1+\Xi_4^{3/2}) \|\bs X(0)\|_2^2(1\wedge t).
\end{align*}
This, \eqref{eq: may need 4}, and \eqref{eq: V bound} is \eqref{eq: may need 2} for small and critical components, which implies the lemma in these cases as claimed.
\end{proof}
\subsection{Proof of Theorem \ref{Theorem: Main continuous time} CT}
\label{Appendix: proof of CT}
Let
\begin{equation*}
    \bs Z_n(t) = N^{-1/2}(\e^{-A t}\bs X_n(t)- N\bs \mu), \quad t\geq 0.
\end{equation*}
 This is a martingale for each fixed \(n\) by Lemma \ref{prop: martingale bp}. We are going to apply Lemma \ref{proposition: Janson 9.1} to this sequence of martingales with limit \((\e^{-At}\bs W_2(t))_{t\geq 0}\). Starting with Condition 2, we must show, for each \(t\geq 0\),
 \begin{equation}
 [\bs Z_n,\bs Z_n] \coninprob \sum_{i=1}^d  \int_{0}^{t_1}\mathrm{e}^{-Av}\E[\bs \xi_i \bs \xi_i']\mathrm{e}^{-A'v}a_i(\mathrm{e}^{Av}\bs \mu)_i\mathrm{d}v. \label{eq: Zn quad var 4}
 \end{equation}
as \(n\rightarrow \infty\). By the same argument used to show \eqref{eq: Zn quad var 3},
 \begin{equation*}
    Y_n(t) =  [\bs Z_n,\bs Z_n]_t - N^{-1}\sum_{i=1}^d\int^{ t}_0 \e^{-Av}\E[\bs \xi_i\bs \xi_i']\e^{-A'v}a_i X_n(v)_i \mathrm{d}v, \quad t\geq 0
 \end{equation*}
 is a martingale. We prove \eqref{eq: Zn quad var 4} by showing, for each \(t\geq 0\),
\begin{align}
   &Y_n(t) \coninprob 0, \label{eq: Yn V2}\\
   &U_n(t):=N^{-1}\sum_{i=1}^d\int^{ t}_0 \e^{-Av}\E[\bs \xi_i\bs \xi_i']\e^{-A'v}a_i( X_n(v)_i-N(\e^{A v}\bs \mu)_i) \mathrm{d}v \coninprob 0. \label{eq: Un V2}
\end{align}
For the remainder of the proof, fix \(t\geq 0\). By the same argument used to show \eqref{eq: Yn component 1}, we see \begin{equation*}
    \E[Y_n(t)Y_n(t)'] = N^{-2}\sum_{i=1}^d\int^{\varepsilon t}_0 (\e^{-Av}\E[\bs \xi_i\bs \xi_i']\e^{-A'v})^2a_i N(\e^{Av}\bs \mu)_i \mathrm{d}v.
\end{equation*}
By (A1), the integrand can be bounded uniformly by \(N\cst\) for \(v \in [0,t]\) (take the bound to be where the integrand reaches its maximum in this interval). Thus, the second moment of \(\|Y_n(t)\|_2\) tends to 0 as \(n\rightarrow \infty\). Therefore \eqref{eq: Yn V2} holds. Equation~\eqref{eq: Un V2} is shown from the ST case in a similar matter, we omit the details. Lastly, Condition 3 follows from the ST case with \(\varepsilon_{\mathrm{max}} =1\). Thus,
\begin{equation*}
    \bs Z_n(t) \conindis  \e^{-At}\bs W_2(t) \text{ in } D[0,\infty)
\end{equation*}
as \(n\rightarrow \infty\). The theorem follows by multiplying on the left by \(\e^{At}\).
\subsection{Proof of Theorem \ref{theorem: cont time 2}}
\label{Appendix: proof of covar}
We first show pairwise independence between families, then show the covariance functions within dependent families. Each case is treated separately.
\\~\\
\textbf{Independence of \(\mathcal{W}_s\) and \(\mathcal{V}_{\ell}\):}
Let \(\bs V_{J} \in \mathcal{V}_{\ell}\), \(T\geq 1/2\), and
\begin{align*}
    &\bs Z_{ns}(t):=  N^{-1/2}\omega^{-(m_1-1)/2}\mathrm{e}^{-\lambda_1(\omega+t)/2}P_s( \bs X_n(\omega+t)-N\mathrm{e}^{A(\omega+t)}\bs \mu), \quad -T \leq t  \leq T ,\\
    & \bs Z_{n\ell}(t):=N^{-1/2} P_{J}(\e^{-A\omega t}\bs X_n(\omega t)-N\bs \mu), \quad 0 <  t < T.
\end{align*}
Asymptotic independence of these two processes follows once we show, for each \(T\geq 1/2\), joint convergence of 
\begin{equation}
    (\bs Z_{ns},\bs Z_{n \ell}) \conindis (\bs W_s, \bs V_J) \text{ in } D[-T,T]\times D(0,T] \label{eq: Convergence 6}
\end{equation}
as \(n\rightarrow \infty\), where \(\bs W_s\) is restricted to \(-T\leq t \leq T\), and \(\bs W_s\) and \(\bs V_J\) are independent. Tightness of the joint sequence immediately follows by convergence of each component. It is left to show independence of the finite dimensional distributions. Since \(\bs V_J\) is a constant process, it suffices to show, for each sequence of times \(-T\leq t_1 \leq ... \leq t_k \leq T\),
\begin{equation*}
     (\bs Z_{ns}(t_1),...,\bs Z_{ns}(t_k),\bs Z_{n \ell}(1/2)) \conindis (\bs W_s(t_1),...,\bs W_s(t_k), \bs V_J)
\end{equation*}
as \(n\rightarrow \infty\). The only part of this result not given by the proof of Theorem \ref{Theorem: Main continuous time} is the independence of \((\bs W_s(t_i))_{i=1}^k\) and \(\bs V_J\). Let \(K>0\), for each \(1\leq i \leq k\), we have
\begin{equation}
    \bs Z_{ns}(t_i) = \e^{(A-\lambda_1/2)K}\bs Z_{ns}(t_i-K)+(\bs Z_{ns}(t_i)-\e^{(A-\lambda_1/2)K}\bs Z_{ns}(t_i-K)). \label{eq: Equality 2}
\end{equation}
By \eqref{eq: t1.1 second conv}, we have that the second component on the r.h.s.\ is asymptotically independent of \(\sigma(\{\bs X_n(s): s\leq \omega +t_i-K\})\). In particular, this is asymptotically independent of \(\bs Z_{n\ell}(1/2)\), since this random sequence only depends on \(\bs X_n(\omega/2)\), and~\(\omega/2<\omega +t_i-K\) for all \(n\) sufficiently large. For the first component on the r.h.s.\ of \eqref{eq: Equality 2}, since \(\bs Z_{ns}\) contains only small projections, we have, by \eqref{eq: matrix exp bounds} and Theorem \ref{Theorem: Main continuous time} LT\textsubscript{s}, for some \(\delta>0\),
\begin{equation*}
\|\e^{(A-\lambda_1/2)K}\bs Z_{ns}(t_i-K)\|_2 \leq \e^{-\delta K}\|\bs Z_{ns}(t_i-K)\|_2 \conindis \e^{-\delta K}\|\bs W_{s}(t_i-K)\|_2
\end{equation*}
as \(n\rightarrow \infty\). Since \(\bs W_s(t)\) is \(\mathcal{O}_p(1)\) for all \(t\), this implies we can send \(K\) to infinity at any rate w.r.t.\ \(n\) to get 
\begin{equation*}
    \e^{(A-\lambda_1/2)K}\bs Z_{ns}(t_i-K) \coninprob 0
\end{equation*}
as \(n\rightarrow \infty\). This, along with the fact that the second component on the r.h.s.\ of \eqref{eq: Equality 2} is asymptotically independent of~\(\bs Z_{n\ell}(1/2)\) for any fixed \(K>0\), imply \(\bs W_s(t_i)\) and \(\bs V_J\) are independent. Therefore \eqref{eq: Convergence 6} holds.
\\~\\
\textbf{Independence of \(\mathcal{W}_s\) and \(\mathcal{W}_{c,|\rho|}\), \(\rho \in \Lambda_c\):}
Let \(\bs W_{J,\kappa} \in \mathcal{W}_{c,|\rho|}\). Assume w.l.o.g.\ that \(J\) has eigenvalue \(\rho\), if not, the argument is the same with \(\rho\) replaced by \(\overline{\rho}\). By Theorem \ref{Theorem: Main continuous time} LT\textsubscript{c}, we have
\begin{equation}
  \bs Z_{nc1}(t) = N^{-1/2}\omega^{-(m_1+2\kappa -2)/2}\e^{-\rho (\omega+t)}N_A^{m-\kappa}P_{J}(\bs X_n(\omega +t)-N\e^{A(\omega+ t)}\bs \mu)\conindis \bs W_{J,\kappa}(1) \text{ in }  D(-\infty,\infty)  \label{eq: Convergence 7} 
\end{equation}
as \(n \rightarrow \infty\). Since the limit is a constant process, the proof of independence of \(\bs W_{s}\) and \(\bs W_{J,\kappa}(1)\) follows an identical structure to the previous case. The one subtle difference is, unlike the previous case where we took the value of \(\bs Z_{n\ell}\) at \(1/2\), here we take \(\bs Z_{nc1}\) at \(-T-K\). This is to guarantee the independence property needed for the second component on the r.h.s.\ of \eqref{eq: Equality 2}. Indeed, \(\bs Z_{nc}(-T-K)\) only depends on \(\omega -T-K\) this being less than \(\omega +t_i-K\) for all \(1\leq i\leq k\). Then, since one can let \(K\) grow with \(n\) at an arbitrarily slow rate, we can still guarantee
\begin{equation*}
    \bs Z_{nc1}(-T-K)\conindis \bs W_{J,\kappa}(1)
\end{equation*}
as \(n,K\rightarrow \infty\). Now, we extend this result to independence of \(\bs W_s\) and \(\bs W_{J,\kappa}\). Let
\begin{equation*}
    \bs Z_{nc2}(t)= N^{-1/2}\omega^{-(m_1+2\kappa -2)/2}\e^{-\rho \omega t }N_A^{m-\kappa}P_{J}(\bs X_n(\omega t)-N\e^{A\omega t}\bs \mu).
\end{equation*}
The case of \(\bs W_{J,\kappa}(t)\) for \(0\leq t <1\) follows an identical structure to our previous proof. The finite dimensional distributions of \((\bs Z_{nc2}(t))_{t\in[0,1)}\) depend on \(\bs X_n\) at times smaller than \(\omega -T-K\) for all \(n\) large enough, hence the independence property needed for the second component on the r.h.s.\ of \eqref{eq: Equality 2} is guaranteed to hold. Now, let \(t>1\), and let \((\varepsilon_n)_{n\geq 1}\) be a positive sequence such that \(\varepsilon_n \rightarrow 0\) as \(n\rightarrow \infty\). We have that
\begin{equation*}
 \bs Z_{nc2}(t) = \e^{(A-\rho)\omega(t-1-\varepsilon_n)}\bs Z_{nc}(1+\varepsilon_n) + (\bs Z_{nc}(t)-\e^{(A-\rho)\omega(t-1-\varepsilon_n)}\bs Z_{nc}(1+\varepsilon_n))    
\end{equation*}
The second component on the r.h.s.\ can be shown in an identical way to \eqref{eq: t1.1 second conv} to be independent of \[\sigma(\{ \bs X_n(s): s\leq\omega(1+\varepsilon_n)\}).\] (For brevity, we do not give this proof; its structure is identical to the proof of Theorem \ref{Theorem: Main continuous time} LT\textsubscript{s}: \eqref{eq: small component convergence} - Condition 1 of Theorem \ref{Theorem:Bill con} (\(k=2\)).) Therefore, if we let \(\varepsilon_n\) tend to 0 slowly enough such that \(\varepsilon_n \omega \rightarrow \infty\) as \(n\rightarrow \infty\), we have asymptotic independence of the second component on the r.h.s.\ and \(\bs Z_{ns}\). For the first component, Theorem \ref{Theorem: Main continuous time} LT\textsubscript{c}, we have
\begin{align*}
    &\e^{(A-\rho)\omega(t-1-\varepsilon_n)}\bs Z_{nc}(1+\varepsilon_n)\\
    &=N^{-1/2}\sum_{i=0}^{\kappa-1}\frac{(t-1-\varepsilon_n)^i}{i!} \omega^{-(\kappa-i -1/2)}\e^{-\rho \omega (1+\varepsilon_n)}N_A^{m-\kappa+i}P_{J,\kappa} (\bs X_n(\omega (1+\varepsilon_n))-N\e^{A\omega (1+\varepsilon_n)}\bs\mu) \\
    &\conindis \sum_{i=0}^{\kappa-1}\frac{(t-1)^i}{i!}\bs W_{J,\kappa-i}(1)
\end{align*}
as \(n\rightarrow \infty\). The limit is the sum of components independent of \(\bs W_s\) and hence the limit is independent of \(\bs W_s\). Thus, the finite dimensional distributions of \((\bs Z_{nc2}(t))_{t\in (1,\infty]}\) converge to random variables independent of the process \(\bs W_s\). Therefore, we have shown the finite dimensional distributions of \(\bs W_s\) and \(\bs W_{J,k}\) are independent which implies they are independent as processes.
\\~\\
\textbf{Independence of \(\mathcal{V}_{\ell}\) and \(\mathcal{W}_{c,|\rho|} \), \(\rho \in \Lambda_c\):}
Let \(\delta>0\), and \((\varepsilon_n)_{n\geq 1}\) be such that \(\varepsilon_n \rightarrow 0\) as \(n \rightarrow \infty\). Assume that \(0<\varepsilon_n \leq \delta \) for all \(n\). Let \(\bs V_{J_1} \in \mathcal{V}_{\ell}\), \(\bs W_{J_2,\kappa} \in \mathcal{W}_{c,|\rho|}\), and
\begin{align*}
& \bs Z_{nc}(t)= N^{-1/2}\omega^{-(\kappa -1/2)}\e^{-\rho \omega t }N_A^{m-\kappa}P_{J_2}(\bs X_n(\omega t)-N\e^{A\omega t}\bs \mu),  \quad t \geq 0,\\
    &\bs Z_{n\ell}(t):=N^{-1/2} P_{J_1}(\e^{-A\omega t}\bs X_n(\omega t)-N\bs \mu),  \quad t \geq 0.
\end{align*}
Further assume that \(\varepsilon_n \rightarrow 0\) slow enough with respect to \(n\), such that \(\bs Z_{n\ell}(\varepsilon_n) \conindis \bs V_{J_1}\) as \(n \rightarrow \infty\). Then, asymptotic independence of \(\bs Z_{nc}\) and \(\bs Z_{n\ell}\) is equivalent to showing
\begin{equation*}
    (\bs Z_{n\ell}(\varepsilon_n),\bs Z_{nc})\conindis (\bs V_{J_1},\bs W_{J_2,\kappa})
\end{equation*}
as \(n\rightarrow \infty\) with \(\bs V_{J_1},\bs W_{J_2,\kappa}\) independent. The only part of this result not given by the proof of Theorem \ref{Theorem: Main continuous time} is independence of the converging finite dimensional distributions. If \(t=0\), then trivially for all \(n\), \(\bs Z_{nc}(0) = \bs W_{J_2,\kappa}(0)=0\) a.s. For the case of \(t>0\), we have
\begin{equation*}
    \bs Z_{nc}(t) = \e^{(A-\rho)\omega(t-\varepsilon_n)}\bs Z_{nc}(\varepsilon_n) + (\bs Z_{nc}(t)-\e^{(A-\rho)\omega(t-\varepsilon_n)}\bs Z_{nc}(\varepsilon_n)) 
\end{equation*}
The second term on the r.h.s.\ is asymptotically independent of \(\sigma (\{\bs X_n(s):0\leq s\leq\varepsilon_n\})\) (see the previous proof). Therefore, it is asymptotically independent of \(\bs Z_{n\ell}(\varepsilon_n)\), and so its distributional limit is independent of \(\bs V_{J_1}\). For the first term, by Theorem \ref{Theorem: Main continuous time} LT\textsubscript{c}, we have
\begin{align*}
\e^{(A-\rho)\omega(\delta-\varepsilon_n)}\bs Z_{nc}(\varepsilon_n)=N^{-1/2}\sum_{i=0}^{\kappa-1}\frac{(\delta-\varepsilon_n)^i}{i!} \omega^{-(\kappa-i -1/2)}\e^{-\rho \omega \varepsilon_n}N_A^{m-\kappa+i}P_{J_2,\kappa} (\bs X_n(\omega \varepsilon_n)-N\e^{A\omega \varepsilon_n}\bs\mu) \coninprob 0 
\end{align*}
as \(n \rightarrow \infty\), where in the convergence, we have used, for \(0 \leq i \leq \kappa-1\), \(\bs W_{J,\kappa-i}(\varepsilon_n)\coninprob 0\) as \(n\rightarrow\infty\). Therefore, \(\bs W_{J_2,\kappa}\) and~\(\bs V_{J_1}\) are independent.
\\~\\
\textbf{Independence of \(\mathcal{W}_{c,|\rho_1|}\) and \(\mathcal{W}_{c,|\rho_2|}\) for \(|\rho_1| \neq |\rho_2|\):} This follows by Lemma \ref{Theorem: large time} and Theorem \ref{theorem: cramer wold}. Let \(\bs W_{J_1,\kappa_1} \in \mathcal{W}_{|\rho_1|,\kappa_1}\), \(\bs W_{J_2,\kappa_2} \in \mathcal{W}_{|\rho_2|,\kappa_2}\), and 
\begin{align*}
    &\bs Z_{nc_1}(t):= N^{-1/2}\omega^{-(m_1+2\kappa -2)/2}\e^{-\rho_1 \omega t }N_A^{m_1-\kappa_1}P_{J_1}(\bs X_n(\omega t)-N\e^{A\omega t}\bs \mu), \quad t\geq 0,\\
    &\bs Z_{nc_2}(t):= N^{-1/2}\omega^{-(m_1+2\kappa -2)/2}\e^{-\rho_2 \omega t }N_A^{m_2-\kappa_2}P_{J_2}(\bs X_n(\omega t)-N\e^{A\omega t}\bs \mu), \quad t\geq 0.
\end{align*}
Since \(\rho_1,\rho_2\) are in different conjugate pairs, by Lemma \ref{Theorem: large time}, we have for any \(d_1,d_2 \in \mathbb{R}\),
\begin{equation*}
    d_1\bs Z_{nc_1}(t)+d_2\bs Z_{nc_2}(t) \conindis \bs V_{d_1,d_2}(t)\text{ in }D[0,\infty),
\end{equation*}
where, for \(0\leq t_1 \leq t_2\),
\begin{align*}
  &  \Cov(\bs V_{d_1,d_2}(t_2),\bs V_{d_1,d_2}(t_1)) =  \Cov(d_1\bs W_{J_1,\kappa_1}(t_2),d_1\bs W_{J_1,\kappa_1}(t_1))+ \Cov(d_2\bs W_{J_2,\kappa_2}(t_2),d_2\bs W_{J_2,\kappa_2}(t_2))\\
    &\Cov(\bs V_{d_1,d_2}(t_2),\overline{\bs V}_{d_1,d_2}(t_1)) =  \Cov(d_1\bs W_{J_1,\kappa_1}(t_2),d_1\overline{\bs W}_{J_1,\kappa_1}(t_1))+ \Cov(d_2\bs W_{J_2,\kappa_2}(t_2),d_2\overline{\bs W}_{J_2,\kappa_2}(t_2)).
\end{align*}
From here, we see the result follows by applying Theorem \ref{theorem: cramer wold}.
\\~\\
\textbf{Covariance matrices within the \(\mathcal{W}_{c,|\rho|}\) family:}
Let \(\bs W_{J_1,\kappa_1},...,\bs W_{J_p,\kappa_p}\in \mathcal{W}_{c,|\rho|}\), where the Jordan blocks \(J_1,...,J_p\) have respective sizes \(m_{c1},...,m_{cp}\), and eigenvalues \(\rho_{1},...,\rho_{p}\in \{\rho,\overline{\rho}\}\). Also, let \(d_1,...,d_p \in \mathbb{R}\). By Lemma \ref{Theorem: large time}, we have that
\begin{equation*}
N^{-1/2}\sum_{i=1}^p\e^{-\rho_i\omega t}\omega^{-(m_1+2\kappa_i-2)/2}d_iN_A^{m_{ci}-\kappa_i}P_{J_i}(\bs X_n(\omega t)-N\e^{A\omega t}\bs \mu) \conindis \bs W_c(t)\text{ in }  D[0,\infty),
\end{equation*}
where, for \(0\leq t_1\leq t_2 <\infty\),
\begin{align*}
   &\Cov(\bs W_{c}(t_2),\bs W_{c}(t_1))=\sum_{k,m=0}^{p}d_{k}d_{m}\bs 1_{\{\rho_{ck}=\overline{\rho}_{cm}\}}N_A^{\kappa_{k}-1}P_{J_{k}}BP_{J_{m}}'N_A'^{\kappa_{m}-1}\int_{0}^{t_1}\frac{(t_1-v)^{\kappa_{m}-1}(t_2-v)^{\kappa_{k}-1}v^{m_1-1}}{(\kappa_{k}-1)!(\kappa_{m}-1)!}\mathrm{d}v.
  \\
  &\Cov(\bs W_{c}(t_2),\overline{\bs W}_{c}(t_1))=\sum_{k,m=0}^{p}d_{k}d_{m}\bs 1_{\{\rho_{ck}=\rho_{cm}\}}N_A^{\kappa_{k}-1}P_{J_{k}}BP_{J_{m}}^*N_A'^{\kappa_{m}-1}\int_{0}^{t_1}\frac{(t_1-v)^{\kappa_{m}-1}(t_2-v)^{\kappa_{k}-1}v^{m_1-1}}{(\kappa_{k}-1)!(\kappa_{m}-1)!}\mathrm{d}v.
 \end{align*}
 Recall \(B= \sum_{i=1}^d a_i v(\bs\mu)_i\E[\bs \xi_i\bs \xi_i']\). This covariance function is equal to the covariance function of \(\sum_{i=1}^pd_i\bs W_{J_i,\kappa_i}\). Since they are both Gaussian processes, this implies
 \begin{equation*}
    (\bs W_c(t))_{t\geq 0} \eqindis \F(\sum_{i=1}^pd_i\bs W_{J_i,\kappa_i}(t)\R)_{t\geq 0}.
 \end{equation*}
The result now follows by applying Theorem \ref{theorem: cramer wold}.
 \\~\\
 \textbf{Covariance matrices within the \(\mathcal{V}_{\ell}\) family:} Let \(J_1,...,J_{k}\) be Jordan blocks with eigenvalues in \(\Lambda_{\ell}\). Since the projection matrices \(P_{J_1},...,P_{J_k}\) are linear and continuous, the result immediately follows by applying the continuous mapping theorem to the result of Lemma \ref{lemma: big time} with the maps \(P_{J_1},...,P_{J_k}\).
 \subsection{Proof of Theorems \ref{theorem: main result simplified} IBD \& TR, \ref{theorem: Main results discrete}, and \ref{theorem: main results discrete 2} cont.}
 \label{Appendix: proof of discrete time}
 \subsubsection{Proof of Theorem \ref{theorem: main result simplified} TR and \ref{theorem: main results discrete 2} TR}
\textbf{Convergence of the random time change:} We are going to show
\begin{align}
     &\phi_{n}(t) = \tau_{n,\lfloor nt \rfloor} \coninprob S^{-1}\log(1+St/\beta_1)  \text{ in } D[0,b)
    \label{eq: discrete 6 TC}
\end{align}
as \(n \rightarrow \infty\). Recall that \(b=\infty\) when \(S>0\), and \(b=-\beta_1/S\) if \(S<0\), where this upper cap arises due to the urns extinction after \(-n\beta_1/S\) draws if \(S<0\). The first part of the argument is the same as the IBD regime up to and including \eqref{eq: Convergence 4}. In particular, we use \eqref{eq: discrete 10} in place of \eqref{eq: improve 32} to get that
\begin{align}
\label{eq: dis 10}
   & 1+(\beta_1 N)^{-1}SB_n(t) \coninprob \e^{St}  \text{ in }  D[0,\infty)
\end{align}
as \(n \rightarrow \infty\). Then, for \(0\leq T<b\), by solving \eqref{eq: dis 10} for \(t\) when \(B_n(t) = \lfloor nT \rfloor\), we see, for any fixed \(\varepsilon>0\), the probability of \[A_{n,T,\varepsilon}:=\{\tau_{n,\lfloor nT \rfloor}\leq S^{-1}\log(1+ST/\beta_1)+\varepsilon\}\] tends to 1 as \(n \rightarrow \infty\). This and \eqref{eq: dis 10} imply
\begin{align}
 &\e^{S\tau_{n,\lfloor nt \rfloor}}\coninprob 1+S\beta_1^{-1}  t  \text{ in }  D[0,b) \label{eq: Convergence 8}
\end{align}
as \(n \rightarrow \infty\), where we have used the same argument that was used to show \eqref{eq: Convergence 4}. Lastly, \eqref{eq: discrete 6 TC} follows by applying the continuous mapping theorem with function \(S^{-1}\log\) to \eqref{eq: Convergence 8}.
\\~\\
\textbf{Functional convergence:} We use the same arguments as in the IBD regime with the TR processes in place of the IBD processes. In particular, we take \(\phi_n\) as in \eqref{eq: discrete 6 TC} and let
\begin{equation*}
    \bs Y_n(t) := N^{-1/2}(\bs X_n(t)-N\e^{At}\bs \mu)\conindis \bs W_2(t) \text{ in }D[0,\infty)
\end{equation*}
as \(n\rightarrow \infty\), where the convergence is given by \eqref{eq: discrete 10}. To apply Theorem \ref{theorem: bill 2}, we need to prove \(\bs W_2\) is continuous a.s. We use Theorem \ref{theorem: Kol cont} to do so. Fix \(T>0\). Since \(\bs W_2\) is a Gaussian process, for any \(0\leq t_1 \leq t_2 \leq T\), \(\bs W_2(t_2)-\bs W_2(t_1)\) is Gaussian. Furthermore, by Theorem \ref{Theorem: Main continuous time}, this Gaussian variable has covariance matrix
\begin{align}
    \Var(\bs W_2(t_2)-\bs W_2(t_1))&=\F(1-\e^{A(t_2-t_1)}\R)\sum_{i=1}^d  \int_{0}^{t_1}\mathrm{e}^{A( t_1-v)}\E[\bs \xi_i \bs \xi_i']\mathrm{e}^{A'(t_1-v)}a_i(\mathrm{e}^{Av}\bs \mu)_i\mathrm{d}v\nonumber \\
    & \quad +\mathrm{e}^{A(t_2-t_1)}\sum_{i=1}^d   \int_{0}^{t_2}\mathrm{e}^{A( t_1-v)}\E[\bs \xi_i \bs \xi_i']\mathrm{e}^{A'(t_1-v)}a_i(\mathrm{e}^{Av}\bs \mu)_i\mathrm{d}v \F(\mathrm{e}^{A(t_2-t_1)}-1\R).\label{eq:kol cont1}
\end{align}
For \(0 \leq t\leq T\), we have that
\begin{equation*}
    \F\|\sum_{i=1}^d  \int_{0}^{t}\mathrm{e}^{A( t-v)}\E[\bs \xi_i \bs \xi_i']\mathrm{e}^{A'(t-v)}a_i(\mathrm{e}^{Av}\bs \mu)_i\mathrm{d}v\R\|_2 \leq \cst_T.
\end{equation*}
Also, we have that \(\|\e^{A(t_2-t_1)}-I\|_2\leq \cst_T (t_2-t_1)\). Using these two results in \eqref{eq:kol cont1} implies
\begin{equation*}
   \| \Var(\bs W_2(t_2)-\bs W_2(t_1)) \|_2 \leq \cst_T (t_2-t_1).
\end{equation*}
Therefore, the conditions of Theorem \ref{theorem: Kol cont} hold with \(\alpha = 4\), \(\beta = 1\), and \(K\) some constant only depending on \(T\). Since \(T>0\) was arbitrary, \(\bs W_2\) is a.s.\ continuous in \(D[0,\infty)\). Thus, we can apply Theorem \ref{theorem: bill 2} giving 
\begin{equation}
    \bs Y_n(\phi_n(t)) = N^{-1/2}(\bs U_n(\lfloor nt \rfloor )-N\mathrm{e}^{A\tau_{n,\lfloor nt \rfloor }}\bs \mu)\conindis \bs W_2(\log(1+S t/\beta_1)/S)\text{ in } D[0,b) \label{eq: TR P1}
\end{equation}
as \(n \rightarrow \infty\). We are left to show
\begin{equation}
\label{eq: dis 11}
    N^{1/2}(\mathrm{e}^{A\tau_{n,\lfloor nt \rfloor }}-\mathrm{e}^{AS^{-1}\log(1+Snt/\beta_1N)})\conindis -A\e^{AS^{-1}\log(1+St/\beta_1)}(\beta_1S+S^2 t)^{-1}\sum_{i=1}^d a_i W_2(\log(1+St/\beta_1)/S)_i \text{ in }  D[0,b)
\end{equation}
as \(n \rightarrow \infty\). Indeed, this and \eqref{eq: TR P1} gives \eqref{eq: TR}, and \eqref{eq: colour comp TR} follows from \eqref{eq: TR}. By the same argument used to show \eqref{eq: dis 3}, we get that
\begin{equation} 
\label{eq: final d 10}
  N^{1/2}(1+(\beta_1 N)^{-1}Snt-\e^{S\tau_{n,\lfloor nt \rfloor}}) \conindis \beta_1 ^{-1}\sum_{i=1}^d a_i W_2(\log(1+St/\beta_1)/S)_i   \text{ in }  D[0,b)
\end{equation}
as \(n \rightarrow \infty\). We take the same Taylor expansion as in \eqref{eq: dis 4}, where in the TR case we have, for all \(t\geq 0\), \((\beta_1 N)^{-1}Sn t \rightarrow S \beta_1^{-1} t\) as \(n \rightarrow \infty\). Then, by the same argument used to show \eqref{eq: dis 5}, this, \eqref{eq: dis 4}, and \eqref{eq: final d 10} imply
\begin{equation}
\label{eq: dis 5 TC}
     N^{1/2}(\log(1+Snt/\beta_1N)-S\tau_{n,\lfloor nt \rfloor }) \conindis (\beta_1+S t)^{-1}\sum_{i=1}^d a_i W_2(\log(1+S t/\beta_1)/S)_i \text{ in }  D[0,b)   
\end{equation}
as \(n \rightarrow \infty\). Lastly, to show \eqref{eq: dis 11} from \eqref{eq: dis 5 TC}, we use the same arguments from equation \eqref{eq: dis 5} to the end of the proof of the IBD regime, where we note, for each \(t \in [0,b)\), \(\e^{AS^{-1}\log(1+Snt/\beta_1N)}\rightarrow \e^{AS^{-1}\log(1+St/\beta_1)}\) as \(n\rightarrow \infty\).
\subsubsection{Proof of Theorems \ref{theorem: Main results discrete} TSD\textsubscript{s} \eqref{eq: small component convergence} and \ref{theorem: main results discrete 2} TSD\textsubscript{s}}
\noindent For the TSD regime, we are always going to take \(\omega(n) = S^{-1}\log(n/N) \) in Theorem \ref{Theorem: Main continuous time} LT. Note, this may be negative for small \(n\) for which the equations in \ref{Theorem: Main continuous time} LT are not well defined. To avoid cumbersome details, we always assume \(n\) is large enough such that \( S^{-1}\log(n/N)\) is positive.
 \\~\\
\textbf{Convergence of the random time change:} We are going to show
\begin{equation}
\label{eq: dis 1 TSDs}
    \phi_n(t) = \tau_{n,\lfloor nt \rfloor}-S^{-1}\log(n/N) \coninprob S^{-1}\log( St/\beta_1)  \text{ in }  D(0,\infty)   
\end{equation}
 as \(n \rightarrow \infty\). The first part of the argument is the same as the IBD regime up to and including \eqref{eq: Convergence 4}. In particular, we take the Jordan projection \(\bs v_1 \bs a'\) in \eqref{eq: improve 80} in place of \eqref{eq: improve 32}, where \(\bs v_1\) is the right eigenvector corresponding to \(\bs a\) normalized so that \(\bs a \cdot \bs v_1 =1\) (this is indeed a large Jordan projection by \eqref{eq: eigenvector proj} and Lemma \ref{lemma: balanced frob}). This gives
 \begin{equation*}
   \e^{-S(\omega+t)}(1+SB_n(\omega+t)/\beta_1N) \coninprob 1  \text{ in }  D(-\infty,\infty)
\end{equation*}
as \(n \rightarrow \infty\), with \(\omega=\omega(n)=S^{-1}\log(n/N)\). Then, by the continuous mapping theorem with function \(-S^{-1}\log\), we have
\begin{equation}
\label{eq: TSDs to TSDl}
    \omega+t-S^{-1}\log(1+SB_n(\omega+t)/\beta_1N) \coninprob 0\text{ in }  D(-\infty,\infty)
\end{equation}
as \(n \rightarrow \infty\). Let \(T>0\). By solving \eqref{eq: TSDs to TSDl} for \(t\) when \(B_n(\omega+t)=\lfloor nT \rfloor\), we see the probability of \[A_{n,T}:=\{S^{-1}\log(ST/2\beta_1) \leq \tau_{n,\lfloor nT \rfloor}-\omega\leq S^{-1}\log(2ST/\beta_1)\}\] tends to 1 as \(n \rightarrow \infty\). This and \eqref{eq: TSDs to TSDl} imply
\begin{equation}
\label{eq: Convergence 9}
    \tau_{n,\lfloor nt \rfloor}-S^{-1}\log(1+S\lfloor nt \rfloor/\beta_1 N) \coninprob 0 \text{ in } D(0,\infty)
\end{equation}
as \(n\rightarrow\infty\), where this follows by the same argument used to show \eqref{eq: Convergence 4}. We see \eqref{eq: dis 1 TSDs} immediately follows from \eqref{eq: Convergence 9}. 
\\~\\
\textbf{Functional convergence:} We use the same arguments as used in the IBD regime with the TSD\textsubscript{\textit{s}} processes in place of the IBD processes. In particular, we take \(\phi_n\) as in \eqref{eq: dis 1 TSDs} and let
\begin{equation}
\label{eq: dis TSDs 2}
    \bs Y_n(t) := N^{-1/2}\e^{-S(\omega+t)/2}P_s(\bs X_n(\omega+t)-N\e^{A(\omega+t)}\bs \mu) \conindis \bs W_s(t) \text{ in } D(-\infty, \infty)
\end{equation}
as \(n\rightarrow \infty\), where the convergence is given by \eqref{eq: small component convergence}. To apply Theorem \ref{theorem: Kol cont}, we need to prove \(\bs W_s\) is continuous a.s. To show this, we take the same approach as in the TR regime. Fix \(T>0\). Since \(\bs W_s\) is a Gaussian process, for any~\(-T \leq t_2 \leq t_1 \leq T\),~\(\bs W_s(t_2)-\bs W_s(t_1)\) is Gaussian. Furthermore, by Theorem \ref{Theorem: Main continuous time}, this Gaussian variable has covariance matrix that satisfies 
\begin{equation*}
    \|\Var(\bs W_s(t_2)-\bs W_s(t_1))\|_2 = \F\|2\F(1-\mathrm{e}^{(A-\lambda_1/2) (t_2-t_1)}\R) \sum_{i=1}^d\int_{0}^{\infty}P_s\mathrm{e}^{Av}a_iv(\bs\mu)_{i}\E[\bs \xi_i \bs \xi_i']\e^{A'v}P_s'\mathrm{e}^{-\lambda_1 v}\mathrm{d}v\R\|_2 \leq \cst_T(t_2-t_1).
\end{equation*}
Therefore, the conditions of Theorem \ref{theorem: Kol cont} hold with \(\alpha =4\), \(\beta =1\), and \(K\) some constant only dependent on \(T\). Since \(T>0\) was arbitrary, \(\bs W_s\) is a.s.\ continuous in \(D(-\infty,\infty)\). Therefore, we can apply Theorem \ref{theorem: bill 2} to get 
\begin{align}
\label{eq: dis TSDs 3}
    &\bs Y_n(\phi_n(t)) = N^{-1/2}\e^{-S\tau_{n,\lfloor nt \rfloor}/2}P_s(\bs U_n(\lfloor nt \rfloor )-N\mathrm{e}^{A\tau_{n,\lfloor nt \rfloor }}\bs \mu)\conindis \bs W_s(\log(St/\beta_1)/S)\text{ in }  D(0,\infty)
\end{align}
as \(n \rightarrow \infty\). Furthermore, by \eqref{eq: dis 1 TSDs} and the continuous mapping theorem with function \(\e^{-\frac{S}{2}\cdot}\), we have
\begin{equation}
    \e^{-S\tau_{n,\lfloor nt \rfloor}/2}(nSt/\beta_1N)^{1/2}\coninprob 1  \text{ in } D(0,\infty)
    \label{eq: need for future append}
\end{equation}
as \(n \rightarrow \infty\). Using this result in \eqref{eq: dis TSDs 3} gives
\begin{equation}
\label{eq: dis TSDs 6}
   n^{-1/2}P_s(\bs U_n(\lfloor nt \rfloor )-N\mathrm{e}^{A\tau_{n,\lfloor nt \rfloor }}\bs \mu)\conindis  (St/\beta_1)^{1/2} \bs W_s(\log(St/\beta_1)/S)\text{ in }  D(0,\infty)
\end{equation}
as \(n \rightarrow \infty\). From this, the theorem follows once we have shown
\begin{equation}
\label{eq: dis 11 TDS}
    Nn^{-1/2}P_s(\mathrm{e}^{A\tau_{n,\lfloor nt \rfloor }}-\mathrm{e}^{AS^{-1}\log(1+Snt/\beta_1N)})\coninprob 0  \text{ in }  D(0,\infty)
\end{equation}
as \(n \rightarrow \infty\). By taking the Jordan projection \(\bs v_1 \bs a'\) in \eqref{eq: improve 80}, we see
\begin{equation}
\label{eq: TSDs 15}
    Z_n(t) := N^{-1/2}\F(\e^{-S(\omega+t)}\sum_{i=1}^d a_i X_n(\omega+t)_i-N\beta_1 \R) \conindis V_1  \text{ in } D(-\infty,\infty)
\end{equation}
as \(n \rightarrow \infty\), where we recall \(V_1 = \bs V_{J} \cdot \bs a\), with \(\bs V_{J}\) as in \eqref{eq: improve 80} and \(J\) the Jordan space with projection matrix \(\bs v_1 \bs a'\). We apply Theorem \ref{theorem: bill 2} to the \( Z_n\) (\(V_1\) is clearly a.s.\ continuous, since it is a constant process) with time change \(\phi_n\) giving
\begin{equation}
\label{eq: dis 7 TDS}
      Z_n(\phi_n(t))= N^{1/2}(\e^{-S\tau_{n,\lfloor nt \rfloor}}(1+S\lfloor nt \rfloor/\beta_1N)-1) \conindis \beta_1^{-1} V_1  \text{ in }  D(0,\infty)
\end{equation}
as \(n \rightarrow \infty\), where in the first equality we have used the balanced condition. Next, we take the Taylor expansion around 1 of \(\log(\e^{-S\tau_{n,\lfloor nt \rfloor}}(1+S\lfloor nt \rfloor/\beta_1N))\). By \eqref{eq: dis 7 TDS}, the second-order terms and higher are \(\mathcal{O}_p(N^{-1})\). Therefore, we have 
\begin{equation}
\label{eq: final d 12}
 N^{1/2}(\log(1+Snt/\beta_1N )-S\tau_{n,\lfloor nt \rfloor}) \conindis  \beta_1^{-1} V_1 \text{ in } D(0,\infty)
\end{equation}
as \(n \rightarrow \infty\), where the floor operator can be dropped, since the difference from the true value is of order \(n^{-1}N^{1/2}\rightarrow 0\) as~\(n \rightarrow \infty\). We multiply the terms in the brackets in \eqref{eq: final d 12} by \(-AS^{-1}\), then take their exponential. By the same argument that gave \eqref{eq: dis 6}, we get that
\begin{equation}
\label{eq: dis 33}
     N^{1/2}(\mathrm{exp}(A\tau_{n,\lfloor nt \rfloor}-AS^{-1}\log(1+Snt/\beta_1N ))-1) \conindis -A(\beta_1 S)^{-1}V_1 \text{ in } D(0,\infty)
\end{equation}
as \(n \rightarrow \infty\). We see the l.h.s.\ of \eqref{eq: dis 11 TDS} is the l.h.s.\ of \eqref{eq: dis 33} multiplied by \(N^{1/2}n^{-1/2}P_s\e^{AS^{-1}\log(1+Snt/\beta_1N )}\). Furthermore, since the projection matrix \(P_s\) only consists of small components, by \eqref{eq: matrix exp bounds}, for any \(t \in (0,\infty)\), we have
\begin{equation*}
    \|N^{1/2}n^{-1/2}P_s\e^{AS^{-1}\log(1+Snt/\beta_1N )}\|_2 \leq \cst n^{-\varepsilon}N^{\varepsilon}t^{1/2-\varepsilon} \rightarrow 0
\end{equation*}
as \(n \rightarrow \infty\), for some \(\varepsilon>0\) independent of \(t\). These two facts imply \eqref{eq: dis 11 TDS} through \eqref{eq: dis 33}, and hence the theorem holds.
\subsubsection{Proof of Theorem \ref{theorem: Main results discrete} TSD\textsubscript{s} \eqref{eq: small component convergence in probability}, TSD\textsubscript{c} and TSD\textsubscript{\(\ell\)}, and Theorem \ref{theorem: main results discrete 2} TSD\textsubscript{c} \eqref{eq: limit for TSDc} and TSD\textsubscript{\(\ell\)} \eqref{eq: large fluctuations depend}}
We group these cases together, since they use the same random time change limit.
\\~\\
\textbf{Convergence of the random time change:} We are going to show
\begin{equation}
\label{eq: time change TSDc}
   \phi_{n}(t) = S\log(n/N)^{-1}\tau_{n,\lfloor N(n/N)^t \rfloor} \coninprob t \text{ in } D(0,\infty) 
\end{equation}
as \(n \rightarrow \infty\). The first part of this proof is almost identical to the TSD\textsubscript{s} case up to \eqref{eq: TSDs to TSDl}, where instead we use that \eqref{eq: improve 80} holds over the larger timescale \((\omega t)_{t\in (0,\infty)}\). This gives us that
\begin{equation}
    \omega t - S^{-1}\log(1+SB_n(\omega t)/\beta_1N) \coninprob 0  \text{ in } D(0,\infty)
    \label{eq: Convergence 10}
\end{equation}
as \(n \rightarrow \infty\). Let \(T>0\). By solving \eqref{eq: Convergence 10} for \(t\) when \(B_n(\omega t)=\lfloor N(n/N)^T \rfloor\), we see the probability of \[A_{n,T}:=\{T/2S \leq \omega^{-1}\tau_{n,\lfloor N(n/N)^T\rfloor}\leq 2T/S\}\] tends to 1 as \(n \rightarrow \infty\). This and \eqref{eq: Convergence 10} imply 
\begin{equation}
\label{eq: Convergence 11}
 \tau_{n,\lfloor N(n/N)^t \rfloor}-S^{-1}\log(1+S\lfloor N(n/N)^t \rfloor/\beta_1N) \coninprob  0\text{ in }  D(0,\infty)   
\end{equation}
as \(n \rightarrow \infty\). This implies \eqref{eq: time change TSDc}, since, for all \(t \geq 0\),
\begin{equation*}
    \log(n/N)^{-1}\log(1+S\lfloor N(n/N)^t\rfloor/\beta_1N)=\frac{\log(1+(\beta_1 N)^{-1}S\lfloor N(n/N)^t\rfloor)}{\log((n/N)^t)}t \rightarrow t
\end{equation*}
as \(n \rightarrow \infty\).
\\~\\
\textbf{Functional convergence for Theorem \ref{theorem: Main results discrete} TSD\textsubscript{\textit{c}}:} We use the same arguments as used in the IBD regime with the TSD\textsubscript{\textit{c}} processes in place of the IBD processes. In particular, we take \(\phi_n\) as in \eqref{eq: time change TSDc} and let
\begin{equation}
\label{eq: dis TSDc 2}
    \bs Y_n(t) := N^{-1/2}\omega^{-(\kappa-1/2)}\e^{-\lambda \omega t}N_A^{m-\kappa}P_J(\bs X_n(\omega t)-N\e^{A\omega t}\bs \mu) \conindis \bs W_{J,\kappa}(t) \text{ in }D[0,\infty)
\end{equation}
as \(n\rightarrow \infty\), where the convergence is given by \eqref{eq: improve 40}. To apply Theorem \ref{theorem: Kol cont}, we need to prove \(\bs W_{J,\kappa}\) is a.s.\ continuous. To do so, we take the same approach as in the TR regime. Fix \(T>0\). Since \(\bs W_{J,\kappa}\) is a Gaussian process, for any \(0\leq t_1 \leq t_2\leq T\), the random variable \(\bs W_{J,\kappa}(t_2)-\bs W_{J,\kappa}(t_1)\) is Gaussian. Furthermore, by Theorem \ref{Theorem: Main continuous time}, we have 
\begin{align*}
    \Cov(\bs W_{J,\kappa}(t_2)-&\bs W_{J,\kappa}(t_1), \overline{\bs W}_{J,\kappa}(t_2)-\overline{\bs W}_{J,\kappa}(t_1)) \\
    &= C \F(\int_{0}^{t_1}(t_2-v)^{2(\kappa-1)}+(t_1-v)^{2(\kappa-1)}-2(t_2-v)^{\kappa-1}(t_1-v)^{\kappa-1}\mathrm{d}v+\int_{t_1}^{t_2}(t_2-v)^{2(\kappa-1)} \mathrm{d}v\R),
\end{align*}
where 
\begin{equation*}
    C=\frac{1}{(\kappa-1)!^2}\sum_{i=1}^dN_A^{m-1}P_{J}a_iv(\bs\mu)_{i}\E[\bs \xi_i \bs \xi_i']P_{J}^*N_A'^{m-1}.
\end{equation*}
For the second integral, since \(0\leq t_2 \leq T\), we have that
\begin{equation*}
 \int_{t_1}^{t_2}(t_2-v)^{2(\kappa-1)} \mathrm{d}v \leq \cst_T(t_2-t_1).
\end{equation*}
For the first integral, we have that
\begin{align*}
 \int_{0}^{t_1}(t_2-v)^{2(\kappa-1)}+(t_1-v)^{2(\kappa-1)}-2(t_2-v)^{\kappa-1}(t_1-v)^{\kappa-1}\mathrm{d}v&=\int_{0}^{t_1}\F((t_2-v)^{\kappa-1}-(t_1-v)^{\kappa-1}\R)^2\mathrm{d}v\\
 &\leq \cst_T (t_2^{\kappa-1}-t_1^{\kappa-1})^2\\
 &\leq \cst_T (t_2-t_1)^2,
\end{align*}
where, in the penultimate inequality, we use that the derivative of \(x^{\kappa-1}\) is non-decreasing for \(x\geq 0\), so the integrand is maximised at \(v=0\). Therefore, we have that
\begin{equation*}
     \|\Var(\bs W_{J,\kappa}(t_2)-\bs W_{J,\kappa}(t_1))\|_2\leq \cst_T (t_2-t_1).
\end{equation*}
This implies the conditions of Theorem \ref{theorem: Kol cont} are satisfied with \(\alpha = 4\), \(\beta =1\), and \(K\) some constant that depends on \(T\). Since \(T>0\) was arbitrary, \(\bs W_{J,\kappa}\) is a.s.\ continuous in \(D[0,\infty)\). Therefore, we can apply Theorem \ref{theorem: bill 2} which gives 
\begin{align}
\label{eq: dis TSDc 3}
    &\bs Y_n(\phi_n(t)) = N^{-1/2}\omega^{-(\kappa-1/2)}\e^{-\lambda \tau_{n,\lfloor N(n/N)^t \rfloor  }}N_A^{m-\kappa}P_J(\bs U_n(\lfloor N(n/N)^t \rfloor )-N\mathrm{e}^{A\tau_{n,\lfloor N(n/N)^t \rfloor  }}\bs \mu)\conindis \bs W_{J,\kappa}(t)
\end{align}
in \(D(0,\infty)\) as \(n \rightarrow \infty\). Next, by \eqref{eq: Convergence 11} and the continuous mapping theorem with function \(\e^{-\lambda \cdot}\), we have
\begin{equation}
\label{eq: dis TSDc 4}
    \e^{-\lambda \tau_{n,\lfloor N(n/N)^t \rfloor  }}(n/N)^{\lambda t/S} \coninprob (\beta_1 /S)^{\lambda/S} \text{ in } D(0,\infty)
\end{equation}
as \(n \rightarrow \infty\), where we have used, for all \(t \geq 0\),
\begin{equation*}
    (1+S\lfloor N(n/N)^t\rfloor/\beta_1N)^{\lambda/S}= (S(n/N)^t/\beta_1)^{\lambda/S} + o(1).
    \end{equation*}
(This can be seen by taking a Taylor expansion of the l.h.s.\ around \(S (n/N)^t/\beta_1\).) Using \eqref{eq: dis TSDc 4} in \eqref{eq: dis TSDc 3} gives
\begin{equation}
\label{eq: dis TSDc 5}
    N^{-1/2}(n/N)^{-\lambda t /S}\omega^{-(\kappa-1/2)}N_A^{m-\kappa}P_J(\bs U_n(\lfloor N(n/N)^t \rfloor )-N\mathrm{e}^{A\tau_{n,\lfloor N(n/N)^t \rfloor  }}\bs \mu)\conindis (S/\beta_1)^{\lambda/S}\bs W_{J,\kappa}(t) \text{ in }  D(0,\infty)
\end{equation}
as \(n \rightarrow \infty\). Since \(\mathrm{Re}\lambda/S=1/2\), by \eqref{eq: dis TSDc 5}, the theorem will follow once we have shown
\begin{equation}
\label{eq: TSDc 15}
       N^{1/2}(n/N)^{-t/2}\omega^{-(\kappa-1/2)}N_A^{m-\kappa}P_J( \mathrm{e}^{A\tau_{n,\lfloor N(n/N)^t \rfloor  }}-\mathrm{e}^{AS^{-1}\log(1+S(n/N)^t)/\beta_1}) \coninprob 0 \text{ in }  D(0,\infty)
\end{equation}
as \(n \rightarrow \infty\). We repeat steps \eqref{eq: TSDs 15}-\eqref{eq: dis 33}, where we take \(\phi_n\) as in \eqref{eq: time change TSDc} and we let 
\begin{equation*}
    Z_n(t) := N^{-1/2}\F(\e^{-S\omega t}\sum_{i=1}^d a_i X_n(\omega t)_i-N\beta_1 \R) \conindis V_1 \text{ in } D(0,\infty)
\end{equation*}
as \(n \rightarrow \infty\), where the convergence is due to \eqref{eq: improve 80}. Repeating these steps for the new \(\phi_n\) and \(Z_n\) gives
\begin{equation}
\label{eq: TSDc 16}
    N^{1/2}(\mathrm{exp}(A\tau_{n,\lfloor N(n/N)^t \rfloor}-AS^{-1}\log(1+S(n/N)^t/\beta_1))-1) \conindis -A(\beta_1 S)^{-1}V_1  \text{ in }  D(0,\infty)
\end{equation}
as \(n \rightarrow \infty\). We see the l.h.s.\ of \eqref{eq: TSDc 15} is the l.h.s.\ of \eqref{eq: TSDc 16} multiplied on the left by \\\((n/N)^{-t/2}\omega^{-(\kappa-1/2)}N_A^{m-\kappa}P_J\e^{AS^{-1}\log(1+S(n/N)^t/\beta_1)}\). Therefore, \eqref{eq: TSDc 15} follows from \eqref{eq: TSDc 16}, since, for any \(t \geq  0\),
\begin{equation}
\label{eq: Bound 10}
    \|(n/N)^{-t/2}\omega^{-(\kappa-1/2)}N_A^{m-\kappa}P_J\e^{AS^{-1}\log(1+S(n/N)^t/\beta_1)}\|_2 \leq \cst \omega^{-1/2}(1+t)^{\kappa-1}\rightarrow 0
\end{equation}
as \(n \rightarrow \infty\). The inequality in \eqref{eq: Bound 10} is a consequence of \eqref{eq: matrix exp bounds}.
\\~\\
\textbf{Functional convergence for Theorem \ref{theorem: Main results discrete} TSD\textsubscript{\(\ell\)}:} We use the same arguments as used in the IBD regime with the TSD\textsubscript{\(\ell\)} processes in place of the IBD processes. In particular, we take \(\phi_n\) as in \eqref{eq: time change TSDc} and let
\begin{equation}
\label{eq: dis TSDl 2}
    \bs Y_n(t) := N^{-1/2}P_J(\e^{-A\omega t}\bs X_n(\omega t)-N\bs \mu) \conindis \bs V_J \text{ in } D(0,\infty)
\end{equation}
as \(n\rightarrow \infty\), where this convergence is given by \eqref{eq: improve 80}. We apply Theorem \ref{theorem: bill 2} (\(\bs V_J\) is clearly a.s.\ continuous since it is constant) which gives 
\begin{align}
\label{eq: dis TSDl 3}
    &\bs Y_n(\phi_n(t)) = N^{-1/2}P_J(\e^{-A\tau_{n,\lfloor N(n/N)^t \rfloor}}\bs U_n({\lfloor N(n/N)^t \rfloor})-N\bs \mu)\conindis \bs V_{J} \text{ in }  D(0,\infty)
\end{align}
as \(n \rightarrow \infty\). Next, by \eqref{eq: TSDc 16}, we have that
\begin{equation}
\label{eq: dis TSDl 4}
    \mathrm{exp}(AS^{-1}\log(1+S(n/N)^t/\beta_1)-A\tau_{n,\lfloor N(n/N)^t \rfloor}) \coninprob 1  \text{ in }  D(0,\infty)
\end{equation}
as \(n \rightarrow \infty\).
 Using this in \eqref{eq: dis TSDl 3} gives
\begin{equation}
\label{eq: dis TSDl 5}
   N^{-1/2}\e^{-AS^{-1}\log(1+S(n/N)^t/\beta_1)}P_J(\bs U_n({\lfloor N(n/N)^t \rfloor})-\e^{A\tau_{n,\lfloor N(n/N)^t \rfloor}}N\bs \mu)\conindis \bs V_{J} \text{ in }  D(0,\infty)
\end{equation}
as \(n \rightarrow \infty\). Next, by applying the continuous mapping theorem to \eqref{eq: TSDc 16} with function \(P_J\), we have that
\begin{equation*}
    N^{1/2}\e^{-AS^{-1}\log(1+S(n/N)^t/\beta_1)}P_J(\e^{A\tau_{n,\lfloor nt \rfloor}}-\e^{AS^{-1}\log(1+S(n/N)^t/\beta_1)}) \conindis -AP_J(\beta_1 S)^{-1}V_1 \text{ in }  D(0,\infty) 
\end{equation*}
as \(n \rightarrow \infty\). This and \eqref{eq: dis TSDl 5} imply
\begin{equation}
     N^{-1/2}P_J(\e^{-AS^{-1}\log(1+S(n/N)^t/\beta_1)}\bs U_n({\lfloor N(n/N)^t \rfloor})-N\bs \mu)\conindis \bs V_{J}-AP_J(\beta_1 S)^{-1}V_1\bs\mu \text{ in } D(0,\infty)
     \label{eq: Convergence 12}
\end{equation}
as \(n\rightarrow \infty\). Then, by \eqref{eq: mat exp bound start},
\begin{align*}
  &N_A^{m-\kappa}P_J(\bs U_n(\lfloor N(n/N)^t\rfloor-N\e^{AS^{-1}\log(1+S(n/N)^t/\beta_1)}\bs \mu )
  \\&= (1+S(n/N)^t/\beta_1)^{\lambda/S}\sum_{i=0}^{\kappa-1}\frac{(S^{-1}\log(1+S(n/N)^t/\beta_1))^{i}}{i!}N_A^{m-\kappa+i}P_J(\e^{-AS^{-1}\log(1+S(n/N)^t/\beta_1)}\bs U_n({\lfloor N(n/N)^t \rfloor})-N\bs \mu)
\end{align*}
The highest order term in this sum is when \(i=\kappa-1\). Using this and \eqref{eq: Convergence 12}, we see
\begin{align*}
   &N^{-1/2}(n/N)^{-\lambda t/S}\log((n/N)^t)^{-(\kappa-1)} N_A^{m-\kappa}P_J(\bs U_n(\lfloor N(n/N)^t\rfloor-N\e^{AS^{-1}\log(1+S(n/N)^t/\beta_1)}\bs \mu )\\
   &\qquad\qquad\qquad\qquad\qquad\qquad\qquad\qquad\qquad\qquad\qquad\qquad \conindis \frac{(S/\beta_1)^{\lambda/S}N_A^{m-1}(\bs V_{J}-AP_J(\beta_1 S)^{-1}V_1\bs\mu)}{S^{\kappa-1}(\kappa-1)!} \text{ in } D(0,\infty)
\end{align*}
as \(n\rightarrow \infty\), as required.
\\~\\
\textbf{Functional convergence for Theorem \ref{theorem: Main results discrete} TSD\textsubscript{s} \eqref{eq: small component convergence in probability}:} We use the same arguments as used in the IBD regime with the TSD\textsubscript{s} processes in place of the IBD processes. In particular, we take \(\phi_n\) as in \eqref{eq: time change TSDc} and let
\begin{equation}
\label{eq: dis TSDss}
    \bs Y_n(t) :=  K^{-1} N^{-1/2}\omega^{-1/4}\mathrm{e}^{-S\omega t/2}P_s( \bs X_n(\omega t)-N\mathrm{e}^{A \omega t}\bs \mu)\coninprob 0 \text{ in } D[0,\infty)
\end{equation}
as \(n\rightarrow \infty\), where the convergence is given by \eqref{eq: small component convergence in probability}. We apply Theorem \ref{theorem: bill 2} which gives
\begin{equation}
   \label{eq: dis TSDss 2}
   \bs Y_n(\phi_n(t)) = K^{-1} N^{-1/2}\omega^{-1/4}\e^{-S\tau_{n,\lfloor N(n/N)^t \rfloor}/2}P_s(\bs U_n({\lfloor N(n/N)^t \rfloor})-N\mathrm{e}^{A\tau_{n,\lfloor N(n/N)^t \rfloor  }}\bs \mu)\coninprob 0 \text{ in }D(0,\infty)
\end{equation}
as \(n \rightarrow \infty\). This and \eqref{eq: dis TSDc 4} (with function \(\e^{-\frac{S}{2}\cdot}\)) imply
\begin{equation}
    K^{-1} N^{-1/2}(n/N)^{-t/2}\omega^{-1/4}P_s(\bs U_n({\lfloor N(n/N)^t \rfloor})-N\mathrm{e}^{A\tau_{n,\lfloor N(n/N)^t \rfloor  }}\bs \mu)\coninprob 0 \text{ in }D(0,\infty)
\end{equation}
as \(n\rightarrow \infty\). Thus, the result will follow once we show
\begin{equation}
   K^{-1} N^{1/2}(n/N)^{-t/2}\omega^{-1/4}P_s( \mathrm{e}^{A\tau_{n,\lfloor N(n/N)^t \rfloor  }}-\mathrm{e}^{AS^{-1}\log(1+S(n/N)^t/\beta_1)})\coninprob 0\text{ in }D(0,\infty)
\end{equation}
as \(n\rightarrow \infty\). This follows by \eqref{eq: TSDc 16}, and that, by \eqref{eq: matrix exp bounds}, for any \(t \geq 0\),
\begin{equation*}
    \| K^{-1} (n/N)^{-t/2}\omega^{-1/4}P_s \mathrm{e}^{AS^{-1}\log(1+S(n/N)^t/\beta_1)}\|_2 \rightarrow 0
\end{equation*}
as \(n\rightarrow \infty\).
\subsection{Proof of Lemma \ref{lemma: import ob}}
\label{Appendix: import obj}
Assume for contradiction that there exists some Jordan block with eigenvalue \(S\) of size larger than 1. In this case, \eqref{eq: dis 1 TSDs} still holds with an identical proof. Furthermore, by \eqref{eq: first order MCBP} with \(\omega = S^{-1}\log(n/N)\), we have that
\begin{equation*}
   \bs Y_n(t):=N^{-1}\log(n/N)^{-1}\e^{-(\log(n/N)+St)}(\bs X_n(S^{-1}\log(n/N)+t)-N\e^{A(S^{-1}\log(n/N)+t)}\bs\mu)\coninprob 0 \text{ in } D(0,\infty)
\end{equation*}
as \(n\rightarrow \infty\). Applying Theorem \ref{theorem: bill 2} with \(\phi_n\) as in \eqref{eq: dis 1 TSDs} and \(\bs Y_n\) gives 
\begin{equation}
\bs Y_n(\phi_n(t))= N^{-1}\log(n/N)^{m_1-1}\e^{-S\tau_{n,\lfloor nt \rfloor }}(\bs U_n(\lfloor nt \rfloor)-N\e^{A\tau_{n,\lfloor nt \rfloor }}\bs\mu)\coninprob 0 \text{ in } D(0,\infty) 
\label{eq: contradiction setup}
\end{equation}
as \(n\rightarrow \infty\). Next, notice that \eqref{eq: need for future append} and \eqref{eq: dis 33} still hold with identical proof. They imply, as \(n\rightarrow \infty\),
\begin{align*}
       & \e^{-S\tau_{n,\lfloor nt \rfloor}}(nSt/\beta_1N)\coninprob 1  \text{ in } D(0,\infty),\\
        &\mathrm{exp}(A\tau_{n,\lfloor nt \rfloor}-AS^{-1}\log(1+Snt/\beta_1N )) \coninprob 1 \text{ in } D(0,\infty).
\end{align*}
Using these in \eqref{eq: contradiction setup} gives
\begin{equation}
    n^{-1}\log(n/N)^{-1}(\bs U_n(\lfloor nt \rfloor)-N\e^{AS^{-1}\log(1+Snt/\beta_1N )}\bs\mu)\coninprob 0 \text{ in } D(0,\infty). \label{eq: cont complete}
\end{equation}
Now, since we have a Jordan block with eigenvalue \(S\) of size larger than 1, we have that \(N\e^{AS^{-1}\log(1+Sn/\beta_1N )}\) is of order at least \(n\log(n/N)\) as \(n\rightarrow \infty\) by \eqref{eq: mat exp bound start}. This gives a contradiction, since this and \eqref{eq: cont complete} imply the urn at draw \(n\) is of order at least \(n\log(n/N)>>n\) as \(n\rightarrow \infty\), which is impossible since we only add mass \(S\) at each draw.
\end{document}